\documentclass[12pt,reqno]{amsart}
\usepackage{fullpage}

\newtheorem{theorem}{Theorem}[section]
\newtheorem{lemma}[theorem]{Lemma}

\newtheorem{cor}[theorem]{Corollary}
\newtheorem{conj}{Conjecture}
\usepackage{graphicx}
\usepackage{color}
\usepackage{subfigure}
\usepackage{amssymb}
\usepackage{amsmath}
\usepackage{colonequals}
\usepackage{hyperref}
\theoremstyle{definition}
\newtheorem{definition}[theorem]{Definition}

\newtheorem{assumption}{Assumption}
\newtheorem{remark}[theorem]{Remark}
\newtheorem{question}{Question}

\renewcommand{\subset}{\subseteq}

\renewcommand{\epsilon}{\varepsilon}
\renewcommand{\nu}{v}

\newcommand{\abs}[1]{\left|#1\right|}                   % Absolute value notation
\newcommand{\absf}[1]{|#1|}                             % small absolute value signs
\newcommand{\vnorm}[1]{\left|\left|#1\right|\right|}    % norm notation
\newcommand{\vnormf}[1]{||#1||}                         % norm notation, forced to be small
\newcommand{\N}{\mathbb{N}}

\newcommand{\R}{\mathbb{R}}

\newcommand{\embolden}[1]{\textbf {#1}}

\begin{document}

\title{Low Correlation Noise Stability of Symmetric Sets}

% NSF Graduate Research Fellowship DGE-0813964
% Part of this work was carried out when S.~H. and A.~N. were visiting the Quantitative Geometry program at MSRI.
%
%\thanks{S. H. is supported by NSF Graduate Research Fellowship DGE-0813964.  Part of this work was carried out while visiting the Quantitative Geometry program at MSRI}
%\author{Steven Heilman}
%\address{Courant Institute, New York University, New York NY 10012}
%\email{heilman@cims.nyu.edu}
%\author{Aukosh Jagannath}
%%\address{Courant Institute, New York University, New York NY 10012}
%\email{asj260@nyu.edu}
%\author{Assaf Naor}
%%\address{Courant Institute, New York University, New York NY 10012}
%\email{naor@cims.nyu.edu}
%\date{..... 2011}
%
\author{Steven Heilman}
%\address{Department of Mathematics, UCLA, Los Angeles, CA 90095-1555\\ E-mail address: {\upshape \texttt{heilman@math.ucla.edu}}}
\date{\today}
\thanks{Department of Mathematics, UCLA, Los Angeles, CA 90095-1555. \\ \textit{Email.} heilman@math.ucla.edu\\ \textit{Acknowledgement.} Thanks to Oded Regev for helpful discussions.  Thanks to Elchanan Mossel for encouraging me to publish these results.}

\begin{abstract}
We study the Gaussian noise stability of subsets $A$ of Euclidean space satisfying $A=-A$.  It is shown that an interval centered at the origin, or its complement, maximizes noise stability for small correlation, among symmetric subsets of the real line of fixed Gaussian measure.  On the other hand, in dimension two and higher, the ball or its complement does not always maximize noise stability among symmetric sets of fixed Gaussian measure.  In summary, we provide the first known positive and negative results for the Symmetric Gaussian Problem.
\end{abstract}
\maketitle
%%outline:
% basic stuff, formulas for coefficients, P_t f, etc.
% E-L equation (make note that we can make volume preserving perturbations)
% flatness argument that relies on symmetry (gives proof for dimension 2,3, gives proof contingent on conjecture of what's his name)
%   also need inductive argument?
%   also need perturbative version of conicality of propeller problem
% direct symmetry argument for dimension 2,3
% extra stuff: direct calculation of (d/dt)P_t f, using results of caffarelli et al, colding and minicozzi too
%    try to relate this to extending the solution of SSC

\section{Introduction}
%\snote{Edit lecture notes on website, and research statement, to reflect these changes; edit abstract}

%\snote{Mention somewhere: when integrating on surfaces, $dx$ denotes the restriction of Lebesgue measure to the surface.}
%rewrite using $\rho$ instead of $\theta$

Gaussian noise stability is a well-studied topic with connections to geometry of minimal surfaces \cite{colding11}, hypercontractivity and invariance principles \cite{mossel10}, isoperimetric inequalities \cite{pisier86,ledoux94,khot07,mossel12,kane13,kane14}, sharp Unique Games hardness results in theoretical computer science \cite{khot07,mossel10,khot11,khot12}, social choice theory, learning theory \cite{klivans01,klivans02,klivans08}, and communication complexity \cite{chakrabarti10,sherstov12,vidick13a}.

In applications, it is often desirable to maximize noise stability.  A sample result is the following well-known Theorem of Borell, which has recently been re-proven and strengthened in various ways:

\begin{theorem}[{\cite{borell85,ledoux94,mossel12,eldan13}}]\label{thm0}
Among all subsets of Euclidean space $\R^{n}$ of fixed Gaussian measure, a half space maximizes noise stability (for positive correlation).
\end{theorem}

Here a half space is any set of points lying on one side of a hyperplane.

A well-known Corollary of Theorem \ref{thm0} says: among all subsets of Euclidean space $\R^{n}$ of fixed Gaussian measure, a half space has minimal Gaussian surface area.  This statement may be surprising if one has only seen the isoperimetric inequality for Lebesgue measure.  The latter inequality says: among all subsets of Euclidean space $\R^{n}$ of fixed Lebesgue measure, a ball has minimal surface area.

The present paper concerns a variant of Theorem \ref{thm0} where we only consider symmetric sets.  We say a subset $A$ of $\R^{n}$ is symmetric if $A=-A$.  Such a variant of Theorem \ref{thm0} is a conjecture.

\begin{conj}[Informal, \cite{barthe01,chakrabarti10,odonnell12}]\label{conj0}
Among all symmetric subsets of $\R^{n}$ of fixed Gaussian measure, the ball centered at the origin or its complement maximizes noise stability (for positive correlation).
\end{conj}

If Conjecture \ref{conj0} were true, then a Corollary would be: among all symmetric subsets of Euclidean space $\R^{n}$ of fixed Gaussian measure, a ball centered at the origin or its complement has minimal Gaussian surface area.  So, by restricting our attention to symmetric sets, the isoperimetric sets for the Gaussian measure and Lebesgue measure become essentially the same.

We show that Conjecture \ref{conj0} holds in the case $n=1$, and it does not hold in general in the case $n\geq2$.  That is, we provide the first known positive and negative results for Conjecture \ref{conj0}.

\subsection{Previous Work}

It is natural to expect that that the approaches for proving Theorem \ref{thm0} taken e.g. in \cite{borell85,ledoux94,mossel12,eldan13} would apply to Conjecture \ref{conj0}.  However, this does not seem to be the case.  The approaches of \cite{borell85,mossel12,eldan13}  in proving Theorem \ref{thm0} all seem to use the following property of a half-space: when a half space is translated, it still maximizes noise stability (with a different measure constraint).  This translation-invariance property goes away when we consider Conjecture \ref{conj0}.

The restriction that the subset $A$ is symmetric in Conjecture \ref{conj0}, i.e. that $A=-A$, immediately removes any translation invariance property of the maximizers for this problem.  That is, if a set $A$ satisfies $A=-A$ and the set $A$ maximizes noise stability among all symmetric subsets of $\R^{n}$ with Gaussian measure $1/2$, then an arbitrary translation of $A$ will no longer be a symmetric set.  So, this translated set cannot maximize noise stability among symmetric sets.  In short, we need to use some approach different from \cite{borell85,mossel12,eldan13} in our investigation of Conjecture \ref{conj0}.  The approach of \cite{heilman12} was designed to avoid this translation-invariance issue, and we can similarly apply the approach of \cite{heilman12} to Conjecture \ref{conj0}.  When applying this approach to Conjecture \ref{conj0}, we first maximize noise stability for the correlation $\rho=0$.  Then, when the correlation $\rho$ is small, one shows that the first variation condition for noise stability essentially defines a contractive mapping in a neighborhood of a global maximum.  On the other hand, the approach of \cite{heilman12} currently seems special to the low correlation regime, whereas the approaches of \cite{borell85,mossel12,eldan13} work for Theorem \ref{thm0} for any correlation value $\rho\in(-1,1)$.

Also, as used in various other works on isoperimetry with respect to the Gaussian measure (see e.g. \cite{corneli08}), one may try to solve Conjecture \ref{conj0} by solving the analogous problem on the unit $n$-dimensional unit sphere $S^{n}$ equipped with its normalized Haar measure.  Solving this analogous problem on $S^{n}$ and letting $n\to\infty$ could potentially solve Conjecture \ref{conj0} itself.  In fact, \cite{barthe01} mentions this strategy for considering Conjecture \ref{conj0}.  However, as communicated to us by K. Oleszkiewicz (and noted in \cite{barthe01}), this strategy seems infeasible for proving Conjecture \ref{conj0}.  There is a symmetric torus in $S^{3}$ of Haar measure $1/2$ which has less surface area than two spherical caps of total measure $1/2$.  Therefore, two spherical caps of total measure $1/2$ cannot maximize noise stability on the sphere $S^{3}$ for all correlation values $\rho\in(-1,1)$.  (If we normalize correctly, the derivative of noise stability at $\rho=1$ is equal to surface area.  That is, maximizing noise stability for $\rho\to1$ corresponds to minimizing surface area.)  It is still possible that spherical caps maximize noise stability on the sphere $S^{n}$ as $n\to\infty$, but this example for $S^{3}$ suggests the situation could be complicated for any fixed $n$.

In contrast to Conjecture \ref{conj0}, \cite[Theorem 1]{barthe01} shows that, if we choose a modified definition of Gaussian surface area, then symmetric strips have minimal modified Gaussian surface area among all sets of fixed Gaussian measure.  Here a symmetric strip is a symmetric set bounded by two parallel hyperplanes.  Since \cite[Theorem 1]{barthe01} uses a modified definition of Gaussian surface area, \cite[Theorem 1]{barthe01} does not contradict Conjecture \ref{conj0}.

Lastly, it is tempting to try to prove Conjecture \ref{conj0} by a symmetrization argument, as in \cite{burchard01,isaksson11}, but it is unclear how to construct such an argument in this setting.  A symmetrization argument would begin with an arbitrary set $A$, and then construct a new set which is ``more symmetric,'' and whose noise stability would be larger than that of $A$.  In other words, a symmetrization argument begins with a set $A$ and then moves $A$ in a direction of ``increasing gradient'' of noise stability.  Symmetrization arguments are best suited for statements such as Theorem \ref{thm0}, where the set of maximum noise stability is unique, up to rotations.  However, since we expect Conjecture \ref{conj0} to have at least two local maxima (namely the ball centered at the origin and its complement), a symmetrization argument seems more difficult to implement.

Noise stability can be interpreted as a nonlocal interaction energy \cite{villani03}.  Note that a theory of nonlocal minimal surfaces has been developed \cite{caf10}, but it does not seem to apply in the present setting.

For the reasons including those mentioned in \cite{barthe01}, Conjecture \ref{conj0} appears to be a difficult problem to solve in general.  Furthermore, Conjecture \ref{conj0} essentially contains the problem of minimizing ``entropy'' among self-shrinking solutions to the mean curvature flow.  This problem has recently found significant progress \cite{colding12}, building on a sequence of works including \cite{colding11}, but this minimization problem is still not fully resolved.  The main result of \cite{colding12} only considers minimizing ``entropy'' among compact sets, so e.g. cones are ignored in their result.

%With the intricacies of Conjecture \ref{conj0} in mind, we therefore try to prove Conjecture \ref{conj0} for certain specific parameters, namely dimensions $n=1,2$ and for low %correlation $\rho$.  We also investigate some asymptotic results for $\rho\to0$ as $n\to\infty$.

\subsection{Basic Definitions}

\begin{definition}[Gaussian Measure]
Let $n$ be a positive integer.  Let $A\subset\R^{n}$ be a measurable set.  Define the \textit{Gaussian measure of $A$} to be
$$\gamma_{n}(A)\colonequals\int_{A}e^{-(x_{1}^{2}+\cdots+x_{n}^{2})/2}\frac{dx}{(2\pi)^{n/2}}.$$
\end{definition}

Let $\N\colonequals\{0,1,2,\ldots\}$.  For any $x=(x_{1},\ldots,x_{n}),y=(y_{1},\ldots,y_{n})\in\R^{n}$, define $\langle x,y\rangle\colonequals\sum_{i=1}^{n}x_{i}y_{i}$, $\vnorm{x}_{2}\colonequals\sqrt{\langle x,x\rangle}$.  For any $f\colon\R^{n}\to\R$, let $\vnorm{f}_{L_{2}(\gamma_{n})}\colonequals(\int_{\R^{n}}\abs{f(x)}^{2}d\gamma_{n}(x))^{1/2}$.  Let $L_{2}(\gamma_{n})\colonequals\{f\colon\R^{n}\to\R\colon\vnorm{f}_{L_{2}(\gamma_{n})}<\infty\}$.  For any $x\in\R^{n}$ and any $r>0$, let $B(x,r)\colonequals\{y\in\R^{n}\colon\vnorm{x-y}_{2}<r\}$.  Let $f\colon\R^{n}\to[0,1]$ and let $\rho\in[-1,1]$,  define the \textit{Ornstein-Uhlenbeck operator with correlation $\rho$} applied to $f$ by
\begin{equation}\label{oudef}
T_{\rho}f(x)\colonequals\int_{\R^{n}}f(x\rho+y\sqrt{1-\rho^{2}})d\gamma_{n}(y),\qquad\forall\,x\in\R^{n}.
\end{equation}
$T_{\rho}$ is a parametrization of the Ornstein-Uhlenbeck operator.  $T_{\rho}$ is not a semigroup, but it satisfies $T_{\rho_{1}}T_{\rho_{2}}=T_{\rho_{1}\rho_{2}}$, as we will see below.  We have chosen this definition since the usual Ornstein-Uhlenbeck operator is only defined for $\rho\in[0,\pi/2]$.

\begin{definition}[Noise Stability]
Let $n$ be a positive integer.  Let $\rho\in(-1,1)$.  Let $A\subset\R^{n}$ be a measurable set.  Define the \textit{Noise Stability of $A$ with correlation $\rho$} to be
$$\int_{\R^{n}}1_{A}(x)T_{\rho}1_{A}(x)d\gamma_{n}(x).$$
\end{definition}

\subsection{The Symmetric Gaussian Problem}

Conjecture \ref{conj0} appeared in \cite{chakrabarti10} in relation to the Gap-Hamming-Distance problem in communication complexity.  There, the following inequality was proven using concentration of measure techniques.  In particular, the following property was used: when $n$ is large, most of the measure of $\gamma_{n}$ is concentrated near the sphere of radius $\sqrt{n}$ centered at the origin.

\begin{theorem}[{\cite[Corollary 3.6]{chakrabarti10}}]\label{thm00}
For all $c,\epsilon>0$, there exists $\delta,N>0$ such that, for all $n>N$, for all $0\leq\rho\leq c/\sqrt{n}$, and for all $A,B\subset\R^{n}$ with $\gamma_{n}(A)\geq e^{-\delta n}$ with $A=-A$, we have
$$\int_{\R^{n}}1_{A}(x)T_{\rho}1_{B}(x)d\gamma_{n}(x)\geq(1-\epsilon)\gamma_{n}(A)\gamma_{n}(B).$$
\end{theorem}

A sharper estimate of the right side would give sharper lower bounds for the Gap-Hamming-Distance problem.  Some related versions of Theorem \ref{thm00} were investigated in \cite{sherstov12} and \cite{vidick13a}.
%\begin{remark}
%...
%\end{remark}
%
%
%--rho near 1 - still unresolved even when restricted to compact, self-shrinking surfaces, and with volume constraint removed

%Discussion, statement of results, see also \cite{khot09,khot11,khot12,heilman11,heilman12}, \cite{chakrabarti10,colding11}.

The following conjecture is suggested in \cite{barthe01,chakrabarti10,odonnell12}.  Conjecture \ref{SGP} below is a formal re-statement of Conjecture \ref{conj0}.

\begin{conj}\label{SGP}
\textbf{(Symmetric Gaussian Problem)}  Let $0<a,b<1$, $-1\leq\rho\leq1$ and let $A,B\subset\R^{n}$ with $\gamma_{n}(A)=a$, $\gamma_{n}(B)=b$. Let $r_{a},r_{b},r_{a}',r_{b}'>0$ so that $\gamma_{n}(B(0,r_{a}))=a$, $\gamma_{n}(B(0,r_{b})^{c})=b$, $\gamma_{n}(B(0,r_{a}')^{c})=a$, and $\gamma_{n}(B(0,r_{b}'))=b$.  If $\rho>0$, then $(B(0,r_{a}),B(0,r_{b})^{c})$ or $(B(0,r_{a}')^{c},B(0,r_{b}'))$ achieves the following infimum
\begin{equation}\label{six1.9}
\inf_{\substack{A,B\subset\R^{n}\colon\\ \gamma_{n}(A)=a,\gamma_{n}(B)=b,\,A=-A}}\int_{\R^{n}}1_{A}(x)T_{\rho}(1_{B})(x)d\gamma_{n}(x).
\end{equation}
If $\rho<0$, the same result holds, with the additional restriction that $B=-B$ in \eqref{six1.9}.
\end{conj}

To see that Conjecture \ref{SGP} is equivalent to that of \cite{chakrabarti10,odonnell12}, let $A,B\subset\R^{n}$ and observe
\begin{flalign*}
&\int_{\R^{n}} 1_{A}(x)T_{\rho}(x)1_{B}d\gamma_{n}(x)
=\int_{\R^{n}} 1_{A}(x)\int 1_{B}(x\rho+y\sqrt{1-\rho^{2}})d\gamma_{n}(y)d\gamma_{n}(x)\\
&\qquad\qquad=\iint_{\R^{n}} 1_{A}(x)1_{B}(x\rho+y\sqrt{1-\rho^{2}})d\gamma_{n}(y)d\gamma_{n}(x)
=\mathbb{P}((X,Y)\in A\times B).
\end{flalign*}
Here $X=(X^{(1)},\ldots,X^{(n)}),Y=(Y^{(1)},\ldots,Y^{(n)})$ are jointly normal standard $n$-dimensional Gaussian random variables such that the covariances satisfy $\mathbb{E}(X^{(i)}Y^{(j)})=\rho\cdot1_{\{i=j\}}$.

Restricting Conjecture \ref{SGP} to the case $a+b=1$ and $A=B^{c}$ gives the following special case of Conjecture \ref{SGP}.

\begin{conj}\label{SGP1}
\textbf{(Symmetric Gaussian Problem, Quadratic Version)}  Let $0< a<1$, $-1\leq\rho\leq1$ and let $A\subset\R^{n}$ with $\gamma_{n}(A)=a$.  Let $r_{a},r_{a}'>0$ so that $\gamma_{n}(B(0,r_{a}))=a$, $\gamma_{n}(B(0,r_{a}')^{c})=a$.  Then $B(0,r_{a})$ or $B(0,r_{a}')^{c}$ achieves the following supremum
\begin{equation}\label{six1.9b}
\sup_{\substack{A\subset\R^{n}\colon\\ \gamma_{n}(A)=a,\,A=-A}}\int_{\R^{n}}1_{A}(x)T_{\rho}(1_{A})(x)d\gamma_{n}(x).
\end{equation}
%If $\rho<0$, the same result holds, with the additional restriction that $B=-B$ in \eqref{six1.9}.
\end{conj}

\subsection{Our Contribution}

In this paper, we provide two distinct approaches to Conjectures \ref{SGP} and \ref{SGP1}.  In our first approach, we examine the first variation of the noise stability directly, as in \cite{heilman12}.  This approach is used in our first main result, Theorem \ref{thm1}.
%In our second approach, in Theorem \ref{thm1.5}, we use Taylor expansions to explicitly compute the second-degree Hermite-Fourier coefficients of the interior of ellipses in the %plane.  We then exploit an unexpected radial symmetry in these formulas in Lemma \ref{lemma43}.
In our second approach, we compute the second variation of noise stability for balls and their complements, which reduces to proving certain Gaussian Poincar\'{e}-type inequalities in Lemma \ref{lemma50}.  The second-variation approach appears to be the first application of second-variation arguments to noise stability problems.  We show that, for certain measure restrictions $0<a<1$, the ball or its complement locally maximizes noise stability.  But for other measure restrictions $a$, the ball or its complement does \textit{not} locally maximize noise stability.  As a result, Conjectures \ref{SGP} and \ref{SGP1} are false, for certain measure restrictions $a$.

Here is our first main result.

\begin{theorem}[Conjecture \ref{SGP}, $n=1$, $\rho$ small]\label{thm1}
Let $n=1$, $0<a,b<1$, and let $\abs{\rho}<\min(e^{-40},\min(a^{20},(1-a)^{20})\min(b^{20},(1-b)^{20}))/1000$.  Then Conjecture \ref{SGP} holds with these parameters.  Consequently, Conjecture \ref{SGP1} holds with these parameters.
\end{theorem}

The proof of Theorem \ref{thm1} adapts the strategy of \cite{heilman12}, though the case $n=1$ of Conjecture \ref{SGP} provides several simplifications compared to the fairly intricate geometric arguments of \cite{heilman12}.  Also, the proof in \cite{heilman12} was only able to handle correlations $\rho>0$, whereas the present paper can handle both positive and negative correlations $\rho$.

Unfortunately, already when $n=2$, Conjecture \ref{SGP1} is incorrect (and consequently Conjecture \ref{SGP} is incorrect).  To see why, let $A\subset\R^{2}$ and define
\begin{equation}\label{zero1}
F(A)=\left(\int_{A}(1-x_{1}^{2})d\gamma_{2}(x)\right)^{2}+\left(\int_{A}(1-x_{2}^{2})d\gamma_{2}(x)\right)^{2}.
\end{equation}
If Conjecture \ref{SGP1} is correct, then by differentiating twice with respect to $\rho$ at $\rho=0$ (see \eqref{one3} below for details), Conjecture \ref{SGP1} for $n=2$ implies that the ball or its complement maximizes $F(A)$ among all symmetric sets $A\subset\R^{2}$ with $\gamma_{2}(A)=a$.  (Without loss of generality, if $A$ maximizes noise stability, then we may apply a rotation to $A$ if necessary to ensure that $\int_{A}x_{1}x_{2}d\gamma_{2}(x)=0$.)  However, $F(A)$ is not always maximized by the ball or its complement.  To see this, define
$$A=\{(x_{1},x_{2})\in\R^{2}\colon x_{1}^{2}/(2.5)^{2}+x_{2}^{2}/(2.31394)^{2}\leq1\}.$$

Let $r=2.4$.  Then $\gamma_{2}(B(0,r))=1-e^{-r^{2}/2}\approx .943865$.  And from Lemma \ref{newlem} below, $F(B(0,r))=F(B(0,r)^{c})=\frac{1}{2}r^{4}e^{-r^{2}}\approx.0522732$.

And if $r'=\sqrt{-2\log(1-e^{-2.88})}$, then $\gamma_{2}(B(0,r'))=\gamma_{2}(B(0,r)^{c})$, and again from Lemma \ref{newlem} below, $F(B(0,r'))=\frac{1}{2}(r')^{4}e^{-(r')^{2}}\approx.0059468$.

Finally, a numerical computation shows that $\gamma_{2}(A)\approx.943865$, and $F(A)=F(A^{c})\approx .0524720>.0522732$.

That is, $\gamma_{2}(B(0,r)^{c})=\gamma_{2}(B(0,r'))\approx\gamma_{2}(A^{c})$, but
$$F(A^{c})>\max(F(B(0,r)^{c}),F(B(0,r'))).$$
That is, Conjectures \ref{SGP} and \ref{SGP1} are false.  A few more details are provided for this numerical calculation in Remark \ref{goodrk} below.  Furthermore, note that if $A'=\{(x_{1},x_{2})\in\R^{2}\colon \abs{x_{1}}\leq 1.90999\}$, then a numerical calculation shows that $\gamma_{2}(A')\approx.943865$, and $F(A')=F((A')^{c})\approx.0604796$.  So, for this measure restriction, the strip actually has larger value than the ball, or the complement of a ball, or the ellipse $A$.  So, at least for this measure restriction, the symmetric set maximizing $F$ seems to be the strip $A'$.  If so, this result would agree with the S-inequality (formerly the S-conjecture) proven in \cite{latala99}, which implies in particular that: for any symmetric convex set $A$ with $\gamma_{2}(A)=\gamma_{2}(A')$, and for any $t\geq1$, we have $\gamma_{2}(tA)\geq\gamma_{2}(tA')$.
%i can't see a way to make the implicit function theorem work to make the
% Psi(x,t,\overline{t}).....
%

The fact that Conjectures \ref{SGP} and \ref{SGP1} are false for $n\geq2$ is especially surprising since they are more or less known to be true in the limit when $\rho\to1$, by e.g. \cite{colding12}.  At very least, the boundary of a set $A$ which maximizes noise stability in the limit $\rho\to1$ should be a minimal surface, i.e. a surface of constant mean curvature.  On the other hand, the example above and Theorem \ref{newthm} suggest that strips could maximize noise stability for small $\rho$, as in the main result of \cite{barthe01}.  In fact, it is entirely unclear which symmetric set maximizes $F$, and it is unclear which symmetric set maximizes noise stability.  It could be the case that noise stability among symmetric sets of fixed Gaussian measure is maximized when $A\subset\R^{n}$ has boundary which is a dilation of the set $S^{k}\times\R^{n-k-1}$ for some $0\leq k\leq n-1$. However, this statement could also be false.

We did not arrive at the above counterexamples by accident.  In fact it is generally true that when a ball or its complement has a large radius, then that set does not maximize the quantity \eqref{zero1}.  This fact is made rigorous by computing the second variation of the quantity \eqref{zero1}.  Before presenting this second variation result, we establish some notation.

Let $A\subset\R^{n}$ be a set with smooth boundary, and let $N\colon\partial A\to S^{n-1}$ denote the unit exterior normal to $\partial A$.  Let $X\colon\R^{n}\to\R^{n}$ be a vector field.  Let $\Psi\colon\R^{n}\times(-1,1)\to\R^{n}$ such that $\Psi(x,0)=x$ and such that $\frac{d}{dt}\Psi(x,t)=X(\Psi(x,t))$ for all $x\in\R^{n},t\in(-1,1)$.  For any $t\in(-1,1)$, let $A^{(t)}=\Psi(A,t)$.    Note that $A^{(0)}=A$.  Let $G\colon\R^{n}\times\R^{n}\to\R$ be a Schwartz function, e.g. to investigate noise stability we let $G(x,y)=(2\pi)^{-n}e^{\frac{-\vnorm{x}_{2}^{2}-\vnorm{y}_{2}^{2}+2\rho\langle x,y\rangle}{2(1-\rho^{2})}}$ $\forall$ $x,y\in\R^{n}$.  Or to investigate the functional in \eqref{zero1}, we let $G(x,y)=\sum_{i=1}^{n}(1-x_{i}^{2})(1-y_{i}^{2})\gamma_{n}(x)\gamma_{n}(y)$.   Define
$$V(x,t)\colonequals\int_{A^{(t)}}G(x,y)dy,\qquad V\colon\R^{n}\times(-1,1)\to\R.$$

\begin{theorem}[Second Variation Formula, {\cite[Theorem 2.6]{chokski07}}]\label{thm4}
Let
$$F(A)\colonequals \int_{\R^{n}}\int_{\R^{n}} 1_{A}(x)G(x,y)1_{A}(y)dxdy.$$  Then
%For any $x\in\partial A$, let $N(x)$ denote the outward pointing unit normal vector of $A$ at $x$.  For any $t\in(-1,1)$, let $A^{(t)}\subset\R^{n}$ so that $A^{(0)}=A$.  .... Then a normal variation of $A$ satisfies
\begin{flalign*}
\frac{1}{2}\frac{d^{2}}{dt^{2}}F(A^{(t)})|_{t=0}
&=\int_{\partial A}\int_{\partial A}G(x,y)\langle X(x),N(x)\rangle\langle X(y),N(y)\rangle dxdy\\
&\qquad+\int_{\partial A}\mathrm{div}(V(x,0)X(x))\langle X(x),N(x)\rangle dx.
\end{flalign*}
\end{theorem}

When integrating on a surface $\partial A$, we let $dx$ denote the restriction of Lebesgue measure to the surface $\partial A$.

The second variation formula of Theorem \ref{thm4} essentially appears in \cite{chokski07}, though their statement differs a bit from ours.  Nevertheless, their proof immediately gives Theorem \ref{thm4}.  We reproduce the details of the proof of \cite[Theorem 2.6]{chokski07} in the Appendix, Section \ref{sec2var}.  Theorem \ref{thm4} does not seem to have been applied to noise stability before.  In particular, optimizing noise stability has typically focused on either first variation arguments, or on heat flow methods.  So, we consider our application of the second variation formula to the noise stability functional to be one of the main contributions of this work.

Combining Theorem \ref{thm4} with a Poincar\'{e}-type inequality on the sphere (Lemma \ref{lemma50} below), we deduce the following second variation calculation for the functional $F$ when $G(x,y)=\sum_{i=1}^{n}(1-x_{i}^{2})(1-y_{i}^{2})\gamma_{n}(x)\gamma_{n}(y)$ $\forall$ $x,y\in\R^{n}$.  This second variation computation constitutes our second main result.

\begin{theorem}[Local Optimality and Non-Optimality Conditions]\label{newthm}
Let $n\geq2$.  Let $r>0$ such that $r^{2}\leq n+2$.  Assume that $\frac{d}{dt}|_{t=0}\gamma_{n}(A^{(t)})=0$ and $\frac{d^{2}}{dt^{2}}|_{t=0}\gamma_{n}(A^{(t)})=0$.  Then if $A^{(0)}=B(0,r)$ or if $A^{(0)}=B(0,r)^{c}$, we have
$$\frac{d}{dt}F(A^{(t)})|_{t=0}=0,\qquad\frac{d^{2}}{dt^{2}}F(A^{(t)})|_{t=0}<0.$$
That is, the sets $B(0,r)$ and $B(0,r)^{c}$ locally maximize the sum of the squares of their second-degree Fourier coefficients among symmetric sets of fixed Gaussian measure.

However, if $r^{2}> n+2$, then $\exists$ symmetric sets $\{A^{(t)}\}_{-1<t<1}$ so that $\frac{d}{dt}|_{t=0}\gamma_{n}(A^{(t)})=0$, $\frac{d^{2}}{dt^{2}}|_{t=0}\gamma_{n}(A^{(t)})=0$, with $A^{(0)}=B(0,r)$ or $A^{(0)}=B(0,r)^{c}$, and such that
$$\frac{d}{dt}F(A^{(t)})|_{t=0}=0,\qquad\frac{d^{2}}{dt^{2}}F(A^{(t)})|_{t=0}>0.$$
That is, the sets $B(0,r)$ and $B(0,r)^{c}$ do \textbf{not} locally maximize the sum of the squares of their second-degree Fourier coefficients among symmetric sets of fixed Gaussian measure.
\end{theorem}

The equality $\frac{d}{dt}F(A^{(t)})|_{t=0}=0$ follows readily from \eqref{one4} below, and the second variation calculation is contained in Corollary \ref{varbndcor} below.
%\begin{remark}
%\end{remark}
%\snote{Would ideally make remark about Implicit function theorem here; can't seem to iron out the details}
%

The condition $r^{2}\leq n+2$ unfortunately does not hold for all measure restrictions as $n\to\infty$.  That is, there are many measure restrictions $a=\gamma_{n}(B(0,r))$ where $r^{2}>n+2$.  To see this, let $r(s,n)=\sqrt{n+s\sqrt{2n}}$ for any $s>0$, $n\in\N$.  Then $\lim_{n\to\infty}\gamma_{n}(B(0,r(s,n)))=\int_{-\infty}^{s}\gamma_{1}(t)dt$, which follows from the Central Limit Theorem.  And $\lim_{n\to\infty}[(r(s,n))^{2}-n-2]<0$ if and only if $s<0$.  That is, the ball $B(0,r(s,n))$ only locally maximizes $F$ for sufficiently large $n$ when $\lim_{n\to\infty}\gamma_{n}(B(0,r(s,n)))\leq1/2$.  And the complement $B(0,r(s,n))^{c}$ only locally maximizes $F$ for sufficiently large $n$ when $\lim_{n\to\infty}\gamma_{n}(B(0,r(s,n))^{c})\geq1/2$.

As we observed in the case $n=2$, with $r=2.4$ and with $r'=\sqrt{-2\log(1-e^{-2.88})}\approx.3399$, we had $\gamma_{2}(B(0,r)^{c})=\gamma_{2}(B(0,r'))$ with $F(B(0,r)^{c})>F(B(0,r'))$.  Theorem \ref{newthm} then says that Conjectures \ref{SGP} and \ref{SGP1} are false in a fairly strong sense, since if the radius of the ball or complement is sufficiently large, then that set does not locally maximize $F$.  It therefore seems natural to try to formulate a weaker version of Conjecture \ref{SGP1} which only identifies sets of large noise stability as $n\to\infty$.  Such a statement may still be suitable for applications to the Gap-Hamming-Distance problem as well.

\begin{question}\label{q1}
For any $a\in(0,1)$, $n\geq1$, let $B_{n,a}\subset\R^{n}$ be the ball centered at the origin such that $\gamma_{n}(B_{n,a})=a$.
%define $\Phi(a)=\int_{-\infty}^{a}\gamma_{1}(t)dt$.  For any $A\subset\R^{n}$, let
%$$F(A)=\sum_{i=1}^{n}\left(\int_{A}(1-x_{i}^{2})d\gamma_{n}(x)\right)^{2}.$$
%Let $a\colonequals\gamma_{n}(A)$, so that $\lim_{n\to\infty}\gamma_{n}\left(B\left(0,\sqrt{n+\sqrt{2n}\Phi^{-1}(a)}\,\right)\right)=\Phi(\Phi^{-1}(a))=a$.

If $A=-A$, and if $\gamma_{n}(A)=a$, is it true that
%$$F(A)\leq\lim_{n\to\infty}F\left(B\left(0,\sqrt{n+\sqrt{2n}\Phi^{-1}(a)}\,\right)\right)?$$
$$F(A)\leq\sup_{n\geq1}\big[\max(F(B_{n,a}),F(B_{n,1-a}^{c}))\big]?$$
%That is, (using our computation of the right side from \eqref{seven24} below), is it true that
%$$F(A)\leq\frac{1}{\pi}e^{-(\Phi^{-1}(a))^{2}}?$$
%In particular, if $a=1/2$, and if $A=-A$, is it true that
%$$F(A)\leq\frac{1}{\pi}?$$
\end{question}

We would like to change the above question, replacing $\sup_{n\geq1}$ with $\lim_{n\to\infty}$ or $\limsup_{n\to\infty}$.  Unfortunately, the following calculation shows that, for any $0<a<.15$, quantity $F(B_{n,a})$ decreases monotonically when $n$ is very large.

From the Central Limit Theorem with error bound (also known as the Edgeworth Expansion) \cite[XVI.4.(4.1)]{feller71}, for any $s\in\R$, the following asymptotic expansion holds as $n\to\infty$:
$$\gamma_{n}\left(B\left(0,\sqrt{n+s\sqrt{2n}}\,\right)\right)=\int_{-\infty}^{s}e^{-t^{2}/2}dt/\sqrt{2\pi}+\frac{(1-s^{2})e^{-s^{2}/2}/\sqrt{2\pi}}{\sqrt{n}}+o(n^{-1/2}).$$
And an asymptotic expansion in Lemma \ref{lemma25} shows that, as $n\to\infty$, if $B(0,\sqrt{n+s\sqrt{2n}})\subset\R^{n}$, we then have
$$F\left(B\left(0,\sqrt{n+s\sqrt{2n}}\,\right)\right)=\frac{1}{\pi}e^{-s^{2}+s^{3}2\sqrt{2}n^{-1/2}/3-O(n^{-1})},\qquad\forall\,s\in\R.$$
%These two facts taken together suggest that $F(B_{n,a})$ can actually decrease when $n$ is large and $s>1$ (i.e. when $a>.84$ in Question \ref{q1}).

For $n$ sufficiently large, if $B(0,\sqrt{n+s\sqrt{2n}})\subset\R^{n}$ and $s<-1$, then $\gamma_{n}(B(0,\sqrt{n+s\sqrt{2n}}))$ decreases as $n$ increases, and $F(B(0,\sqrt{n+s\sqrt{2n}}))$ increases as $n$ increases.  And if $s>1$, then $\gamma_{n}(B(0,\sqrt{n+s\sqrt{2n}})^{c})$ increases as $n$ increases, and $F(B(0,\sqrt{n+s\sqrt{2n}}))$ decreases as $n$ increases.  Therefore, if $a<.15$, $F(B_{n,a})<F(B_{n,1-a}^{c})$, and $B_{n,1-a}^{c}$ does not locally maximize $F$, by Theorem \ref{newthm}. That is, when $0<a<.15$, when $n$ is large, and when $\abs{\rho}$ is near zero, Conjectures \ref{SGP} and \ref{SGP1}.

If $r>0$, then since $B(0,r)\subset\R^{n}$ is rotationally symmetric, we conclude that $F(B(0,r))=\frac{1}{n}(\int_{B(0,r)}(n-\sum_{i=1}^{n}x_{i}^{2})d\gamma_{n}(x))^{2}$.  So, if $n$ is fixed,  $\max_{a>0}F(B_{n,a})=F(B(0,\sqrt{n}))$.  Also, if $B(0,\sqrt{n})\subset\R^{n}$, then $F(B(0,\sqrt{n}))< F(B(0,\sqrt{n+1}))$, by Lemma \ref{lemma25} below.  And since $\lim_{n\to\infty}F(B(0,\sqrt{n}))=\frac{1}{\pi}$, a variant of Question \ref{q1} can be:
\begin{question}\label{q1.1}
If $A\subset\R^{n}$ with $A=-A$ (with no restriction on the measure of $A$), is it true that
$$F(A)\leq\frac{1}{\pi}?$$
\end{question}

If Question \ref{q1} is incorrect, then Conjecture \ref{SGP} is false in a much stronger sense than mentioned above.  That is, if Question \ref{q1} is incorrect, then there exists some $A\subset\R^{n}$ such that $\partial A$ is a the level set of a degree two polynomial such that $A$ has larger noise stability than any ball or ball complement in any Euclidean space of any dimension (for noise stability with small correlation $\rho$).  If Question \ref{q1} is correct, then Question \ref{q1} can be interpreted as an ``infinite-dimensional'' special case of the following weakened ``infinite-dimensional'' version of Conjecture \ref{SGP}:

\begin{question}\label{q2}
Let $0<\rho<1$.  For any $a\in(0,1)$, $n\geq1$, let $B_{n,a}\subset\R^{n}$ be the ball centered at the origin such that $\gamma_{n}(B_{n,a})=a$.
If $A=-A$, and if $\gamma_{n}(A)=a$, is it true that
$$F(A)\leq\sup_{n\geq1}\big[\max(F(B_{n,a}),F(B_{n,1-a}^{c}))\big]?$$

If $A=-A$, and if $\gamma_{n}(A)=a$, is it true that
\begin{flalign*}
&\int_{\R^{n}}1_{A}(x)T_{\rho}(1_{A})(x)d\gamma_{n}(x)
\leq\sup_{n\geq1}\Big(\max\big(\int_{\R^{n}}1_{B_{n,a}}(x)T_{\rho}(1_{B_{n,a}})(x)d\gamma_{n}(x),\\
&\qquad\qquad\qquad\qquad\qquad\qquad\qquad\quad
\int_{\R^{n}}1_{B_{n,1-a}^{c}}(x)T_{\rho}(1_{B_{n,1-a}^{c}})(x)d\gamma_{n}(x)\big)\Big)?
\end{flalign*}
\end{question}

% P(X1^2+...+Xn^2\leq n+s\sqrt{2n}) = P((X1^2 +...+Xn^2 - n)/\sqrt{2n}\leq s )
% now E X1^4 =3, so var(X1^2)= E X1^4 - (EX1^2)^2=3-1=2

%Since the noise stability of sets is essentially a variational problem involving an $L_{\infty}$ constraint, the ``usual'' framework of the calculus of variations does not apply for taking variations of noise stability.  That is, for variational arguments for noise stability, it is less desirable to treat a set $A\subset\R^{n}$ as a function $1_{A}$, and then to vary $1_{A}$ by adding other functions to $1_{A}$.  It is more helpful to change the set $A$ using normal variations $\Psi(A,t)$ as in Theorem \ref{thm4}.

Despite the negative results mentioned above, including Theorem \ref{newthm}, we present some evidence toward positive answers to Question \ref{q1} and \ref{q2}.  First, Theorem \ref{newthm} can be extended to handle the noise stability functional, as long as $\rho$ is small.  Before stating this Theorem, for any $A\subset\R^{n}$, and for any $-1<\rho<1$, let
$$F_{\rho}(A)=\int_{\R^{n}}1_{A}(x)T_{\rho}(1_{A}(x))d\gamma_{n}(x).$$
Below, we continue to use the notational conventions of Theorem \ref{thm4}.

\begin{theorem}[Local Optimality and Non-Optimality Conditions for Noise Stability]\label{thm3}
Assume $\int_{\partial A}\abs{\langle X(x),N(x)\rangle}^{2}dx=1$.

Let $r>0$ such that $r^{2}\leq n+2$.  Let $0<a<1$.  Assume that $\gamma_{n}(A^{(0)})=a$.  Then there exists $\rho_{0}=\rho_{0}(a,r,n)>0$ such that, for all $\abs{\rho}<\rho_{0}$, the following holds.  Assume that $\frac{d}{dt}|_{t=0}\gamma_{n}(A^{(t)})=0$ and $\frac{d^{2}}{dt^{2}}|_{t=0}\gamma_{n}(A^{(t)})=0$.  Then if $A^{(0)}=B(0,r)$ or if $A^{(0)}=B(0,r)^{c}$, we have
$$\frac{d}{dt}F_{\rho}(A^{(t)})|_{t=0}=0,\qquad\frac{d^{2}}{dt^{2}}F_{\rho}(A^{(t)})|_{t=0}<0.$$
That is, the sets $B(0,r)$ and $B(0,r)^{c}$ locally maximize noise stability among symmetric sets of fixed Gaussian measure.

However, if $r^{2}> n+2$, if $0<a<1$, then there exists $\rho_{0}=\rho_{0}(a,r,n)>0$ such that, for all $\abs{\rho}<\rho_{0}$, the following holds.  There exist symmetric sets $\{A^{(t)}\}_{-1<t<1}$ so that $\gamma_{n}(A^{(0)})=a$, $\frac{d}{dt}|_{t=0}\gamma_{n}(A^{(t)})=0$, $\frac{d^{2}}{dt^{2}}|_{t=0}\gamma_{n}(A^{(t)})=0$, with $A^{(0)}=B(0,r)$ or $A^{(0)}=B(0,r)^{c}$, and such that
$$\frac{d}{dt}F_{\rho}(A^{(t)})|_{t=0}=0,\qquad\frac{d^{2}}{dt^{2}}F_{\rho}(A^{(t)})|_{t=0}>0.$$
That is, the sets $B(0,r)$ and $B(0,r)^{c}$ do \textbf{not} locally maximize noise stability among symmetric sets of fixed Gaussian measure.
\end{theorem}

\begin{remark}\label{rotrk}
Let $A\subset\R^{n}$ with $A=-A$.  Since $A=-A$, $\frac{d}{d\rho}|_{\rho=0}\int_{\R^{n}}1_{A}(x)T_{\rho}1_{A}(x)d\gamma_{n}(x)=0$ (see \eqref{one1} below).  Moreover, we have the Taylor series estimate
\begin{equation}\label{newone}
\abs{\int_{\R^{n}}1_{A}(x)T_{\rho}1_{A}(x)d\gamma_{n}(x)
-\left[(\gamma_{n}(A))^{2}+\frac{\rho^{2}}{2}\frac{d^{2}}{d\rho^{2}}|_{\rho=0}\int_{\R^{n}}1_{A}(x)T_{\rho}1_{A}(x)d\gamma_{n}(x)\right]}
\leq\abs{\rho}^{3}.
\end{equation}
So, the second derivative of noise stability at $\rho=0$ is most significant when $\rho$ is near zero.  In fact, as shown in \eqref{one3} below,
$$
\frac{d^{2}}{d\rho^{2}}|_{\rho=0}\int_{\R^{n}}1_{A}(x)T_{\rho}1_{A}(x)d\gamma_{n}(x)
=\sum_{i=1}^{n}(\int_{A}(1-x_{i}^{2})d\gamma_{n}(x))^{2}
+2\sum_{i,j\in\{1,\ldots,n\}\colon i\neq j}(\int_{A}x_{i}x_{j}d\gamma_{n}(x))^{2}.
$$
However, the last sum can be ignored for the following reason.  If we interpret $x\in\R^{n}$ as a column vector, consider the $n\times n$ matrix $M\colonequals\int_{A}xx^{T}d\gamma_{n}(x)$.  Then if $Q$ is an orthogonal $n\times n$, matrix, a change of variables shows that $QMQ^{T}=\int_{A}(Qx)(Qx)^{T}d\gamma_{n}(x)=\int_{QA}xx^{T}d\gamma_{n}(x)$.  So, to diagonalize $M$ with an orthogonal matrix $Q$, it suffices to replace $A$ with the set $QA$.  So, when we maximize $\int_{\R^{n}}1_{A}(x)T_{\rho}1_{A}(x)d\gamma_{n}(x)$ or when we consider variations of the noise stability, we can and will assume that $\int_{A^{(t)}}x_{i}x_{j}d\gamma_{n}(x)=0$ for any $i,j\in\{1,\ldots,n\}$ with $i\neq j$, for any $t\in(-1,1)$.  In particular (using $t=0$), we can and will assume that $\int_{A}x_{i}x_{j}d\gamma_{n}(x)=0$ for any $i,j\in\{1,\ldots,n\}$ with $i\neq j$

%\snote{This is okay for noise stability since that is rotation invariant.  But $F$ is not rotation invariant.  Maybe should add: if ball or its complement maximizes sum of squares of all second order Fourier coefficients, then it also maximizes $F$.  Also, maximizing the sum of squares is equivalent to maximizing $F$ among sets $A$ with $M$ diagonal.  Second variation of $F$ ends up being maximized by sets with $M$ diagonal.}
\end{remark}

%And Theorem \ref{thm3} identifies the maximum of this second derivative quantity.
%
%Unfortunately, the relation between noise stability of a ball in $\R^{n}$ and in $\R^{n+1}$ is not immediately obvious when the respective balls have the same volume.  We therefore use an asymptotic expansion to show the following result, which can be considered positive evidence toward Questions \ref{q1} and \ref{q2}.  The argument of Theorem \ref{thm5} is tedious though straightforward.

%For any $x,y\in\R^{n}$, define $G(x,y)=\sum_{i=1}^{n}(1-x_{i}^{2})(1-y_{i}^{2})\gamma_{n}(x)\gamma_{n}(y)$, and for any $A\subset\R^{n}$ define
%
%$$F(A)\colonequals \int_{\R^{n}}\int_{\R^{n}} 1_{A}(x)G(x,y)1_{A}(y)dxdy.$$
%\begin{theorem}\label{thm5}
%If $0<a<1$, and if $B_{n}\subset\R^{n}$ is a ball centered at the origin with $\gamma_{n}(B_{n})=a$, then $F(B_{n+1})>F(B_{n})$ for $n$ sufficiently large.
%\end{theorem}

%To prove Theorem \ref{thm2}, we use the second variation formula below, which may be of independent interest.

%..............
%\begin{definition}\label{sdef}  We consider the set of bounded symmetric functions
%$$
%S(\gamma_{n})\colonequals\{f\colon\R^{n}\to\R\colon\forall\,x\in\R^{n},f(x)=f(-x),0\leq f(x)\leq1\}
%$$
%\end{definition}
%
%................

\subsection{General Framework}

Though Conjecture \ref{SGP} and other noise stability optimization problems concern the optimization of a very specific functional, i.e. noise stability, our treatment of Conjecture \ref{SGP} uses a fairly general strategy.  That is, we can consider our approach to Conjecture \ref{SGP} within the following general context:

\begin{itemize}
\item We are given some Banach space $V$, and for each $\rho\in(-1,1)$, we have a function $F_{\rho}\colon V\to\R$ to be maximized.
\item The maximum of $F_{0}$ over $V$ is equivalent to maximizing $F_{0}$ over a finite-dimensional manifold.
\item We would like to show: if $v_{0}\in V$ maximizes $F_{0}$, then $v_{0}$ also maximizes $F_{\rho}$ for all $\rho$ close to $0$.
\end{itemize}

It is generally impossible that the final statement holds.  For example, suppose we are asked to maximize $F_{\rho}(v)=-(v-\rho)^{2}$ where $v\in \R$.  Then $v=0$ maximizes $F_{0}$, but $v=0$ does not maximize $F_{\rho}$ when $\rho\neq0$.

Our main strategy in proving Theorem \ref{thm1} is to try to relate the first variation (i.e. first derivative) of $F_{\rho}$ to that of $F_{0}$ when $\rho$ is near $0$.
\begin{itemize}
\item[(i)] Prove some stability estimate for $F_{0}$.  (If $v$ nearly maximizes $F_{0}$, then $v$ is close to $v_{0}$.)
\item[(ii)] Show that if $\rho$ is close to $0$, then the first variation of $F_{\rho}$ is close to that of $F_{0}$.
\item[(iii)] Assume that $F_{\rho}$ depends continuously on $\rho$, if $v$ maximizes $F_{\rho}$, then $v$ nearly maximizes $F_{0}$.  So, $v$ is close to $v_{0}$ by (ii).
\item[(iv)] Since $v$ is close to $v_{0}$ by (iii), an appropriate version of $(ii)$ implies that $v$ is very close to $v_{0}$.  Then, by iterating $(ii)$ an infinite number of times, we conclude that $v=v_{0}$, as desired.
\end{itemize}

This strategy was used in \cite{heilman12} to show that if $\rho$ is close to zero, then the maximum noise stability of three sets partitioning $\R^{n}$ each with Gaussian measure $1/3$ occurs when the three sets are cones each with cone angle $2\pi/3$.  This appeared to be the first use of this strategy applied to noise stability problems.  However, a similar strategy has been used for perturbations of perimeter functionals \cite{julin14,figalli15}

In the present paper, we will consider the Banach $V$ consisting of symmetric bounded functions: $f\colon\R^{n}\to[0,1]$ with $f(-x)=f(x)$.  So, if $A$ is symmetric, i.e. $A=-A$ and $A\subset\R^{n}$, then $1_{A}\in V$.  We will also let $F_{\rho}$ be the noise stability.

The strategy depicted above, as used in \cite{heilman12}, however has some shortcomings for Conjecture \ref{SGP} when $n\geq2$.  In particular, part (iv) of the above strategy seems most natural only when we impose the additional restriction that the set $A\subset\R^{n}$ satisfies $\int_{A}(1-x_{i}^{2})d\gamma_{n}(x)=\int_{A}(1-x_{j}^{2})d\gamma_{n}(x)$ for all $i,j\in\{1,\ldots,n\}$.  This assumption imposes additional constraints beyond the assumption that $A=-A$.  It is possible to prove a version of Theorem \ref{thm1} under this additional constraint, but we choose not to do so, since this constraint seems too restrictive to be of interest.
%However, in some sense this condition is not quite so restrictive, since if we are given $A\subset\R^{n}$, we can replace $1_{A}$ with the averaged function %$f=\frac{1}{n!}\sum_{i=1}^{n!}1_{P_{i}A}$, where $P_{1},\ldots,P_{n!}$ are the set of all $n\times n$ permutation matrices.  Then the function $f$ satisfies %$\int_{\R^{n}}f(x)(1-x_{i}^{2})d\gamma_{n}(x)=\int_{\R^{n}}f(x)(1-x_{j}^{2})d\gamma_{n}(x)$ for all $i,j\in\{1,\ldots,n\}$, but the noise stability of $f$ will be smaller than that of %$A$.

In any case, in order to prove Theorems \ref{newthm} and \ref{thm3}, we abandon the strategy of \cite{heilman12}, and we instead use the second variation formula Theorem \ref{thm4}.  That is, we change the itemized strategy above to the following.
\begin{itemize}
\item[(i)] Compute the second variation of $F_{0}$ is negative at a ball centered at the origin.
\item[(ii)] Show that if $\rho$ is close to $0$, then the second variation of $F_{\rho}$ is close to that of $F_{0}$.
\end{itemize}

\subsection{Organization}

Sections \ref{secvar} through \ref{2setvar} provide supporting lemmas for the proof of Theorem \ref{thm1}, which appears in Section \ref{secmain}.
%Theorem \ref{thm1.5} is proven in Section \ref{secdim2}.
Theorems \ref{newthm} and \ref{thm3} are proven in Section \ref{secopt}.

\subsection{Some Hermite-Fourier Analysis}

Let $\lambda>0$.  Recall that the Hermite polynomials $h_{0},h_{1},h_{2},\ldots$ of one variable are defined by
$$e^{\lambda x-\lambda^{2}/2}\equalscolon\sum_{\ell\in\N}\lambda^{\ell}h_{\ell}(x),\qquad\forall\,x\in\R.$$
Note that $\int_{\R}h_{\ell}(x)^{2}d\gamma_{1}(x)=1/\ell!$, and $\{\sqrt{\ell!}\,h_{\ell}\}_{\ell\in\N}$ is an orthonormal basis of $L_{2}(\gamma_{1})$ with respect to the inner product $\langle f,g\rangle=\int_{\R}f(x)\overline{g(x)}d\gamma_{1}(x)$, $f,g\colon\R\to\R$.  Set $f(x)\colonequals e^{\lambda x-\lambda^{2}/2}$.  A routine computation \cite{heilman12} shows that $T_{\rho}(f)(x)=e^{(\lambda\rho)x-(\lambda\rho)^{2}/2}$, $\forall$ $x\in\R$, $\forall$ $\rho\in(-1,1)$.

We therefore have the relation
\begin{equation}\label{six1}
T_{\rho}f(x)=\sum_{\ell\in\N}\lambda^{\ell}\rho^{\ell}h_{\ell}(x),\qquad\forall\,x\in\R,\quad\forall\,\rho\in(-1,1).
\end{equation}
So, by linearity, $T_{\rho}h_{\ell}(x)=\rho^{\ell}h_{\ell}(x)$, $\forall$ $x\in\R$, $\forall$ $\ell\in\N$, $\forall$ $\rho\in(-1,1)$.

We now extend the above observations to higher dimensions.  Let $f\in L_{2}(\gamma_{n})$, so that $f=\sum_{\ell\in\N^{n}}\langle f,h_{\ell}\sqrt{\ell!}\rangle h_{\ell}\sqrt{\ell!}$ in the $L_{2}(\gamma_{n})$ sense, where $\ell=(\ell_{1},\ldots,\ell_{n})\in\N^{n}$ and $h_{\ell}(x)=\prod_{i=1}^{n}h_{\ell_{i}}(x_{i})$ $\forall$ $x\in\R^{n}$.  Write $\vnorm{\ell}_{1}\colonequals \ell_{1}+\cdots+\ell_{n}$ and $\ell!\colonequals (\ell_{1}!)\cdots(\ell_{n}!)$.  Then $T_{\rho}$ satisfies $T_{\rho}h_{\ell}=\rho^{\vnorm{\ell}_{1}}h_{\ell}$ for any $\ell\in\N^{n}$, and for any $\rho\in(-1,1)$,
\begin{equation}\label{six1.8}
T_{\rho}f(x)=\sum_{\ell\in\N^{n}}\rho^{\vnorm{\ell}_{1}}\sqrt{\ell!}\,h_{\ell}(x)\left(\int_{\R}\sqrt{\ell!}\,h_{\ell}(y)f(y)d\gamma_{n}(y)\right),
\qquad\forall\,x\in\R^{n}.
\end{equation}

Let $f,g\in L_{2}(\gamma_{n})$.  By Plancherel's Theorem and \eqref{six1.8} we have
\begin{equation}\label{one1}
\int_{\R^{n}} f(x) T_{\rho}g(x)d\gamma_{n}(x)=\sum_{\ell\in\N^{n}}\rho^{\vnorm{\ell}_{1}}
\int_{\R^{n}} f(x)\sqrt{\ell!}h_{\ell}(x)d\gamma_{n}(x)
\int_{\R^{n}} g(y)\sqrt{\ell!}h_{\ell}(y)d\gamma_{n}(y).
\end{equation}
%By formally taking derivatives of \eqref{one1}, we get
%%By formally taking the derivative $d/d\rho$ of \eqref{one1} at $\rho=0$, we get
%\begin{equation}\label{one2}
%\left.\frac{d}{d\rho}\right|_{\rho=0}\int f T_{\rho}(1-f)d\gamma_{n}
%=-\vnorm{\int_{\R^{n}}xf(x)d\gamma_{n}(x)}_{\ell_{2}^{n}}^{2}.
%\end{equation}
By formally taking the second derivative $d^{2}/d\rho^{2}$ of \eqref{one1}, we get
\begin{equation}\label{one2.5}
\frac{d^{2}}{d\rho^{2}}\int f T_{\rho}fd\gamma_{n}
=\sum_{\ell\in\N^{n}}\vnorm{\ell}_{1}(\vnorm{\ell}_{1}-1)\rho^{\vnorm{\ell}_{1}-2}\abs{\int f\sqrt{\ell!}h_{\ell}d\gamma_{n}}^{2}.
\end{equation}
Evaluating \eqref{one2.5} at $\rho=0$,
\begin{equation}\label{one3}
\left.\frac{d^{2}}{d\rho^{2}}\right|_{\rho=0}\int f T_{\rho}fd\gamma_{n}
=2\sum_{\ell\in\N^{n}\colon\vnorm{\ell}_{1}=2}\abs{\int f\sqrt{\ell!}h_{\ell}d\gamma_{n}}^{2}.
\end{equation}

Suppose $f(x)=f(-x)$, $\forall$ $x\in\R^{n}$.  Applying this property and then changing variables,
\begin{equation}\label{one3.1}
\int_{\R^{n}}h_{\ell}(x)f(x)d\gamma_{n}(x)=(-1)^{\vnorm{\ell}_{1}}\int_{\R^{n}}h_{\ell}(-x)f(-x)d\gamma_{n}(x)
=(-1)^{\vnorm{\ell}_{1}}\int_{\R^{n}}h_{\ell}(x)f(x)d\gamma_{n}(x).
\end{equation}
%Therefore, $\int_{\R^{n}}xf(x)d\gamma_{n}(x)=0$, i.e. $\frac{d}{d\rho}|_{\rho=0}\int f T_{\rho}(1-f)d\gamma_{n}=0$.
%
%first derivative is \abs{j}\cos\theta^{\abs{j}-1}(-\sin\theta)\abs{a_{j}}^{2}
%second derivative is
% \abs{j}\cos\theta^{\abs{j}-1}(-\cos\theta)\abs{a_{j}}^{2}+\abs{j}(\abs{j}-1)\cos\theta^{\abs{j}-2}(-\sin\theta)^{2}\abs{a_{j}}^{2}
%so evaluate at \theta=\pi/2, get second derivative of
% \abs{j}(\abs{j}-1)\abs{a_{2}}^{2}

%\section{Maximizing the First Eigenvalue}

\section{Maximizing Second Degree Fourier Coefficients}\label{secvar}

%[Need to fix these lemmas to also work for $\rho\in(-1,1)$, with and without restrictions on the measure.  For example, need to symmetrize $U_{1},U_{2}$, then it should all work.]
%[change this lemma to just work for sum of second degree coeffs]

We begin with the following adaptation of \cite[Lemma 2.1]{khot11}.  The following Lemma provides existence and first variation conditions \eqref{one4} for maximizing the second degree term in \eqref{one1}, or equivalently the second derivative term in \eqref{newone}.

\begin{lemma}\label{lemma0}
Let $0<a<1$.  Then there exists $A\subset\R^{n}$ such that
$$
\sum_{\substack{\ell\in\N^{n}\colon\\ \vnorm{\ell}_{1}=2}}\left(\int_{\R^{n}}1_{A}(x)h_{\ell}(x)\sqrt{\ell!}d\gamma_{n}(x)\right)^{2}
=\sup_{\substack{\{f\colon\R^{n}\to[0,1],\\ \int_{\R^{n}} f(x)d\gamma_{n}(x)=a\}}}
\sum_{\substack{\ell\in\N^{n}\colon\\ \vnorm{\ell}_{1}=2}}\left(\int_{\R^{n}}f(x)h_{\ell}(x)\sqrt{\ell!}d\gamma_{n}(x)\right)^{2}.
$$
Moreover, $A=-A$, and there exists $c\in\R$ such that
\begin{equation}\label{one4}
A=\left\{x\in\R^{n}\colon
\int_{\R^{n}}1_{A}(y)\left[\sum_{i=1}^{n}(x_{i}^{2}-1)(y_{i}^{2}-1)+2\sum_{i\neq j}(x_{i}x_{j}y_{i}y_{j})\right]d\gamma_{n}(y)\geq c\right\}.
\end{equation}
%note, the sum "double counts" the sum over i,j
\end{lemma}
\begin{proof}
The set $C\colonequals\{f\colon\R^{n}\to[0,1],\int_{\R^{n}} f(x)d\gamma_{n}(x)=a\}$ is a norm closed, convex and norm bounded subset of the Hilbert space $L_{2}(\gamma_{n})$.  Therefore, $C\subset L_{2}(\gamma_{n})$ is weakly closed.  Also, $C$ is weakly compact by the Banach-Alaoglu Theorem.  Define $T\colon C\to\R$ by
\begin{equation}\label{one8}
T(f)\colonequals\sum_{\ell\in\N^{n}\colon\vnorm{\ell}_{1}=2}\abs{\int_{\R^{n}} f(x)\sqrt{\ell!}h_{\ell}(x)d\gamma_{n}(x)}^{2}.
\end{equation}
Since $\vnormf{h_{\ell}\sqrt{\ell!}}_{L_{2}(\gamma_{n})}=1$, $T$ is a finite sum of weakly continuous functions.  Therefore, $T$ is weakly continuous on the weakly compact set $C\subset L_{2}(d\gamma_{n})$.  So, there exists $f\in C$ such that $T(f)=\max_{g\in C}T(g)$.

Now, the function $f_{s}(x)\colonequals(f(x)+f(-x))/2$ satisfies
$$
\int_{\R^{n}} f_{s}(x)\sqrt{\ell!}h_{\ell}(x)d\gamma_{n}(x)
=
\begin{cases}
\int_{\R^{n}} f(x)\sqrt{\ell!}h_{\ell}(x)d\gamma_{n}(x) & ,\vnorm{\ell}_{1}\,\mbox{even, }\ell\in\N^{n}\\
0 & ,\vnorm{\ell}_{1}\,\mbox{odd, }\ell\in\N^{n}.
\end{cases}
$$
So, $T(f_{s})\geq T(f)$.  Let $C^{s}\colonequals\{f\colon\R^{n}\to[0,1],f(x)=f(-x),\int_{\R^{n}} f(x)d\gamma_{n}(x)=a\}$.  We have just shown that
\begin{equation}\label{two3}
\max_{g\in C}T(g)=\max_{g\in C^{s}}T(g).
\end{equation}

We therefore try to maximize $T(g)$ on $C^{s}$.  We now show that $T$ is convex on $B^{s}$.  Let $f_{1},f_{2}\in C^{s}$, and let $\lambda\in[0,1]$.  Then
\begin{equation}\label{two4}
\lambda T(f_{1})+(1-\lambda)T(f_{2})-T(\lambda f_{1}+(1-\lambda)f_{2})
\stackrel{\eqref{one8}}{=}\lambda(1-\lambda)T(f_{1}-f_{2})
\geq0.
\end{equation}
So, $T$ is a weakly continuous convex function on the weakly compact set $C^{s}\subset L_{2}(\gamma_{n})$.  Therefore, there exists $A\subset\R^{n}$ such that $1_{A}\in C^{s}$ satisfies $T(1_{A})=\max_{g\in C^{s}}T(g)$ \cite[Lemma 2.1]{khot11}.  Combining this observation with \eqref{two3}, $T(1_{A})=\max_{g\in C^{s}}T(g)=\max_{g\in C}T(g)$, and $A=-A$ since $1_{A}\in C^{s}$.  The existence of $A$ is therefore proven.

We now prove \eqref{one4}.  We argue by contradiction.  Define
$$\overline{T}(f)(x)\colonequals\left.\frac{1}{2}\frac{d^{2}}{d\rho^{2}}\right|_{\rho=0}T_{\rho}f(x)
\stackrel{\eqref{six1.8}}{=}\sum_{\ell\in\N^{n}\colon\vnorm{\ell}_{1}=2}\left(\int_{\R^{n}}f(y)\sqrt{\ell!}h_{\ell}d\gamma_{n}(y)\right)\sqrt{\ell!}h_{\ell}(x).$$
Note that $\int_{\R^{n}} f(x)\overline{T}f(x)d\gamma_{n}(x)=T(f)$.  Suppose there exists $x_{1},x_{2}\in\R^{n}$, $x_{1}\notin A,x_{2}\in A$ such that $\overline{T}1_{A}(x_{1})>\overline{T}1_{A}(x_{2})$.  Let $U_{1}\subset\R^{n}$ be a small ball around $x_{1}$ and let $U_{2}$ be a small ball around $x_{2}$ such that $\overline{T}1_{A}(u_{1})>\overline{T}1_{A}(u_{2})$, $\forall$ $u_{1}\in U_{1},u_{2}\in U_{2}$.  Also, assume that $U_{1}\cap U_{2}=\emptyset$ and $\gamma_{n}(U_{1})=\gamma_{n}(U_{2})$.  Define $A'\colonequals (A\setminus U_{2})\cup U_{1}$.  Then $1_{A'}=1_{A}-1_{U_{2}}+1_{U_{1}}$, and for $U_{1},U_{2}$ sufficiently small,
\begin{flalign*}
\int_{\R^{n}}1_{A'}(x)\overline{T}1_{A'}(x)d\gamma_{n}(x)
&=\int_{\R^{n}} 1_{A}(x)\overline{T}1_{A}(x)d\gamma_{n}(x)
+2\int_{\R^{n}}(1_{U_{1}}(x)-1_{U_{2}}(x))\overline{T}1_{A}(x)d\gamma_{n}(x)\\
&\qquad+\int_{\R^{n}}(1_{U_{1}}(x)-1_{U_{2}}(x))\overline{T}(1_{U_{1}}-1_{U_{2}})(x)d\gamma_{n}(x)\\
&>\int_{\R^{n}} 1_{A}(x)\overline{T}1_{A}(x)d\gamma_{n}(x).
\end{flalign*}

This inequality contradicts the maximality of $A$.  We conclude that no such $x_{1},x_{2}$ exist, so \eqref{one4} holds.
\end{proof}

\begin{remark}
One difficulty in proving Question \ref{q2} is that there are many potential critical points for the noise stability.  For example, if the boundary of $A$ is of the form $S^{m}\times\R^{n-m}$, with $0\leq m\leq n$, then $A$ satisfies \eqref{one4}.  Also, if the boundary of the set $A$ is any Simons-Lawson cone, then $A$ satisfies \eqref{one4}.  That is, in the limit $\rho\to0$, $A$ is a candidate critical point of noise stability if the boundary of $A$ is equal to
$$\{(x_{1},\ldots,x_{2n})\in\R^{2n}\colon \sum_{i=1}^{n}x_{i}^{2}=\sum_{i=n+1}^{2n}x_{i}^{2}\}.$$
\end{remark}

\section{Iterative Estimates}

The following inequality for Hermite polynomials will be useful in the sequel.

\begin{lemma}[{\cite[Lemma 5.1]{heilman12}}]\label{lemma6}
For $\ell=(\ell_{1},\ldots,\ell_{n})\in\N^{n}$ and $x=(x_{1},\ldots,x_{n})\in\R^{n}$,
$$
h_{\ell}(x)\sqrt{\ell!}
\leq\vnorm{\ell}_{1}^{n}3^{\vnorm{\ell}_{1}}\prod_{i=1}^{n}\max(1,\abs{x_{i}}^{\ell_{i}}).
$$
%And for $\abs{\ell_{i}}$ odd,
%$$
%h_{\ell}(x)\sqrt{\ell!}
%\leq\abs{x_{i}}\abs{\ell}^{n}3^{\abs{\ell}}\prod_{i=1}^{n}\max(1,\abs{x_{i}}^{\ell_{i}}).
%$$
%$$
%h_{\ell}(x)\sqrt{\ell!}\leq e^{x^{2}/4}.
%$$
\end{lemma}
%The second estimate is proven in \cite{indritz61}.
%
%H_{n}(x)\leq(2^{n}n!)^{1/2}e^{x^{2}/2}
%h_{n}(x)=2^{-n/2}(n!)^{-1}H_{n}(x/\sqrt{2})
%h_{n}(x)\sqrt{n!}\leq2^{-n/2}2^{n/2}(n!)^{-1}(n!)^{1/2}(n!)^{1/2}e^{(x/\sqrt{2})^{2}/2}
% so h_{n}(x)\leq e^{x^{2}/4}

Below we will also require the following bounds on $T_{\rho}$ applied to the indicator function of an interval.

\begin{lemma}\label{lemma7}
Let $B=B(0,r)\subset\R$ with $\gamma_{1}(B)=a$.  Let $x\in\R$ with $\abs{x}\leq\sqrt{-4\log\abs{\rho}}$, and let $\abs{\rho}<e^{-40}$.  Then
%$$
%\sqrt{2\pi}\frac{d}{dx}T_{\rho}1_{B}(x)=\rho e^{-(r+x\rho)^{2}/[2(1-\rho^{2})]}-\rho e^{-(r-x\rho)^{2}/[2(1-\rho^{2})]}.
%$$
$$
\bigg|\frac{d}{dx}T_{\rho}1_{B}(x)+\rho^{2}\sqrt{\frac{2}{\pi}}xre^{-r^{2}/2}\bigg|\leq\min(a^{1/2},(1-a)^{1/2})10\abs{\rho}^{15/4}.
$$
Also, for any $f\colon\R\to[-1,1]$,
$$
\frac{d}{dx}T_{\rho}f(x)=\frac{\rho}{\sqrt{1-\rho^{2}}}\int_{\R} yf(x\rho+y\sqrt{1-\rho^{2}})d\gamma_{1}(y).
$$
%(d/dx)T_{\rho}f(x)=\int (d/dx)f(x\rho+y\sqrt{1-\rho^{2}})e^{-y^{2}/2}dy/\sqrt{2\pi}
%                  =\int(df/dz)(x\rho+y\sqrt{1-\rho^{2}})\rho e^{-y^{2}/2}dy/\sqrt{2\pi}
%                  =\int(d/dy)f(x\rho+y\sqrt{1-\rho^{2}})(\rho/\sqrt{1-\rho^{2}})e^{-y^{2}/2}dy/\sqrt{2\pi}
%                  =-(\rho/\sqrt{1-\rho^{2}})\int f(x\rho+y\sqrt{1-\rho^{2}})(-y)e^{-y^{2}/2}dy/\sqrt{2\pi}
\end{lemma}
\begin{proof}
%\begin{flalign*}
%\frac{d}{dx}T_{\rho}1_{B}(x)
%&=\frac{d}{dx}\int_{\R}1_{B}(x\rho+y\sqrt{1-\rho^{2}})d\gamma_{1}(y)
%=\frac{d}{dx}\int_{(-r-x\rho)/\sqrt{1-\rho^{2}}}^{(r-x\rho)/\sqrt{1-\rho^{2}}}e^{-y^{2}/2}dy/\sqrt{2\pi}\\
%&=-\rho e^{-(r-x\rho)^{2}/[2(1-\rho^{2})]}/\sqrt{2\pi}+\rho e^{-(r+x\rho)^{2}/[2(1-\rho^{2})]}/\sqrt{2\pi}.
%\end{flalign*}
%
%1_{[-r,r]}(x\rho+y\sqrt{1-\rho^{2}})  %% x\rho+y\sqrt{1-\rho^{2}}\in[-r,r]
%=1_{[-r-x\rho,r-x\rho]}(y\sqrt{1-\rho^{2}})   %%  y\sqrt{1-\rho^{2}}\in[-r-x\rho,r-x\rho]
%=1_{[(-r-x\rho)/\sqrt{1-\rho^{2}},(r-x\rho)/\sqrt{1-\rho^{2}}]}(y)  %%y\in(1-\rho^{2})^{-1/2}[-r-x\rho,r-x\rho]

Recall that $h_{\ell}(x)=\sum_{m=0}^{\lfloor \ell/2\rfloor}\frac{x^{\ell-2m}(-1)^{m}2^{-m}}{m!(\ell-2m)!}$.  So, $h_{1}(x)=x$, $h_{2}(x)=(1/2)(x^{2}-1)$, $(d/dx)h_{\ell}=h_{\ell-1}$ for $\ell\geq1$, and $\int_{\R} 1_{B}(x)h_{2}(x)\sqrt{2}d\gamma_{1}(x)=-re^{-r^{2}/2}/\sqrt{\pi}$.
%$$
%\int_{-r}^{r}e^{-\alpha x^{2}/2}dx=\int_{-\sqrt{\alpha}r}^{\sqrt{\alpha}r}e^{-x^{2}/2}dx\alpha^{-1/2}.
%$$
%$$
%\frac{d}{d\alpha}\int_{0}^{z\sqrt{\alpha}}f(t)dt=f(z\sqrt{\alpha})\frac{z}{2\sqrt{\alpha}}.
%$$
%\begin{flalign*}
%\int_{-r}^{r}(-x^{2}/2)e^{-x^{2}/2}dx
%&=\frac{d}{d\alpha}|_{\alpha=1}\int_{-r}^{r}e^{-\alpha x^{2}/2}dx\\
%&=-\frac{1}{2}\int_{-r}^{r}e^{-x^{2}/2}dx+\frac{d}{d\alpha}|_{\alpha=1}\int_{-\sqrt{\alpha}r}^{\sqrt{\alpha}r}e^{-x^{2}/2}dx\\
%&=-(1/2)a+e^{-r^{2}/2}r.
%\end{flalign*}
%$$\int_{-r}^{r}(-x^{2}/2+1/2)e^{-x^{2}/2}dx=re^{-r^{2}/2}.$$
Since $\gamma_{1}(B)=a$, $\vnorm{1_{B}}_{L_{2}(\gamma_{n})}=a^{1/2}$ and $\vnorm{1_{B^{c}}}_{L_{2}(\gamma_{n})}=(1-a)^{1/2}$.  Then, for $h_{\ell}$ with $\ell\geq1$, the Cauchy-Schwarz inequality implies that
\begin{equation}\label{two6}
\begin{aligned}
\abs{\int_{\R} 1_{B}(x)h_{\ell}(x)\sqrt{\ell!}d\gamma_{1}(x)}
&=\min\left(\,\abs{\int_{\R} 1_{B}(x)h_{\ell}(x)\sqrt{\ell!}d\gamma_{1}(x)},\abs{\int_{\R} 1_{B^{c}}(x)h_{\ell}(x)\sqrt{\ell!}d\gamma_{1}(x)}\,\right)\\
&\leq\min(a^{1/2},(1-a)^{1/2}).
\end{aligned}
\end{equation}

By \eqref{six1.8}, and using $\int_{\R}1_{B}(x)h_{1}(x)d\gamma_{1}(x)=\int_{\R}1_{B}(x)h_{3}(x)d\gamma_{1}(x)=0$, which follows from \eqref{one3.1}, we have for any $x\in\R$,
\begin{flalign*}
\frac{d}{dx}T_{\rho}1_{B}(x)
&=\sum_{\ell\in\N}\rho^{\abs{\ell}}\left(\int_{\R}1_{B}(y)h_{\ell}(y)\sqrt{\ell!}d\gamma_{1}(y)\right)h_{\ell-1}(x)\sqrt{\ell!}\\
&=-\rho^{2}\sqrt{2}xre^{-r^{2}/2}/\sqrt{\pi}+\sum_{\ell\geq4}
\rho^{\abs{\ell}}\left(\int_{\R}1_{B}(y)h_{\ell}(y)\sqrt{\ell!}d\gamma_{1}(y)\right)h_{\ell-1}(x)\sqrt{\ell!}.\\
\end{flalign*}

Then, by \eqref{two6} and Lemma \ref{lemma6}, if $\abs{x}\leq\sqrt{-4\log\rho}$,
\begin{flalign*}
\bigg|\frac{d}{dx}T_{\rho}1_{B}(x)+\rho^{2}\sqrt{\frac{2}{\pi}}xre^{-r^{2}/2}\bigg|
&\leq\min(a^{1/2},(1-a)^{1/2})\sum_{\ell\geq4}
\abs{\rho}^{\abs{\ell}}3^{\abs{\ell}-1}(\abs{\ell}-1)\max(1,\abs{x}^{\ell-1})\\
&\leq10 \min(a^{1/2},(1-a)^{1/2})\abs{\rho}^{15/4}.
\end{flalign*}
\end{proof}

\section{Perturbation of Fourier Coefficients}\label{secpert}
%-EL equation should say (d^{2}/d\rho^{2})T_{\rho}f=constant
%-optimizing second fourier coefficient is trivial
%-then iterative argument should work

When $n=1$, we would like to show: if $B'\subset\R$ exists such that $f=1_{B'}$ nearly maximizes \eqref{one8}, then $B'$ is close to a ball $B\subset\R$ centered at the origin.  When $n=1$, this statement amounts to a simple rearrangement argument.

\begin{lemma}\label{lemma10}
For any $x\in\R$, let $g(x)=1-x^{2}$, and let $0<c<d$.  Then
$$
g(c)\geq\frac{1}{\gamma_{1}([c,d])}\int_{c}^{d}g(x)d\gamma_{1}(x)\geq g(d-(d-c)/3).
$$
\end{lemma}
%\begin{lemma}\label{lemma11}
%Let $0\leq k\leq4$, $k\in\Z$, $t>0$.  Then
%$$
%\int_{t}^{\infty}y^{k}d\gamma_{1}(y)\leq e^{-t^{2}/2}(t^{k}+3).
%$$
%\end{lemma}
%h_{2}(x)=(1/2)(x^{2}-1)
\begin{lemma}\label{lemma8.0}
Let $n=1$.  Let $B=B(0,r)$ such that $\gamma_{1}(B)=a$ and
\begin{equation}\label{three0.0}
\int_{\R}1_{B}(x)\sqrt{2!}h_{2}(x)d\gamma_{1}(x)
=\inf_{\substack{\{f\colon\R\to[0,1],\\ \int_{\R} f(x)d\gamma_{1}(x)=a\}}}
\int_{\R} f(x)\sqrt{2!}h_{2}(x)d\gamma_{1}(x).
\end{equation}
Let $B'\subset\R$.  Assume that there is an $\epsilon>0$ such that $B'$ satisfies
\begin{equation}\label{three1.0}
\int_{\R}1_{B'}(x)\sqrt{2!}h_{2}(x)d\gamma_{1}(x)
<\inf_{\substack{\{f\colon\R\to[0,1],\\ \int_{\R} f(x)d\gamma_{1}(x)=a\}}}
\int_{\R} f(x)\sqrt{2!}h_{2}(x)d\gamma_{1}(x)
+\epsilon a.
\end{equation}
%Let $\widetilde{a}$ such that $\int_{0}^{1}d\gamma_{1}=\widetilde{a}$.  Then there exists $0<\epsilon<\epsilon_{0}(a)$ such that
Then
\begin{equation}\label{three2.0}
\int_{\R}\abs{1_{B'}(x)-1_{B}(x)}d\gamma_{1}(x)<10\epsilon^{1/2}.
\end{equation}
\end{lemma}
%change all A to B'
\begin{proof}
We use a rearrangement argument.  Note that
$$\gamma_{1}(B'\setminus B)=\gamma_{1}(B')-\gamma_{1}(B'\cap B)=\gamma_{1}(B)-\gamma_{1}(B'\cap B)=\gamma_{1}(B\setminus B').$$
%So shift mass from $A\setminus B$ to $B\setminus A$.  Note that $(A\setminus B)\cap(B\setminus A)=\emptyset$.  The gain is greater than or equal to the gain from adding the smallest values of $h_{2}$.
Since $B=(B'\cap B)\cup(B\setminus B')$ and $B'=(B'\cap B)\cup(B'\setminus B)$,
\begin{equation}\label{three3}
\int_{B}(1-x^{2})d\gamma_{1}(x)-\int_{B'}(1-x^{2})d\gamma_{1}(x)
=\int_{B\setminus B'}(1-x^{2})d\gamma_{1}(x)-\int_{B'\setminus B}(1-x^{2})d\gamma_{1}(x).
\end{equation}

Let $r_{0}\in[0,r)$ such that $\gamma_{1}([r_{0},r])=(1/2)\gamma_{1}(B\setminus B')$, and let $r_{1}\in(r,\infty]$ such that $\gamma_{1}([r,r_{1}])=(1/2)\gamma_{1}(B'\setminus B)$.  Then, since $(B\setminus B')\subset B=B(0,r)$,
$$\int_{B\setminus B'}(1-x^{2})d\gamma_{1}(x)
\geq2\int_{r_{0}}^{r}(1-x^{2})d\gamma_{1}(x).
$$
Also, since $(B'\setminus B)\subset B^{c}=B(0,r)^{c}$,
$$
\int_{B'\setminus B}(1-x^{2})d\gamma_{1}(x)
\leq2\int_{r}^{r_{1}}(1-x^{2})d\gamma_{1}(x).
$$
Let $f(x)\colonequals1-x^{2}$.  From \eqref{three3} and Lemma \ref{lemma10},
\begin{equation}\label{three4}
\begin{aligned}
&\int_{B}(1-x^{2})d\gamma_{1}(x)-\int_{B'}(1-x^{2})d\gamma_{1}(x)
\geq2\int_{r_{0}}^{r}(1-x^{2})d\gamma_{1}(x)-2\int_{r}^{r_{1}}(1-x^{2})d\gamma_{1}(x)\\
&\quad=\gamma_{1}(B\setminus B')\left(\frac{1}{\gamma_{1}([r_{0},r])}\int_{r_{0}}^{r}(1-x^{2})d\gamma_{1}(x)
-\frac{1}{\gamma_{1}([r,r_{1}])}\int_{r}^{r_{1}}(1-x^{2})d\gamma_{1}(x)\right)\\
&\quad\geq\gamma_{1}(B\setminus B')(f(r-(r-r_{0})/3)-f(r)\\
&\quad\geq(4r/3)(r-r_{0})(1/3)\gamma_{1}(B\setminus B')
\geq(4/9)\sqrt{2\pi}r\gamma_{1}(B\setminus B')\gamma_{1}([r_{0},r])\\
&\quad\geq(2/9)\sqrt{2\pi}r\gamma_{1}(B\setminus B')^{2}
\geq(2\pi/9)\gamma_{1}(B\setminus B')^{2}a.
\end{aligned}
\end{equation}
%derivative at left endpoint, multiplied by increment
%also, (r-r_{0})\geq\int_{r_{0}}^{r}dx\geq\int_{r_{0}}^{r}e^{-x^{2}/2}dx=\sqrt{2\pi}\gamma_{1}[r_{0},r]=(\sqrt{2\pi}/2)\gamma_{1}(B\setminus A)
%
%Note that, by Lemma \ref{lemma9}, we may assume $r>1/6$.
Finally, by \eqref{three4} we have
\begin{equation}\label{three6}
\int_{B}(x^{2}-1)d\gamma_{1}(x)-\int_{B'}(x^{2}-1)d\gamma_{1}(x)
\leq-a(1/6)(\int\abs{1_{B'}(x)-1_{B}(x)}d\gamma_{1}(x))^{2}.
\end{equation}
%note that $\delta=2\gamma(B\setminus A)$
So, combining \eqref{three6}, \eqref{three0.0} and \eqref{three1.0}, $\int_{\R}\abs{1_{B'}(x)-1_{B}(x)}d\gamma_{1}(x)<10\epsilon^{1/2}$.
\end{proof}

%perturbative statement
\begin{lemma}\label{lemma8}
Let $n=1$.  Let $B=B(0,r')^{c}$ such that $\gamma_{1}(B)=a$ and
\begin{equation}\label{three0}
\int_{\R}1_{B}(x)\sqrt{2!}h_{2}(x)d\gamma_{1}(x)
=\sup_{\substack{\{f\colon\R\to[0,1],\\ \int_{\R} f(x)d\gamma_{1}(x)=a\}}}
\int_{\R} f(x)\sqrt{2!}h_{2}(x)d\gamma_{1}(x).
\end{equation}
Let $B'\subset\R^{1}$.  Assume that there is an $\epsilon>0$ such that $B'$ satisfies
\begin{equation}\label{three1}
\int_{\R}1_{B'}(x)\sqrt{2!}h_{2}(x)d\gamma_{1}(x)
>\sup_{\substack{\{f\colon\R\to[0,1],\\ \int_{\R} f(x)d\gamma_{1}(x)=a\}}}
\int_{\R} f(x)\sqrt{2!}h_{2}(x)d\gamma_{1}(x)
-\epsilon(1-a).
\end{equation}
%Let $\widetilde{a}$ such that $\int_{0}^{1}d\gamma_{1}=\widetilde{a}$.  Then there exists $0<\epsilon<\epsilon_{0}(a)$ such that
Then
\begin{equation}\label{three2}
\int_{\R}\abs{1_{B'}(x)-1_{B}(x)}d\gamma_{1}(x)<10\epsilon^{1/2}.
\end{equation}
\end{lemma}
\begin{proof}
Apply Lemma \ref{lemma8.0} to $B^{c}$.
\end{proof}

\section{Existence Lemma}\label{secexist}

We prove the existence of two sets which minimize Gaussian correlation.  The argument below is almost identical to Lemma \ref{lemma8}.

%R(x,y)-R(x_{n},y_{n})=R(x,y-y_{n})-R(x_{n}-x,y_{n})
\begin{lemma}\label{lemma14}
Let $\rho\in(0,1)$, $0<a,b<1$.  Then there exist $A,B\subset\R^{n}$ with $A=-A$, $B=-B$ such that
\begin{equation}\label{zero0}
\int_{\R^{n}}1_{A}(x)T_{\rho}1_{B}(x)d\gamma_{n}(x)
=\inf_{\substack{\{f,g\colon\R^{n}\to[0,1],\int_{\R^{n}} f(x)d\gamma_{n}(x)=a\\\int_{\R^{n}} g(x)d\gamma_{n}(x)=b,f(x)=f(-x)\,\forall\,x\in\R^{n}\}}}
\int_{\R^{n}} f(x)T_{\rho}g(x)d\gamma_{n}(x).
\end{equation}
If $\rho\in(-1,0)$, the same result holds, with the additional restriction $g(x)=g(-x)$ $\forall$ $x\in\R^{n}$ in \eqref{zero0}.
\end{lemma}
\begin{proof}
Define the set $C\colonequals\{f,g\colon\R^{n}\to[0,1],\int_{\R^{n}} f(x)d\gamma_{n}(x)=a,
\int_{\R^{n}} g(x)d\gamma_{n}(x)=b,f(x)=f(-x)\,\forall x\in\R^{n}\}$.  Then $C$ is a norm closed, convex and norm bounded subset of the Hilbert space $L_{2}(\gamma_{n})\oplus L_{2}(\gamma_{n})$.  Therefore, $C\subset L_{2}(\gamma_{n})\oplus L_{2}(\gamma_{n})$ is weakly closed.  Also, $C$ is weakly compact by the Banach-Alaoglu Theorem.  Define $T\colon C\to\R$ by
\begin{equation}\label{zero8}
T(f,g)\colonequals\int_{\R^{n}} f(x)T_{\rho}g(x)d\gamma_{n}(x).
\end{equation}
%(a+b)/2\leq(ab)^{1/2}, (ab)\leq(a^{2}+b^{2})^{1/2}
From the Cauchy-Schwarz inequality and \eqref{six1.8}, $\abs{T(f,g)}\leq\vnorm{f}_{L_{2}(\gamma_{n})}\vnorm{g}_{L_{2}(\gamma_{n})}$.  That is, $T$ is a strongly bounded bilinear function, so $T$ is weakly continuous.  So, $T$ is weakly continuous on the weakly compact set $C\subset L_{2}(d\gamma_{n})\oplus L_{2}(d\gamma_{n})$.  And there exist $f,g\in C$ such that $T(f,g)=\min_{(f',g')\in C}T(f',g')$.

From \eqref{six1.8}, we have the following absolutely convergent sum
\begin{equation}\label{zero2}
\begin{aligned}
&T(f,g)
=\sum_{\substack{\ell\in\N^{n}\colon\\\vnorm{\ell}_{1}\,\mathrm{even}}}\rho^{\vnorm{\ell}_{1}}
\int_{\R^{n}} f(x)\sqrt{\ell!}h_{\ell}(x)d\gamma_{n}(x)\int_{\R^{n}} g(y)\sqrt{\ell!}h_{\ell}(y)d\gamma_{n}(y)\\
&\qquad+\sum_{\substack{\ell\in\N^{n}\colon\\\vnorm{\ell}_{1}\,\mathrm{odd}}}\rho^{\vnorm{\ell}_{1}}
\int_{\R^{n}} f(x)\sqrt{\ell!}h_{\ell}(x)d\gamma_{n}(x)\int_{\R^{n}} g(y)\sqrt{\ell!}h_{\ell}(y)d\gamma_{n}(y).
\end{aligned}
\end{equation}
Since $f(x)=f(-x)$ for all $x\in\R^{n}$, the sum over odd terms in \eqref{zero2} is zero by \eqref{one3.1}.

Now, the function $g_{s}(x)\colonequals (g(x)+g(-x))/2$ satisfies
$$
\int_{\R^{n}} g_{s}(x)\sqrt{\ell!}h_{\ell}(x)d\gamma_{n}(x)
=
\begin{cases}
\int_{\R^{n}} g(x)\sqrt{\ell!}h_{\ell}(x)d\gamma_{n}(x) & ,\vnorm{\ell}_{1}\,\mbox{even}\\
0 & ,\vnorm{\ell}_{1}\,\mbox{odd}.
\end{cases}
$$

So, $T(f,g_{s})\leq T(f,g)$.  (If $\rho<0$, then we have already assumed that $g(x)=g(-x)$ for all $x\in\R^{n}$, so that $g_{s}=g$.)  Let $C^{s}\colonequals\{f,g\colon\R^{n}\to[0,1],f(x)=f(-x),g(x)=g(-x),\forall x\in\R^{n},\int_{\R^{n}} f(x)d\gamma_{n}(x)=a,\int_{\R^{n}} g(x)d\gamma_{n}(x)=b\}$.  We have just shown that
\begin{equation}\label{zero3}
\min_{(f',g')\in C}T(f',g')=\min_{(f',g')\in C^{s}}T(f',g').
\end{equation}

We therefore try to minimize $T$ on $C^{s}$.  But $T$ is linear in each of its arguments, and $T$ is a weakly continuous function on the weakly compact set $C^{s}\subset L_{2}(\gamma_{n})\oplus L_{2}(\gamma_{n})$.  Therefore, there exist $A,B\subset\R^{n}$ such that $1_{A},1_{B}\in C^{s}$ satisfy $T(1_{A},1_{B})=\min_{(f',g')\in C^{s}}T(f',g')$.  Combining this fact with \eqref{zero3}, $T(1_{A},1_{B})=\min_{(f',g')\in C^{s}}T(f',g')=\min_{(f',g')\in C}T(f',g')$, and $A=-A,B=-B$ since $(1_{A},1_{B})\in C^{s}$.
\end{proof}

\section{Perturbation Lemma}

Similar to Section \ref{secpert}, we show: if two sets $A,B\subset\R$ nearly minimize the product of their second-order Hermite-Fourier coefficients, then these sets are close to a ball and the complement of a ball, respectively.

\begin{lemma}\label{lemma18}
Let $n=1$, $0<a,b<1$.  Let $(A,B)=(B(0,r_{a}),B(0,r_{b}')^{c})$ or let $(A,B)=(B(0,r_{a}')^{c},B(0,r_{b}))$ such that $\gamma_{1}(A)=a,\gamma_{1}(B)=b$ and such that
\begin{equation}\label{four0}
\begin{aligned}
&\int_{\R}1_{A}h_{2}\sqrt{2!}(x)d\gamma_{1}(x)\int_{\R}1_{B}(y)h_{2}(y)\sqrt{2!}d\gamma_{1}(y)\\
&\qquad=\inf_{\substack{\{f,g\colon\R\to[0,1],\\ \int_{\R} f(x)d\gamma_{1}(x)=a,\int_{\R} g(y)d\gamma_{1}(y)=b\}}}
\left(\int_{\R} f(x)\sqrt{2!}h_{2}(x)d\gamma_{1}(x)\right)\left(\int_{\R} g(y)\sqrt{2!}h_{2}(y)d\gamma_{1}(y)\right).
\end{aligned}
\end{equation}
Let $A',B'\subset\R^{1}$ with $\gamma_{1}(A')=a$ and $\gamma_{1}(B')=b$.  Assume that there is an $\epsilon>0$ such that
\begin{equation}\label{four1}
\begin{aligned}
&\int_{\R}1_{A'}(x)h_{2}(x)\sqrt{2!}d\gamma_{1}(x)\int_{\R}1_{B'}(y)h_{2}(y)\sqrt{2!}d\gamma_{1}(y)\\
&\qquad<\inf_{\substack{\{f,g\colon\R\to[0,1],\\ \int_{\R} f(x)d\gamma_{1}(x)=a,\int_{\R} g(y)d\gamma_{1}(y)=b\}}}
\left(\int_{\R} f(x)\sqrt{2!}h_{2}(x)d\gamma_{1}(x)\right)\left(\int_{\R} g(y)\sqrt{2!}h_{2}(y)d\gamma_{1}(y)\right)\\
&\qquad\qquad\qquad
+\epsilon\min(\sqrt{a},\sqrt{1-a})\min(\sqrt{b},\sqrt{1-b}).
\end{aligned}
\end{equation}
%Let $\widetilde{a}$ such that $\int_{0}^{1}d\gamma_{1}=\widetilde{a}$.  Then there exists $0<\epsilon<\epsilon_{0}(a)$ such that
Then
\begin{equation}\label{four2.3}
\min\left(\int_{\R}\abs{1_{A}(x)-1_{A'}(x)}d\gamma_{1}(x),\int_{\R}\abs{1_{A^{c}(x)}-1_{A'}(x)}d\gamma_{1}(x)\right)<10\epsilon^{1/2}\min(a,1-a)^{-1/4},
\end{equation}
\begin{equation}\label{four2.4}
\min\left(\int_{\R}\abs{1_{B}(y)-1_{B'}(y)}d\gamma_{1}(y),\int\abs{1_{B^{c}}(y)-1_{B'}(y)}d\gamma_{1}(y)\right)<10\epsilon^{1/2}\min(b,1-b)^{-1/4}.
\end{equation}
\end{lemma}
\begin{proof}
%Should follow from Lemma \ref{lemma8}
Suppose without loss of generality that $(A,B)=(B(0,r_{a}')^{c},B(0,r_{b}))$.  First, note that there exists $\widetilde{B}\subset\R$ with $\gamma_{1}(\widetilde{B})=b$ such that
$$\int_{\R} 1_{\widetilde{B}}(y)\sqrt{2!}h_{2}(y)d\gamma_{1}(y)=\inf_{\substack{\{g\colon\R\to\R,0\leq g\leq 1\\ \int_{\R} g(y)d\gamma_{1}(y)=b\}}}\int_{\R} g(y)\sqrt{2!}h_{2}(y)d\gamma_{1}(y).$$
So, using $\vnormf{1_{\widetilde{B}}}_{L_{2}(\gamma_{1})}=\sqrt{b}$ and the Cauchy Schwarz inequality,
\begin{equation}\label{four2.9}
\begin{aligned}
&\bigg|\inf_{\substack{\{g\colon\R\to\R,0\leq g\leq 1\\ \int_{\R} g(y)d\gamma_{1}(y)=b\}}}\int_{\R} g(y)\sqrt{2!}h_{2}(y)d\gamma_{1}(y)\bigg|\\
&\qquad=\abs{\min\left(\int_{\R} 1_{\widetilde{B}}(y)\sqrt{2!}h_{2}(y)d\gamma_{1}(y),\int_{\R} 1_{\widetilde{B}^{c}}(y)\sqrt{2!}h_{2}(y)d\gamma_{1}(y)\right)}
\leq\min(\sqrt{b},\sqrt{1-b}).
\end{aligned}
\end{equation}

By \eqref{four1},
\begin{equation}\label{four2.1}
\int_{\R}1_{A'}(x)h_{2}(x)\sqrt{2!}d\gamma_{1}(x)>\sup_{\substack{\{f\colon\R\to[0,1],\\ \int_{\R} f(x)d\gamma_{1}(x)=a\}}}
\left(\int_{\R} f(x)\sqrt{2!}h_{2}(x)d\gamma_{1}(x)\right)-\epsilon\min(\sqrt{a},\sqrt{1-a}),
\end{equation}
\begin{equation}\label{four2.2}
\int_{\R}1_{B'}(y)h_{2}(y)\sqrt{2!}d\gamma_{1}(y)<\inf_{\substack{\{g\colon\R\to\R,0\leq g\leq 1,\\ \int_{\R} g(y)d\gamma_{1}(y)=b\}}}
\left(\int_{\R} g(y)\sqrt{2!}h_{2}(y)d\gamma_{1}(y)\right)+\epsilon\min(\sqrt{b},\sqrt{1-b}).
\end{equation}
%e.g., if int A...<sup...-e, then the product is >optimum+e*optimum 2nd term>optimum+e*min(b,1-b)/100>optimum+e*min(a,1-a)min(b-1-b)/100
For example, %if $\int_{\R}1_{A'}(x)h_{2}(x)\sqrt{2!}d\gamma_{1}(x)>0$, if $\int_{\R}1_{B'}(y)h_{2}(y)\sqrt{2!}d\gamma_{1}(y)<0$ and
if \eqref{four2.1} is false, then \eqref{four2.9} implies
\begin{flalign*}
&\int_{\R}1_{A'}(x)h_{2}(x)\sqrt{2!}d\gamma_{1}(x)\int_{\R}1_{B'}(y)h_{2}(y)\sqrt{2!}d\gamma_{1}(y)\\
&\geq\left[\int_{\R}1_{A'}(x)h_{2}(x)\sqrt{2!}d\gamma_{1}(x)\right]\cdot\inf_{\substack{\{g\colon\R\to\R,0\leq g\leq 1,\\ \int_{\R} g(y)d\gamma_{1}(y)=b\}}}
\left(\int_{\R} g(y)\sqrt{2!}h_{2}(y)d\gamma_{1}(y)\right)\\
&>\left[\sup_{\substack{\{f\colon\R\to[0,1],\\ \int_{\R} f(x)d\gamma_{1}(x)=a\}}}
\left(\int_{\R} f(x)\sqrt{2!}h_{2}(x)d\gamma_{1}(x)\right)-\epsilon\min(\sqrt{a},\sqrt{1-a})\right]\cdot\\
&\qquad\qquad\inf_{\substack{\{g\colon\R\to\R,0\leq g\leq 1,\\ \int_{\R} g(y)d\gamma_{1}(y)=b\}}}
\left(\int_{\R} g(y)\sqrt{2!}h_{2}(y)d\gamma_{1}(y)\right)\\
&\geq\inf_{\substack{\{f,g\colon\R\to[0,1],\\ \int_{\R} f(x)d\gamma_{1}(x)=a,\int_{\R} g(y)d\gamma_{1}(y)=b\}}}
\left(\int_{\R} f(x)\sqrt{2!}h_{2}(x)d\gamma_{1}(x)\right)\left(\int_{\R} g(y)\sqrt{2!}h_{2}(y)d\gamma_{1}(y)\right)\\
&\qquad\qquad\qquad+\epsilon\min(\sqrt{a},\sqrt{1-a})\min(\sqrt{b},\sqrt{1-b}).
\end{flalign*}
This inequality contradicts \eqref{four1}, so that \eqref{four2.1} holds.  Similarly, \eqref{four2.2} holds.

So, \eqref{four2.1}, \eqref{four2.2} and Lemmas \ref{lemma8.0} and \ref{lemma8} imply \eqref{four2.3} and \eqref{four2.4}.
% by assumption, we have, e.g. A<B+\epsilon\min(\sqrt{a},\sqrt{1-a})<B+\epsilon\sqrt{a}=B+\epsilon a(a^{-1/2}), so apply lemma
\end{proof}

\section{First Variation}\label{2setvar}

The following first variation argument is well-known.

\begin{lemma}\label{lemma000}
Let $\rho\in(0,1)$ and let $0<a,b<1$.  From \eqref{one1} and Lemma \ref{lemma14}, let $(A,B)\subset\R^{n}\times\R^{n}$ with $A=-A,B=-B$ such that
\begin{equation}\label{two5.0}
\int_{\R^{n}}1_{A}(x)T_{\rho}1_{B}(x)d\gamma_{n}(x)=
\inf_{\substack{\{f,g\colon\R^{n}\to[0,1],\int_{\R^{n}} f(x)d\gamma_{n}(x)=a\\ \int_{\R^{n}} g(x)d\gamma_{n}(x)=b,f(x)=f(-x)\,\forall\,x\in\R^{n}\}}}
\int_{\R^{n}}f(x)T_{\rho}g(x)d\gamma_{n}(x).
\end{equation}
Then there exist $c,c'\in\R$ such that
% if B is a ball, then T_{\rho}1_{B} is larger near zero, and A is the complement of a ball, so A=\{x\colon T_{\rho}1_{B}\leq c\}
% if B is a complement of a ball, then T_{\rho}1_{B} is smaller near zero, and A is a ball, so A=\{x\colon T_{\rho}1_{B}\leq c\}
\begin{equation}\label{two5pp}
A=\{x\in\R^{n}\colon T_{\rho}1_{B}(x)\leq c\}\,\wedge\,B=\{x\in\R^{n}\colon T_{\rho}1_{A}(x)\leq c'\}.
\end{equation}
If $\rho\in(-1,0)$, the same result holds, with the additional restriction $g(x)=g(-x)$ $\forall$ $x\in\R^{n}$ in \eqref{two5.0}.
\end{lemma}
\begin{proof}
We argue by contradiction.  Suppose there exists $x_{1},x_{2}\in\R^{n}$, $x_{1}\notin A,x_{2}\in A$ such that $T_{\rho}1_{B}(x_{1})>T_{\rho}1_{B}(x_{2})$.  Let $U_{1}\subset\R^{n}$ be a small ball around $x_{1}$ and let $U_{2}$ be a small ball around $x_{2}$ such that $T_{\rho}1_{B}(u_{1})>T_{\rho}1_{B}(u_{2})$, $\forall$ $u_{1}\in U_{1},u_{2}\in U_{2}$.  Also, assume that $U_{1}\cap U_{2}=\emptyset$ and $\gamma_{n}(U_{1})=\gamma_{n}(U_{2})$.
%Since $A=-A,B=-B$, the sets $V_{1}\colonequals U_{1}\cup(-U_{1})$ and $V_{2}\colonequals U_{2}\cup(-U_{2})$ satisfy $T_{\rho}1_{B}(v_{1})>T_{\rho}1_{B}(v_{2})$ $\forall$ $v_{1}\in V_{1},v_{2}\in V_{2}$.
Define $A'\colonequals (A\setminus U_{2})\cup U_{1}$.  Then $1_{A'}=1_{A}-1_{U_{2}}+1_{U_{1}}$, and for $U_{1},U_{2}$ sufficiently small,
\begin{flalign*}
\int_{\R^{n}}1_{A'}(x)T_{\rho}1_{B}(x)d\gamma_{n}(x)
&=\int_{\R^{n}} 1_{A}(x)T_{\rho}1_{B}(x)d\gamma_{n}(x)
+\int_{\R^{n}}(1_{U_{1}}(x)-1_{U_{2}}(x))T_{\rho}1_{B}(x)d\gamma_{n}(x)\\
&>\int_{\R^{n}} 1_{A}(x)T_{\rho}1_{B}(x)d\gamma_{n}(x).
\end{flalign*}

This inequality contradicts the maximality of $A$.  We conclude that no such $x_{1},x_{2}$ exist, so \eqref{two5pp} holds.
\end{proof}

\section{First Main Theorem}\label{secmain}

\begin{theorem}[Conjecture \ref{SGP}, $n=1$, $\rho$ small]\label{thm2.0}
Let $n=1$, $0<a,b<1$, and let $\abs{\rho}<\min(e^{-40},a^{20},(1-a)^{20},b^{20},(1-b)^{20})/1000$.  By Lemma \ref{lemma0}, let $(A,B)=(B(0,r_{a}),B(0,r_{b}')^{c})$ or let $(A,B)=(B(0,r_{a}')^{c},B(0,r_{b}))$ such that $\gamma_{1}(A)=a,\gamma_{n}(B)=b$ and such that
\begin{equation}\label{four10}
\begin{aligned}
&\left(\int_{\R}1_{A}(x)\sqrt{2!}h_{2}(x)d\gamma_{1}(x)\right)
\left(\int_{\R}1_{B}(y)\sqrt{2!}h_{2}(y)d\gamma_{1}(y)\right)\\
&\qquad=\inf_{\substack{\{f,g\colon\R\to[0,1],\\ \int_{\R} f(x)d\gamma_{1}(x)=a,\int_{\R} g(y)d\gamma_{1}(y)=b\}}}
\left(\int_{\R} f(x)\sqrt{2!}h_{2}(x)d\gamma_{1}(x)\right)
\left(\int_{\R} g(x)\sqrt{2!}h_{2}(x)d\gamma_{1}(x)\right).
\end{aligned}
\end{equation}
From Lemma \ref{lemma14}, let $A',B'\subset\R$ such that $\gamma_{1}(A')=a,\gamma_{1}(B')=b$ and such that
\begin{equation}\label{four11}
\int_{\R} 1_{A'}(x)T_{\rho}1_{B'}(x)d\gamma_{1}(x)
=\inf_{\substack{\{f,g\colon\R^{n}\to[0,1],\int_{\R} f(x)d\gamma_{1}(x)=a,\\
\int_{\R} g(x)d\gamma_{1}(x)=b,\,f(x)=f(-x),\forall\,x\in\R\}}}
\int_{\R}f(x)T_{\rho}g(x)d\gamma_{n}(x).
\end{equation}
If $\rho>0$, then $(A,B)=(A',B')$.  If $\rho<0$, the same result holds, with the additional restriction $g(x)=g(-x)$ $\forall$ $x\in\R^{n}$ in \eqref{four11}.
\end{theorem}
\begin{proof}
Without loss of generality $(A,B)=(B(0,r_{a}),B(0,r_{b}')^{c})$.

\noindent\textbf{Step 1.}  Approximating Noise Stability using second order Hermite-Fourier coefficients.

From \eqref{one1}, and using that $A'=-A'$ with \eqref{one3.1}
\begin{equation}\label{four12}
\begin{aligned}
&\left|\frac{\int_{\R^{n}}1_{A'}(x)T_{\rho}1_{B'}(x)d\gamma_{1}(x)-\gamma_{1}(A')\gamma_{1}(B')}{\rho^{2}}\right.\\
&\qquad\qquad\qquad\left.-2\int_{\R}1_{A'}(x)h_{2}(x)\sqrt{2!}d\gamma_{1}(x)\int1_{B'}(y)h_{2}(y)\sqrt{2!}d\gamma_{1}(y)\right|\\
&\qquad\qquad\leq\abs{\rho}^{2}\min(\sqrt{a},\sqrt{1-a})\min(\sqrt{b},\sqrt{1-b}).
\end{aligned}
\end{equation}
From \eqref{four12} and \eqref{four11},
\begin{flalign}
&2\int_{\R}1_{A'}(x)h_{2}(x)\sqrt{2!}d\gamma_{1}(x)\int_{\R}1_{B'}(y)h_{2}(y)\sqrt{2!}d\gamma_{1}(y)\\
&\quad\leq\frac{\int_{\R}1_{A'}(x)T_{\rho}1_{B'}(x)d\gamma_{1}(x)-\gamma_{1}(A')\gamma_{1}(B')}{\rho^{2}}
+\abs{\rho}^{2}\min(\sqrt{a},\sqrt{1-a})\min(\sqrt{b},\sqrt{1-b})\nonumber\\
&\quad\stackrel{\eqref{four11}}{=}\inf_{\substack{\{f,g\colon\R\to[0,1],\\ \int_{\R} f(x)d\gamma_{1}(x)=a,\\ \int_{\R} g(x)d\gamma_{1}(x)=b,f(x)=f(-x)\,\forall x\in\R\}}}
\frac{\int_{\R} f(x)T_{\rho}g(x)d\gamma_{1}(x)-ab}{\rho^{2}}\nonumber\\
&\qquad\qquad
+\abs{\rho}^{2}\min(\sqrt{a},\sqrt{1-a})\min(\sqrt{b},\sqrt{1-b}).\label{four12.5}
\end{flalign}
Similarly, from \eqref{four12}
\begin{flalign}
&2\int_{\R}1_{A}(x)h_{2}(x)\sqrt{2!}d\gamma_{1}(x)\int_{\R}1_{B}(x)h_{2}(x)\sqrt{2!}d\gamma_{1}(x)\\
&\quad\geq\frac{\int_{\R}1_{A}(x)T_{\rho}1_{B}(x)d\gamma_{1}(x)-\gamma_{1}(A)\gamma_{1}(B)}{\rho^{2}}
-\abs{\rho}^{2}\min(\sqrt{a},\sqrt{1-a})\min(\sqrt{b},\sqrt{1-b})\nonumber\\
&\quad\geq\inf_{\substack{\{f,g\colon\R\to[0,1],\\ \int_{\R} f(x)d\gamma_{1}(x)=a,\\ \int_{\R} g(x)d\gamma_{1}(x)=b,f(x)=f(-x)\,\forall x\in\R\}}}
\frac{\int_{\R} f(x)T_{\rho}g(x)d\gamma_{1}(x)-ab}{\rho^{2}}\nonumber\\
&\qquad\qquad
-\abs{\rho}^{2}\min(\sqrt{a},\sqrt{1-a})\min(\sqrt{b},\sqrt{1-b}).\label{four12.7}
\end{flalign}

Combining \eqref{four12.5}, \eqref{four12.7} and \eqref{four10},
\begin{equation}\label{four12.9}
\begin{aligned}
&\int_{\R}1_{A'}(x)h_{2}(x)\sqrt{2!}d\gamma_{1}\int_{\R}1_{B'}(y)h_{2}(y)\sqrt{2!}d\gamma_{1}(y)\\
&\qquad\leq\inf_{\substack{\{f,g\colon\R\to[0,1],\\ \int_{\R} f(x)d\gamma_{1}(x)=a,\int_{\R} g(y)d\gamma_{1}(y)=b\}}}
\left(\int_{\R} f(x)\sqrt{2!}h_{2}(x)d\gamma_{1}(x)\right)\left(\int_{\R} g(y)\sqrt{2!}h_{2}(y)d\gamma_{1}(y)\right)\\
&\qquad\qquad+2\abs{\rho}^{2}\min(\sqrt{a},\sqrt{1-a})\min(\sqrt{b},\sqrt{1-b}).
\end{aligned}
\end{equation}

\noindent\textbf{Step 2.}  Optimal sets are close to balls or their complement.

From \eqref{four12.9} and Lemma \ref{lemma18},
\begin{equation}\label{four13}
\int_{\R}\abs{1_{A}(x)-1_{A'}(x)}^{2}d\gamma_{1}(x)<10\abs{\rho}^{7/8}\,\wedge\,\int_{\R}\abs{1_{B}(y)-1_{B'}(y)}^{2}d\gamma_{1}(y)<10\abs{\rho}^{7/8}.
\end{equation}
Then, by the Cauchy-Schwarz inequality, for every $\ell\in\N$,
\begin{equation}\label{four14}
\begin{aligned}
\abs{\int_{\R}(1_{A}(x)-1_{A'}(x))\sqrt{\ell!}h_{\ell}(x)d\gamma_{1}(x)}&<\sqrt{10}\abs{\rho}^{7/16}\\
\abs{\int_{\R}(1_{B}(y)-1_{B'}(y))\sqrt{\ell!}h_{\ell}(y)d\gamma_{1}(y)}&<\sqrt{10}\abs{\rho}^{7/16}.
\end{aligned}
\end{equation}
%
%From \eqref{four13}, \eqref{four14}, let $(A_{(0)},B_{(0)})\equalscolon(B(0,r_{0}),B(0,r_{0}')^{c})$ such that $\int1_{A_{(0)}}h_{2}d\gamma_{1}=\int 1_{A}h_{2}d\gamma_{1}$ and $\int1_{B_{(0)}}h_{2}d\gamma_{1}=\int1_{B}h_{2}d\gamma_{1}$.  Recall that
%\begin{equation}\label{three14.5}
%\int\absf{1_{B}-1_{B_{(0)}}}d\gamma_{1}\leq10\left(\int\absf{1_{B}(x)-1_{B_{(0)}}(x)}\absf{1-x^{2}}d\gamma_{1}(x)\right)^{1/2}.
%\end{equation}
%From \eqref{three14.5}, we have $\int\absf{1_{A}-1_{A_{(0)}}}d\gamma_{1}<10^{5/4}\rho^{3/32}$ and $\int\absf{1_{B}-1_{B_{(0)}}}d\gamma_{1}<10^{5/4}\rho^{3/32}$.

\noindent\textbf{Step 3.}  Estimating $T_{\rho}1_{B'}$.

Let $g=1_{B}-1_{B'}$.  Recall that $b=2\int_{r_{b}}^{\infty}e^{-x^{2}/2}dx/\sqrt{2\pi}$.  Then $\min(b,1-b)/10\leq r_{b}\leq\sqrt{-3\log\min(b,1-b)}$.  Since $0<\abs{\rho}<\min(b,1-b)<1$, we have $-\log\abs{\rho}>-\log\min(b,1-b)$.  Let $\abs{x}\leq\sqrt{-4\log\abs{\rho}}$.  Since $\abs{\rho}<e^{-10}$,
\begin{equation}\label{fourzero}
10\abs{\rho}^{11/16}\sum_{\ell\in\N\colon\abs{\ell}\geq4}\abs{\rho}^{\abs{\ell}-3}
\abs{\ell}3^{\abs{\ell}}(-4\log\abs{\rho})^{\abs{\ell}/2}<1.
\end{equation}
By \eqref{one3.1}, $\int_{\R}g(x)h_{3}(x)d\gamma_{1}(x)=0$.  So, using Lemma \ref{lemma6},
\begin{flalign}
&\abs{T_{\rho}g(x)-\rho^{2}\frac{\sqrt{2}}{2}(x^{2}-1)\int_{\R}(1_{B}(x)-1_{B'}(x))\sqrt{2}h_{2}(x)d\gamma_{1}(x)}\\
&\qquad\stackrel{\eqref{six1.8}}{\leq}\sum_{\ell\in\N\colon\abs{\ell}\geq4}\abs{\rho}^{\abs{\ell}}
\abs{\sqrt{\ell!}h_{\ell}(x)}\abs{\int_{\R} h_{\ell}(x)\sqrt{\ell!}g(x)d\gamma_{1}(x)}\nonumber\\
&\qquad\stackrel{\mathrm{Lemma}\,\ref{lemma6}}{\leq}\abs{\rho}^{3}\sum_{\ell\in\N\colon\abs{\ell}\geq4}\abs{\rho}^{\abs{\ell}-3}
\abs{\ell}3^{\abs{\ell}}\max(1,\abs{x}^{\abs{\ell}})\abs{\int_{\R} h_{\ell}(x)\sqrt{\ell!}g(x)d\gamma_{1}(x)}\label{four8}\\
&\qquad\stackrel{\eqref{four14}}{\leq}10\abs{\rho}^{55/16}\sum_{\ell\in\N\colon\abs{\ell}\geq4}\abs{\rho}^{\abs{\ell}-3}
\abs{\ell}3^{\abs{\ell}}\max(1,\abs{x}^{\abs{\ell}})\nonumber\\
&\qquad\label{four7}\leq10\abs{\rho}^{55/16}\sum_{\ell\in\N\colon\abs{\ell}\geq4}\abs{\rho}^{\abs{\ell}-3}
\abs{\ell}3^{\abs{\ell}}(-4\log\abs{\rho})^{\abs{\ell}/2}
\stackrel{\eqref{fourzero}}{\leq}\abs{\rho}^{11/4}.
\end{flalign}

That is, for any $\abs{x}\leq\sqrt{-4\log\abs{\rho}}$,
\begin{equation}\label{four15}
\absf{T_{\rho}1_{B}(x)-T_{\rho}1_{B'}(x)-\rho^{2}\frac{\sqrt{2}}{2}(x^{2}-1)\int_{\R}(1_{B}(x)-1_{B'}(x))\sqrt{2}h_{2}(x)d\gamma_{1}(x)}
=\abs{T_{\rho}g(x)}\leq\abs{\rho}^{11/4}.
\end{equation}
%for a small, the complement of the ball is the best.  In this case we have
% a=2\int_{r}^{\infty}e^{-x^{2}/2}dx/\sqrt{2\pi}
% so, 1/(2r\sqrt{2\pi})e^{-r^{2}/2}\leq a\leq 1/(r\sqrt{2\pi})e^{-r^{2}/2}.
% so, \log(1/(2r\sqrt{2\pi}))-r^{2}/2\leq\log(a) \leq \log(1/(r\sqrt{2\pi})) -r^{2}/2
% so, -\log(2r\sqrt{2\pi})-r^{2}/2 \leq\log(a) \leq -\log(r\sqrt{2\pi})-r^{2}/2
% so, -\log(r)-\log(2\sqrt{2\pi})-r^{2}/2 \leq\log(a) \leq -\log(r) -\log(\sqrt{2\pi}) -r^{2}/2
% would be nice to have r\leq\sqrt{-2\log(a)}.  if so, r^{2}/2\leq-2\log a, so \log(a)\leq-r^{2}/2
%
%1_{[-r,r]}(x\rho+y\sqrt{1-\rho^{2}})  %% x\rho+y\sqrt{1-\rho^{2}}\in[-r,r]
%=1_{[-r-x\rho,r-x\rho]}(y\sqrt{1-\rho^{2}})   %%  y\sqrt{1-\rho^{2}}\in[-r-x\rho,r-x\rho]
%=1_{[(-r-x\rho)/\sqrt{1-\rho^{2}},(r-x\rho)/\sqrt{1-\rho^{2}}]}(y)  %%y\in(1-\rho^{2})^{-1/2}[-r-x\rho,r-x\rho]

%have a problem if B is very small, but B is never very small, since for small measure, the complement of the ball is better
% if we expand the radius r by epsilon, then we change the intergral against h_2 by an amount \int_{r}^{r+\epsilon}(1-x^{2})e^{-x^{2}}
% whereas the measure changes by an amount \int_{r}^{r+\epsilon}e^{-x^{2}}

Similarly, for any $\abs{x}\leq\sqrt{-3\log\abs{\rho}}$,
\begin{equation}\label{four15.5}
\absf{\frac{d}{dx}T_{\rho}(1_{B}-1_{B'})(x)-\rho^{2}\sqrt{2}x\int_{\R}(1_{B}(y)-1_{B'}(y))\sqrt{2}h_{2}(y)d\gamma_{1}(y)}\leq\abs{\rho}^{11/4}
\end{equation}

\noindent\textbf{Step 4.}  Finding the level sets of $T_{\rho}1_{B'}$

We now apply Lemma \ref{lemma7}.  Let $\min(b,1-b)/10\leq\abs{x}\leq\sqrt{-3\log\abs{\rho}}$.  Then, using that $\min(b,1-b)/10\leq r_{b}'\leq\sqrt{-2\log\min(b,1-b)}$,
\begin{equation}\label{four16}
\mathrm{sign}(x)\cdot\frac{d}{dx}T_{\rho}1_{B}(x)
\geq\abs{x}\rho^{2}\min(b,(1-b))/10
\geq\rho^{2}\min(b,1-b,a,1-a)^{2}/10.
\end{equation}
%distance between the functions is very small

%By Lemma \ref{lemma0}, $A=\{x\in\R\colon(d^{2}/d\rho^{2})T_{\rho}1_{A}(x)\geq c\}$.  By \eqref{three15}, $(d^{2}/d\rho^{2})T_{\rho}1_{B}(x)\geq c-\abs{\rho}^{11/10}$ for $x\in A\cap B(0,2r)$.  Then, by \eqref{three16}, $(d^{2}/d\rho^{2})T_{\rho}1_{B}(x)\geq c+\abs{\rho}^{11/10}$ for $x\in B(0,r-2\min(a^{1/2},(1-a)^{1/2}))$.  So, using \eqref{three15} again, $(d^{2}/d\rho^{2})T_{\rho}1_{A}(x)\geq c$ for $x\in B(0,r-2\min(a^{1/2},(1-a)^{1/2}))$.
%
%.... may need to change this somehow.  try again, but start with level sets of B first
Let $\min(b,1-b)/10\leq r_{0}\leq\sqrt{-3\log\abs{\rho}}$.  By \eqref{four16}, using that $\mathrm{sign}(x)\frac{d}{dx}T_{\rho}1_{B}(x)<0$ for all $x\neq0$, there is a $\lambda=\lambda(r_{0})\in\R$ such that
\begin{flalign}
&x\in B(0,r_{0})\Longrightarrow T_{\rho}1_{B}(x)\leq\lambda,\nonumber\\
&x\in B(0,\sqrt{-3\log\abs{\rho}})\setminus B(0,r_{0})\Longrightarrow T_{\rho}1_{B}(x)>\lambda,\nonumber\\
&x\in B(0,r_{0}-(1/10)\min(e^{-40},b,1-b,a,1-a))\nonumber\\
&\qquad\qquad\qquad\Longrightarrow T_{\rho}1_{B}(x)\leq\lambda-(1/100)\rho^{2}\min(e^{-40},b,1-b,a,1-a)^{3},\label{four16.5}\\
&x\in B(0,\sqrt{-3\log\abs{\rho}})\setminus B(0,r_{0}+(1/10)\min(e^{-40},b,1-b,a,1-a))\nonumber\\
&\qquad\qquad\qquad\Longrightarrow T_{\rho}1_{B}(x)>\lambda+(1/100)\rho^{2}\min(e^{-40},b,1-b,a,1-a)^{3}.\label{four16.7}
\end{flalign}

Also, we may take $\lambda$ to be a continuous, strictly increasing function of $r_{0}$.  By \eqref{four15}, \eqref{four16.5} and \eqref{four16.7}, and using $\abs{\rho}<\min(e^{-40},a^{20},(1-a)^{20},b^{20},(1-b)^{20})/1000$.
\begin{equation}\label{four16.9}
\begin{aligned}
&x\in B(0,r_{0}-(1/10)\min(e^{-40},b,1-b,a,1-a))\Longrightarrow T_{\rho}1_{B'}(x)\leq\lambda,\\
&x\in B(0,\sqrt{-3\log\abs{\rho}})\setminus B(0,r_{0}+(1/10)\min(e^{-40},b,1-b,a,1-a))\\
&\qquad\qquad\qquad\Longrightarrow T_{\rho}1_{B'}(x)>\lambda.
\end{aligned}
\end{equation}

By Lemma \ref{lemma000}, there exists $c_{1},c_{2}\in\R$ such that
\begin{equation}\label{four000}
A'=\{x\in\R\colon T_{\rho}1_{B'}(x)\leq c_{1}\}\,\wedge\,
B'=\{x\in\R\colon T_{\rho}1_{A'}(x)\leq c_{2}\}.
\end{equation}
Since $\gamma_{1}(B(0,\sqrt{-3\log\abs{\rho}})^{c})<\min(a,1-a)$, $B(0,\sqrt{-3\log\abs{\rho}})\cap A'\neq\emptyset$.  So, by \eqref{four000} there exists an $x\in B(0,\sqrt{-3\log\abs{\rho}})$ such that $T_{\rho}1_{B'}(x)\leq c_{1}$.  So, there exists $r_{0}$ such that $\lambda(r_{0})=c_{1}$.  Rewriting \eqref{four16.9},
\begin{equation}\label{four17}
\begin{aligned}
&B(0,r_{0}-(1/10)\min(e^{-40},b,1-b,a,1-a))\subset\{x\in\R\colon T_{\rho}1_{B'}(x)\leq c_{1}\}\\
&\,\wedge\,(B(0,\sqrt{-3\log\abs{\rho}})\setminus B(0,r_{0}+(1/10)\min(e^{-40},b,1-b,a,1-a)))\\
&\qquad\qquad\qquad\qquad\cap\{x\in\R\colon T_{\rho}1_{B'}(x)\leq c_{1}\}=\emptyset.
\end{aligned}
\end{equation}
%Then \eqref{four17} says
%\begin{equation}\label{four18}
%\begin{aligned}
%&B(0,r_{0}-(1/10)\min(e^{-40},a,1-a))\subset A'\\
%&\,\wedge\,[B(0,\sqrt{-3\log\abs{\rho}}\setminus B(0,r_{0}+(1/10)\min(e^{-40},a,1-a))]\cap A'=\emptyset.
%\end{aligned}
%\end{equation}

Combining \eqref{four16} and \eqref{four15.5}, we have $\frac{d}{dx}T_{\rho}1_{B'}(x)\mathrm{sign}(x)>0$ for all $x$ such that $\sqrt{-3\log\abs{\rho}}\geq\abs{x}\geq\min(b,1-b,a,1-a)/10$.  Using this fact and \eqref{four17}, there exists $\min(b,1-b)/10\leq r_{1}\leq r_{b}'$ such that
\begin{equation}\label{four18.5}
\begin{aligned}
&B(0,r_{1})\subset\{x\in\R\colon T_{\rho}1_{B'}(x)\leq c_{1}\}\\
&\qquad\,\wedge\,[B(0,\sqrt{-3\log\abs{\rho}})\setminus B(0,r_{1})]\cap\{x\in\R\colon T_{\rho}1_{B'}(x)\leq c_{1}\}=\emptyset.
\end{aligned}
\end{equation}

Repeating the above implications with the roles of $A'$ and $B'$ reversed, there exists $\min(a,1-a,b,1-b)/10\leq r_{2}\leq r_{a}$ such that
\begin{equation}\label{four18.8}
\begin{aligned}
&B(0,r_{2})\cap\{x\in\R\colon T_{\rho}1_{A'}(x)\leq c_{2}\}=\emptyset\\
&\qquad\,\wedge\,[B(0,\sqrt{-3\log\abs{\rho}})\setminus B(0,r_{2})]\subset\{x\in\R\colon T_{\rho}1_{A'}(x)\leq c_{2}\}.
\end{aligned}
\end{equation}

\noindent\textbf{Step 5.}  A final iterative argument to eliminate points far from the origin.

We now construct an iteration.  Let $k\in\N$.  It is given that
\begin{equation}\label{four20}
\begin{aligned}
&B(0,r_{1})\subset\{x\in\R\colon T_{\rho}1_{B'}(x)\leq c_{1}\}\\
&\qquad\,\wedge\,[B(0,\sqrt{-(k+2)\log\abs{\rho}})\setminus B(0,r_{1})]\cap\{x\in\R\colon T_{\rho}1_{B'}(x)\leq c_{1}\}=\emptyset,\\
&B(0,r_{2})\cap\{x\in\R\colon T_{\rho}1_{A'}(x)\leq c_{2}\}=\emptyset\\
&\qquad\,\wedge\,[B(0,\sqrt{-(k+2)\log\abs{\rho}})\setminus B(0,r_{2})]\subset\{x\in\R\colon T_{\rho}1_{A'}(x)\leq c_{2}\}.
\end{aligned}
\end{equation}
We then conclude that
\begin{equation}\label{four21}
\begin{aligned}
&B(0,r_{1})\subset\{x\in\R\colon T_{\rho}1_{B'}(x)\leq c_{1}\}\\
&\qquad\,\wedge\,[B(0,\sqrt{-(k+3)\log\abs{\rho}})\setminus B(0,r_{1})]\cap\{x\in\R\colon T_{\rho}1_{B'}(x)\leq c_{1}\}=\emptyset,\\
&B(0,r_{2})\cap\{x\in\R\colon T_{\rho}1_{A'}(x)\leq c_{2}\}=\emptyset\\
&\qquad\,\wedge\,[B(0,\sqrt{-(k+3)\log\abs{\rho}})\setminus B(0,r_{2})]\subset\{x\in\R\colon T_{\rho}1_{A'}(x)\leq c_{2}\}.
\end{aligned}
\end{equation}
(Note that $k=2$ for \eqref{four20} is exactly \eqref{four18.5} and \eqref{four18.8}.)

Let $x$ with $\sqrt{-(k+2)\log\abs{\rho}}\leq\abs{x}\leq\sqrt{-(k+3)\log\abs{\rho}}$.  From Lemma \ref{lemma7} and \eqref{four20},
\begin{flalign*}
\frac{\sqrt{1-\rho^{2}}}{\abs{\rho}}\frac{d}{dx}T_{\rho}1_{A'}(x)
&=\frac{\sqrt{1-\rho^{2}}}{\abs{\rho}}\frac{d}{dx}T_{\rho}1_{B(0,r_{1})}(x)
+\frac{\sqrt{1-\rho^{2}}}{\abs{\rho}}\frac{d}{dx}T_{\rho}(1_{A'}-1_{B(0,r_{1})})(x)\\
&\leq\int_{-(r_{1}+x\abs{\rho})/\sqrt{1-\rho^{2}}}^{-(r_{1}-x\abs{\rho})/\sqrt{1-\rho^{2}}}yd\gamma_{1}(y)
+\int_{-(k+2)(1-\abs{\rho})\log\abs{\rho}}^{\infty}yd\gamma_{1}(y)\\
&\leq\int_{-(r_{1}-3\abs{\rho}\log\abs{\rho})/\sqrt{1-\rho^{2}}}^{-r_{1}/\sqrt{1-\rho^{2}}}yd\gamma_{1}(y)+\abs{\rho}^{(k+2)^{2}(1-\rho)^{2}/2}\\
&\leq\frac{30\abs{\rho}}{30\abs{\rho}}\int_{-(r_{1}+30\abs{\rho})/\sqrt{1-\rho^{2}}}^{-r_{1}/\sqrt{1-\rho^{2}}}yd\gamma_{1}(y)+\abs{\rho}^{(k+2)^{2}(1-\rho)^{2}/2}\\
&\leq-30\abs{\rho}\frac{1}{30\sqrt{2\pi}}e^{-(r_{1}+30\abs{\rho})^{2}/[2(1-\rho^{2})]}+\abs{\rho}^{(k+2)^{2}(1-\rho)^{2}/2}\\
&\leq-\abs{\rho}e^{-(11/10)r_{1}^{2}/2}+\abs{\rho}^{(k+2)^{2}(1-\rho)^{2}/2}<0.
\end{flalign*}

%here is the same thing for B'
%\begin{flalign*}
%\frac{\sqrt{1-\rho^{2}}}{\abs{\rho}}\frac{d}{dx}T_{\rho}1_{B'}(x)
%&=\frac{\sqrt{1-\rho^{2}}}{\abs{\rho}}\frac{d}{dx}T_{\rho}1_{B(0,r_{2})^{c}}(x)
%+\frac{\sqrt{1-\rho^{2}}}{\abs{\rho}}\frac{d}{dx}T_{\rho}(1_{B'}-1_{B(0,r_{2})^{c}})(x)\\
%&\geq\int_{(r_{2}-x\abs{\rho})/\sqrt{1-\rho^{2}}}^{(r_{2}+x\abs{\rho})/\sqrt{1-\rho^{2}}}yd\gamma_{1}(y)
%-\int_{-(k+2)(1-\abs{\rho})\log\abs{\rho}}^{\infty}yd\gamma_{1}(y)\\
%&\geq\int_{r_{2}/\sqrt{1-\rho^{2}}}^{(r_{2}+x\abs{\rho})/\sqrt{1-\rho^{2}}}yd\gamma_{1}(y)+\abs{\rho}^{(k+2)^{2}(1-\rho)^{2}/2}\\
%&\geq\frac{30\abs{\rho}}{30\abs{\rho}}\int_{r_{2}/\sqrt{1-\rho^{2}}}^{(r_{2}+30\abs{\rho})/\sqrt{1-\rho^{2}}}yd\gamma_{1}(y)+\abs{\rho}^{(k+2)^{2}(1-\rho)^{2}/2}\\
%&\geq30\abs{\rho}\frac{1}{30\sqrt{2\pi}}e^{-(r_{2}+30\abs{\rho})^{2}/[2(1-\rho^{2})]}+\abs{\rho}^{(k+2)^{2}(1-\rho)^{2}/2}\\
%&\geq\abs{\rho}e^{-(11/10)r_{2}^{2}/2}+\abs{\rho}^{(k+2)^{2}(1-\rho)^{2}/2}>0.
%\end{flalign*}

Similarly, $(\sqrt{1-\rho^{2}}/\abs{\rho})\frac{d}{dx}T_{\rho}1_{B'}(x)>0$.  Therefore, \eqref{four20} implies that \eqref{four21} holds.  So, let $k\to\infty$ in \eqref{four21}.  Combining \eqref{four21} and \eqref{four000} then completes the theorem.
$$
B(0,r_{2})=\{x\in\R\colon T_{\rho}1_{A'}\leq c_{1}\}=B'\,\wedge\,
B(0,r_{1})=\{x\in\R\colon T_{\rho}1_{B'}\leq c_{1}\}=A'.
$$
\end{proof}

\section{A Second Variation Formula}\label{secvars}

In preparation for later sections, we now investigate a second variation formula for quadratic functionals.  Lemma \ref{lemma2} below essentially appears in \cite[Theorem 2.6]{chokski07}.  However, their statement and proof are slightly different than we require.  We prove Lemmas \ref{lemma2} and Lemma \ref{lemma3} in the Appendix, Section \ref{sec2var} (see Lemmas \ref{latelemma} and \ref{lemmab}.)

\begin{assumption}\label{as1}
Let $A\subset\R^{n}$ be a set with smooth boundary, and let $N\colon\partial A\to S^{n-1}$ denote the unit exterior normal to $\partial A$.  Let $X\colon\R^{n}\to\R^{n}$ be a vector field.  Let $\Psi\colon\R^{n}\times(-1,1)$ such that $\Psi(x,0)=x$ and such that $\frac{d}{dt}|_{t=0}\Psi(x,t)=X(\Psi(x,t))$ for all $x\in\R^{n},t\in(-1,1)$.  For any $t\in(-1,1)$, let $A^{(t)}=\Psi(A,t)$.    Note that $A^{(0)}=A$.  Define
$$V(x,t)\colonequals\int_{A^{(t)}}G(x,y)dy,\qquad \forall\,x\in\R^{n},\,\forall\,t\in(-1,1).$$
\end{assumption}

\begin{lemma}[\embolden{The Second Variation}, {\cite[Theorem 2.6]{chokski07}}]\label{lemma2}

Let $G\colon\R^{n}\times\R^{n}\to\R$ be a Schwartz function.  For any $A\subset\R^{n}$, let $F(A)\colonequals \int_{\R^{n}}\int_{\R^{n}} 1_{A}(x)G(x,y)1_{A}(y)dxdy$.  Then
%For any $x\in\partial A$, let $N(x)$ denote the outward pointing unit normal vector of $A$ at $x$.  For any $t\in(-1,1)$, let $A^{(t)}\subset\R^{n}$ so that $A^{(0)}=A$.  .... Then a normal variation of $A$ satisfies

\begin{flalign*}
\frac{1}{2}\frac{d^{2}}{dt^{2}}F(A^{(t)})|_{t=0}
&=\int_{\partial A}\int_{\partial A}G(x,y)\langle X(x),N(x)\rangle\langle X(y),N(y)\rangle dxdy\\
&\qquad+\int_{\partial A}\mathrm{div}(V(x,0)X(x))\langle X(x),N(x)\rangle dx.
\end{flalign*}
\end{lemma}
\begin{lemma}[Variation of Gaussian Measure]\label{lemma3}
$$\frac{d}{dt}|_{t=0}\gamma_{n}(A^{(t)})=\int_{\partial A}\langle X(x),N(x)\rangle \gamma_{n}(x)dx.$$
$$\frac{d^{2}}{dt^{2}}|_{t=0}\gamma_{n}(A^{(t)})=\int_{\partial A}(\mathrm{div}(X(x))-\langle X(x),x\rangle)\langle X(x),N(x)\rangle \gamma_{n}(x)dx.$$

%We take $G(x,y)=e^{-(x^{2}-2\rho xy+y^{2})/[2(1-\rho^{2})]}$ or $G(x,y)=\langle x,y\rangle e^{-(x^{2}+y^{2})/2}$.
%...
\end{lemma}

\subsection{Noise Stability}

As our first application of Lemma \ref{lemma2}, we study the second variation of the noise stability.

For any $x,y\in\R^{n}$, let $G(x,y)=\frac{e^{-\vnorm{\rho x-y}^{2}/[2(1-\rho^{2})]}\gamma_{n}(x)}{(1-\rho^{2})^{n/2}(2\pi)^{n/2}}=\frac{e^{\frac{-\vnorm{x}_{2}^{2}-\vnorm{y}_{2}^{2}+2\rho\langle x,y\rangle}{2(1-\rho^{2})}}}{(1-\rho^{2})^{n/2}(2\pi)^{n}}$.  Suppose Assumption \ref{as1} holds.  For any $A\subset\R^{n}$, for any $\rho\in(-1,1)$, define
$$F_{\rho}(A)\colonequals\int_{\R^{n}}\int_{\R^{n}} 1_{A}(x)G(x,y)1_{A}(y)dxdy
=\int_{\R^{n}}\int_{\R^{n}} 1_{A}(x)T_{\rho}1_{A}(x)d\gamma_{n}(x).
$$
Then, for any $t\in(-1,1)$ and for any $x\in\R^{n}$, using Assumption \ref{as1} we have
\begin{equation}\label{Afour1}
V(x,t)\colonequals\int_{A^{(t)}}G(x,y)dy=T_{\rho}1_{A^{(t)}}(x)\gamma_{n}(x).
\end{equation}
\begin{lemma}[\embolden{Second Variation of Noise Stability}]\label{lemma60}
Let $\rho\in(-1,1)$.  Assume Assumption \ref{as1} holds with $\frac{d^{2}}{dt^{2}}|_{t=0}\gamma_{n}(A^{(t)})=0$.  Assume also that $T_{\rho}1_{A}(x)$ is constant for all $x\in\partial A$.  Then
\begin{equation}\label{newfive}
\begin{aligned}
\frac{1}{2}\frac{d^{2}}{dt^{2}}F(A^{(t)})|_{t=0}
&=\int_{\partial A}\int_{\partial A}G(x,y)\langle X(x),N(x)\rangle\langle X(y),N(y)\rangle dxdy\\
&\quad+\int_{\partial A}\langle\nabla T_{\rho}1_{A}(x),X(x)\rangle\langle X(x),N(x)\rangle dx.
\end{aligned}
\end{equation}
\end{lemma}
\begin{proof}
Applying Lemma \ref{lemma2},
\begin{equation}\label{Afour2}
\begin{aligned}
\frac{1}{2}\frac{d^{2}}{dt^{2}}F(A^{(t)})|_{t=0}
&=\int_{\partial A}\int_{\partial A}G(x,y)\langle X(x),N(x)\rangle\langle X(y),N(y)\rangle dxdy\\
&\qquad+\int_{\partial A}\mathrm{div}(V(x,0)X(x))\langle X(x),N(x)\rangle dx\\
&\stackrel{\eqref{Afour1}}{=}\int_{\partial A}\int_{\partial A}G(x,y)\langle X(x),N(x)\rangle\langle X(y),N(y)\rangle dxdy\\
&\qquad+\int_{\partial A}\mathrm{div}(T_{\rho}1_{A}(x)\gamma_{n}(x)X(x))\langle X(x),N(x)\rangle dx\\
&=\int_{\partial A}\int_{\partial A}G(x,y)\langle X(x),N(x)\rangle\langle X(y),N(y)\rangle dxdy\\
&\qquad+\int_{\partial A}(\sum_{i=1}^{n}T_{\rho}1_{A}(x)\frac{\partial}{\partial x_{i}}X^{(i)}(x)-x_{i}T_{\rho}1_{A}(x)X^{(i)}(x)\\
&\qquad\qquad+\frac{\partial}{\partial x_{i}}T_{\rho}1_{A}(x)X^{(i)}(x))\langle X(x),N(x)\rangle \gamma_{n}(x)dx.
\end{aligned}
\end{equation}

Using Lemma \ref{lemma3}, and that $T_{\rho}1_{A}(x)$ is constant when $x\in\partial A$, we then get \eqref{newfive}.
\end{proof}
%
%If the first variation is zero (with respect to a variation that ``preserves flatness''), then from above, $\int_{\partial A}T_{\rho}1_{A}(y)\langle X(y),N(y)\rangle d\gamma_{n}(y)=0$.
%
%$A$ or $A$ and $B$?
%
%Consider the case $n=1$.

\subsection{A Poincar\'{e}-Type Inequality}

In our investigation of the second variation of the ball or its complement in Section \ref{secfourvar} below, we require the following Poincar\'{e}-type inequality.

\begin{lemma}[A Poincar\'{e}-Type Inequality on the Sphere]\label{lemma50}
Let $r>0$ and let $B(0,r)\subset\R^{n}$.  Let $f\colon \partial B(0,r)\to\R$ with $\int_{\partial B(0,r)}f(x)dx=0$.
%and $\int_{\partial B(0,r)}x_{i}x_{j}f(x)dx=0$ for any $i,j\in\{1,\ldots,n\}$ with $i\neq j$.
Then
\begin{equation}\label{ten4}
\sum_{i=1}^{n}\left(\int_{\partial B(0,r)}x_{i}^{2}f(x)dx\right)^{2}\leq \frac{2r^{n+3}\mathrm{Vol}(S^{n-1})}{n(n+2)}\int_{\partial B(0,r)}(f(x))^{2}dx.
\end{equation}Moreover, equality occurs when $f(x)=(n-1)x_{1}^{2}-\sum_{j=2}^{n}x_{j}^{2}$ for any $x\in \partial B(0,r)$.
\end{lemma}
\begin{proof}
For any $n\geq1$, write $\mathrm{Vol}(S^{n-1}) = 2\pi\prod_{\ell=1}^{n-2}\int_{0}^{\pi} \sin^{\ell} (x)dx$.  We first claim
\begin{equation}\label{newthree}
\int_{\partial B(0,r)}x_{1}^{4}dx=\frac{3r^{n+3}\mathrm{Vol}(S^{n-1})}{n(n+2)},
\qquad \int_{\partial B(0,r)}x_{1}^{2}x_{2}^{2}dx=\frac{r^{n+3}\mathrm{Vol}(S^{n-1})}{n(n+2)}.
\end{equation}
Indeed, using (hyper)-spherical coordinates,

\begin{flalign*}
&\int_{\partial B(0,r)}x_{1}^{4}dx
=\int_{0}^{\pi}\cdots\int_{0}^{\pi}\int_{0}^{2\pi}r^{4}\cos^{4}(\phi_{n-1})
\prod_{i=1}^{n-2}\sin^{4}\phi_{i}\prod_{j=1}^{n-2}\sin^{n-1-j}(\phi_{j})r^{n-1}d\phi_{n-1}\cdots d\phi_{1}\\
&\quad=r^{n+3}\int_{0}^{\pi}\cdots\int_{0}^{\pi}\int_{0}^{2\pi}\frac{1}{4}(1+\cos(2\phi_{n-1}))^{2}
\prod_{j=1}^{n-2}\sin^{n+3-j}(\phi_{j})d\phi_{n-1}\cdots d\phi_{1}\\
&\quad=r^{n+3}\int_{0}^{\pi}\cdots\int_{0}^{\pi}\int_{0}^{2\pi}\frac{3}{8}
\prod_{j=1}^{n-2}\sin^{n+3-j}(\phi_{j})d\phi_{n-1}\cdots d\phi_{1}
=\frac{3\cdot 2\pi}{8} r^{n+3}\prod_{\ell=5}^{n+2}\int_{0}^{\pi}\sin^{\ell}(x)dx\\
%&=r^{n+1}\int_{0}^{\pi}\cdots\int_{0}^{\pi}\int_{0}^{2\pi}
%\prod_{j=0}^{n-2}\sin^{n-j}(\phi_{j})d\phi_{n-1}\cdots d\phi_{1}\\
&\quad=\frac{3}{8}r^{n+3}\mathrm{Vol}(S^{n-1})\frac{\int_{0}^{\pi}\sin^{n+2}(x)dx\int_{0}^{\pi}\sin^{n+1}(x)dx\int_{0}^{\pi}\sin^{n}(x)dx\int_{0}^{\pi}\sin^{n-1}(x)dx}
{\int_{0}^{\pi}\sin(x)dx\int_{0}^{\pi}\sin^{2}(x)dx\int_{0}^{\pi}\sin^{3}(x)dx\int_{0}^{\pi}\sin^{4}(x)dx}\\
&\quad\stackrel{\eqref{five20}}{=}\frac{3}{8}r^{n+3}\mathrm{Vol}(S^{n-1})\frac{(2\pi)^{2}\frac{(n+1)!!}{(n+2)!!}\frac{(n)!!}{(n+1)!!}\frac{(n-1)!!}{n!!}\frac{(n-2)!!}{(n-1)!!}}
{2(\pi/2)(4/3)(3\pi/8)}
=\frac{3r^{n+3}\mathrm{Vol}(S^{n-1})}{n(n+2)}.
\end{flalign*}
A similar calculation proves the other part of \eqref{newthree}.

Now, for any $i\in\{1,\ldots,n\}$, let $g_{i}\colon\R^{n}\to\R$ so that $g_{i}(x)=(n-1)x_{i}^{2}-\sum_{j\neq i}x_{j}^{2}$.  It then follows from \eqref{newthree} that
\begin{equation}\label{newten}
\int_{\partial B(0,r)}(g_{i}(x))^{2}dx
=n(n-1)\left(\int_{\partial B(0,r)}x_{1}^{4}dx-\int_{\partial B(0,r)}x_{1}^{2}x_{2}^{2}dx\right)
=\frac{2(n-1)}{n+2}r^{n+3}\mathrm{Vol}(S^{n-1}).
\end{equation}
And if $j\in\{1,\ldots,n\}$ with $j\neq i$, then
\begin{equation}\label{neweleven}
\int_{\partial B(0,r)}g_{i}(x)g_{j}(x)dx
=n\left(-\int_{\partial B(0,r)}x_{1}^{4}dx+\int_{\partial B(0,r)}x_{1}^{2}x_{2}^{2}dx\right)
=-\frac{2}{n+2}r^{n+3}\mathrm{Vol}(S^{n-1}).
\end{equation}
So, the functions $g_{1},\ldots,g_{n}$ span an $(n-1)$-dimensional vector space, $\sum_{i=1}^{n}g_{i}=0$, and $\langle g_{i},g_{j}\rangle\colonequals\int_{\partial B(0,r)}g_{i}(x)g_{j}(x)dx=-2(n-1)\frac{r^{n+3}\mathrm{Vol}(S^{n-1})}{n(n+2)}$ if $i,j\in\{1,\ldots,n\}$ with $i\neq j$.

%
%\begin{flalign*}
%&\int_{\partial B(0,r)}x_{1}^{2}x_{2}^{2}dx\\
%&=\int_{0}^{\pi}\cdots\int_{0}^{\pi}\int_{0}^{2\pi}r^{4}\cos^{2}(\phi_{n-1})\sin^{2}(\phi_{n-1})
%\prod_{i=1}^{n-2}\sin^{4}\phi_{i}\prod_{j=1}^{n-2}\sin^{n-1-j}(\phi_{j})r^{n-1}d\phi_{n-1}\cdots d\phi_{1}\\
%&=r^{n+3}\int_{0}^{\pi}\cdots\int_{0}^{\pi}\int_{0}^{2\pi}(\cos^{2}(\phi_{n-1})-\cos^{4}(\phi_{n-1}))
%\prod_{j=1}^{n-2}\sin^{n+3-j}(\phi_{j})d\phi_{n-1}\cdots d\phi_{1}\\
%&=r^{n+3}\int_{0}^{\pi}\cdots\int_{0}^{\pi}\int_{0}^{2\pi}\frac{1}{8}
%\prod_{j=1}^{n-2}\sin^{n+3-j}(\phi_{j})d\phi_{n-1}\cdots d\phi_{1}
%=\frac{2\pi}{8} r^{n+3}\prod_{\ell=5}^{n+2}\int_{0}^{\pi}\sin^{\ell}(x)dx\\
%%&=r^{n+1}\int_{0}^{\pi}\cdots\int_{0}^{\pi}\int_{0}^{2\pi}
%%\prod_{j=0}^{n-2}\sin^{n-j}(\phi_{j})d\phi_{n-1}\cdots d\phi_{1}\\
%&=\frac{1}{8}r^{n+3}\mathrm{Vol}(S^{n-1})\frac{\int_{0}^{\pi}\sin^{n+2}(x)dx\int_{0}^{\pi}\sin^{n+1}(x)dx\int_{0}^{\pi}\sin^{n}(x)dx\int_{0}^{\pi}\sin^{n-1}(x)dx}
%{\int_{0}^{\pi}\sin(x)dx\int_{0}^{\pi}\sin^{2}(x)dx\int_{0}^{\pi}\sin^{3}(x)dx\int_{0}^{\pi}\sin^{4}(x)dx}\\
%&\stackrel{\eqref{five20}}{=}\frac{1}{8}r^{n+3}\mathrm{Vol}(S^{n-1})\frac{(2\pi)^{2}\frac{(n+1)!!}{(n+2)!!}\frac{(n)!!}{(n+1)!!}\frac{(n-1)!!}{n!!}\frac{(n-2)!!}{(n-1)!!}}
%{2(\pi/2)(4/3)(3\pi/8)}
%=\frac{r^{n+3}\mathrm{Vol}(S^{n-1})}{n(n+2)}.
%\end{flalign*}
%
%Vol(S^n-1) = 2\pi\prod_{i=1}^{n-2}\int \sin^i (x)
% so 2\pi \prod_{i=3}^{n}\int \sin^i (x) = 2\pi \prod_{i=1}^{n-2}\int \sin^i (x)   [i=n-1 term, i=n term / i=1 term i=2 term]

We now proceed to prove \eqref{ten4}.  Using $\int_{\partial B(0,r)}f(x)dx=0$, we have
\begin{flalign*}
\sum_{i=1}^{n}\left(\int_{\partial B(0,r)}x_{i}^{2}f(x)dx\right)^{2}
&=\sum_{i=1}^{n}\left(\int_{\partial B(0,r)}\frac{1}{n}[(x_{1}^{2}+\cdots+x_{n}^{2})+\sum_{j\neq i}(x_{i}^{2}-x_{j}^{2})]f(x)dx\right)^{2}\\
&=\sum_{i=1}^{n}\frac{1}{n^{2}}\left(\int_{\partial B(0,r)}g_{i}(x)f(x)dx\right)^{2}.\\
\end{flalign*}
So, to prove \eqref{ten4}, we can equivalently prove that
\begin{equation}\label{newseven}
\sum_{i=1}^{n}\left(\int_{\partial B(0,r)}g_{i}(x)f(x)dx\right)^{2}
\leq\frac{2n}{n+2}r^{n+3}\mathrm{Vol}(S^{n-1})\int_{\partial B(0,r)}(f(x))^{2}dx.
\end{equation}

Since the polynomials $g_{1},\ldots,g_{n}$ are polynomials which are homogeneous of degree $2$, by expanding $f$ in spherical harmonics, it suffices to assume that $f$ is also a polynomial which is homogeneous of degree $2$.  Then, the left side of \eqref{newseven} is the sum of the squared lengths of the projections of $f$ onto $g_{1},\ldots,g_{n}$.  Since $\sum_{i=1}^{n}g_{i}=0$, $g_{1},\ldots,g_{n}$ span an $(n-1)$-dimensional space, and $\langle g_{i},g_{j}\rangle$ is a constant for any $i,j\in\{1,\ldots,n\}$ with $i\neq j$, we conclude that the left side of \eqref{newseven} is bounded by a multiple of the right side.  More specifically, if $\int_{\partial B(0,r)}(f(x))^{2}dx$ is fixed, then the left side of \eqref{newseven} is maximized whenever $f$ is in the span of $g_{1},\ldots,g_{n}$.  In particular, the left side of \eqref{newseven} is maximized when $f$ is a multiple of $f_{1}$.  And indeed, equality holds in \eqref{newseven} in this case, since if $f=g_{1}$, we have by \eqref{newten} and \eqref{neweleven},

\begin{flalign*}
&\sum_{i=1}^{n}\left(\int_{\partial B(0,r)}g_{i}(x)f(x)dx\right)^{2}
=\langle g_{1},g_{1}\rangle^{2}+(n-1)\langle g_{1},g_{2}\rangle^{2}\\
&\quad=4(n-1)^{2}\frac{r^{2n+6}\mathrm{Vol}(S^{n-1})^{2}}{(n+2)^{2}}
+4(n-1)\frac{r^{2n+6}\mathrm{Vol}(S^{n-1})^{2}}{(n+2)^{2}}
=\frac{4n(n-1)}{(n+2)^{2}}r^{2n+6}\mathrm{Vol}(S^{n-1})^{2}.
\end{flalign*}

And by \eqref{newten},
$$
\sum_{i=1}^{n}\int_{\partial B(0,r)}(f(x))^{2}dx
=\frac{2(n-1)}{n+2}r^{n+3}\mathrm{Vol}(S^{n-1}).
$$

So,
$$
\frac{\sum_{i=1}^{n}\left(\int_{\partial B(0,r)}g_{i}(x)f(x)dx\right)^{2}}{\int_{\partial B(0,r)}(f(x))^{2}dx}
=\frac{\frac{4n(n-1)}{(n+2)^{2}}r^{2n+6}\mathrm{Vol}(S^{n-1})^{2}}{\frac{2(n-1)}{n+2}r^{n+3}\mathrm{Vol}(S^{n-1})}
=\frac{2n}{n+2}r^{n+3}\mathrm{Vol}(S^{n-1}).
$$
That is, \eqref{newseven} holds, and the proof is complete.
\end{proof}

\begin{lemma}
Let $k$ be a positive integer.  Then
\begin{equation}\label{five20}
\int_{0}^{\pi}\sin^{k}(\theta)d\theta
=\begin{cases}\pi\frac{(k-1)!!}{(k)!!}&\mbox{, if $k$ is even}\\
2\frac{(k-1)!!}{(k)!!}&\mbox{, if $k$ is odd.}
\end{cases}
\end{equation}
\end{lemma}
\begin{proof}
Let $k\geq0$.  Let $c_{k}\colonequals\int_{0}^{\pi}\sin^{k}(\theta)d\theta$.  Then
\begin{flalign*}
c_{k}&=-\int_{0}^{\pi}\sin^{k-1}(\theta)\frac{d}{d\theta}\cos(\theta) d\theta
=\int_{0}^{\pi}(k-1)\sin^{k-2}(\theta)\cos^{2}(\theta) d\theta\\
&=\int_{0}^{\pi}(2k-1)\sin^{k-2}(\theta)(1-\sin^{2}(\theta)) d\theta
=(k-1)(c_{k-2}-c_{k}).
\end{flalign*}
So, if $d_{k}=c_{0}=\pi$ when $k$ is even, and $d_{k}=c_{1}=2$ when $k$ is odd,
$$
c_{k}=\frac{k-1}{k}c_{k-2}=\frac{(k-1)(k-2)}{k(k-2)}c_{k-4}
=\cdots=\frac{(k-1)!!}{(k)!!}d_{k}
$$
\end{proof}

\begin{lemma}\label{newlem}
Let $r>0$.  Then
\begin{equation}\label{neweight}
\int_{B(0,r)}(1-x_{1}^{2})d\gamma_{n}(x)=\frac{\mathrm{Vol}(S^{n-1})r^{n} e^{-r^2/2}}{n(2\pi)^{n/2}}.
\end{equation}
\end{lemma}
\begin{proof}
For $\alpha>0$, define
\begin{flalign*}
g(\alpha)\colonequals\int_{B(0,\sqrt{\alpha}r)}d\gamma_{n}(x)
&=\mathrm{Vol}(S^{n-1})\int_{0}^{\sqrt{\alpha}r}s^{n-1}e^{-s^{2}/2}ds/(2\pi)^{n/2}\\
&=\mathrm{Vol}(S^{n-1})\alpha^{n/2}\int_{0}^{r}s^{n-1}e^{-\alpha s^{2}/2}ds/(2\pi)^{n/2}.
\end{flalign*}
Then, using the product rule,
 %\sqrt{\alpha}y=r, dr=\sqrt{\alpha}dy
$$
g'(\alpha)
=\mathrm{Vol}(S^{n-1})\left(\alpha^{\frac{n}{2}}\int_{0}^{r}(-s^{n+1}/2)e^{-\alpha s^{2}/2}\frac{ds}{(2\pi)^{n/2}}
+\frac{n}{2}\alpha^{\frac{n}{2}-1}\int_{0}^{r}s^{n-1}e^{-\alpha s^{2}/2}\frac{ds}{(2\pi)^{n/2}}\right).
$$
Plugging in $\alpha=1$, we then get
$$g'(1)=\mathrm{Vol}(S^{n-1})(1/2)\int_{0}^{r}(n-s^{2})s^{n-1}e^{-s^{2}/2}ds/(2\pi)^{n/2}=\frac{1}{2}\int_{B(0,r)}(n-\vnorm{x}_{2}^{2})d\gamma_{n}(x).$$
Also, applying the Fundamental Theorem of Calculus to the definition of $g$,
$$g'(1)=(1/2)\mathrm{Vol}(S^{n-1})r^{n}e^{-r^{2}/2}/(2\pi)^{n/2}.$$
Combining the two formulas for $g'(1)$, we get
$$\int_{B(0,r)}(n-\vnorm{x}_{2}^{2})d\gamma_{n}(x)=2g'(1)=\mathrm{Vol}(S^{n-1})r^{n}e^{-r^{2}/2}/(2\pi)^{n/2}.$$
That is, \eqref{neweight} holds.
\end{proof}

\subsection{Sum of Squared Fourier Coefficients}\label{secfourvar}

%Taking two derivatives in $\rho$ and setting $\rho=0$...

For our second application of Lemma \ref{lemma2}, we consider the sum of squared second-order Hermite-Fourier coefficients of a set.

\begin{lemma}[\embolden{Second Variation of Sum of Squared Fourier Coefficients}]\label{lemma61}
 Suppose Assumption \ref{as1} holds with $\frac{d^{2}}{dt^{2}}|_{t=0}\gamma_{n}(A^{(t)})=0$.  For any $A\subset\R^{n}$, define
$$F(A)\colonequals\sum_{i=1}^{n}\int_{\R^{n}}\int_{\R^{n}} 1_{A}(x)(1-x_{i}^{2})(1-y_{i}^{2})\gamma_{n}(x)\gamma_{n}(y)1_{A}(y)dxdy.$$
Assume there exists $r>0$ such that $A=B(0,r)$ or $A=B(0,r)^{c}$.  Then
\begin{equation}\label{newsix}
\begin{aligned}
\frac{1}{2}\frac{d^{2}}{dt^{2}}F(A^{(t)})|_{t=0}
&=\sum_{i=1}^{n}\left(\int_{\partial A}(1-x_{i}^{2})\langle X(x),N(x)\rangle\gamma_{n}(x)dx\right)^{2}\\
%&\qquad+\sum_{i=1}^{n}(\int_{A} (1-y_{i}^{2})d\gamma_{n}(y))\int_{\partial A}(1-x_{i}^{2})\langle-x,X(x)\rangle\langle X(x),N(x)\rangle \gamma_{n}(x)dx\\
&\qquad+2\sum_{i=1}^{n}\int_{\partial A}(-x_{i})X^{(i)}\langle X(x),N(x)\rangle d\gamma_{n}(x)(\int_{A} (1-y_{i}^{2})d\gamma_{n}(y)).
\end{aligned}
\end{equation}
\end{lemma}
\begin{proof}
Let $i\in\{1,\ldots,n\}$ and let $G_{i}(x,y)=(1-x_{i}^{2})(1-y_{i}^{2})\gamma_{n}(x)\gamma_{n}(y)$ for any $x,y\in\R^{n}$.  Define
$$F_{i}(A^{(t)})\colonequals\int_{\R^{n}}\int_{\R^{n}} 1_{A^{(t)}}(x)G_{i}(x,y)1_{A^{(t)}}(y)=(\int_{A^{(t)}} (1-x_{i}^{2})d\gamma_{n}(x))^{2}.$$
Then $F(A^{(t)})=\sum_{i=1}^{n}F_{i}(A^{(t)})$.  Define
\begin{equation}\label{Athree1}
V_{i}(x,t)\colonequals\int_{A^{(t)}}G_{i}(x,y)dy=(\int_{A^{(t)}}(1-y_{i}^{2})d\gamma_{n}(y))(1-x_{i}^{2})\gamma_{n}(x).
\end{equation}

Applying Lemma \ref{lemma2},
\begin{equation}\label{Athree2}
\begin{aligned}
\frac{1}{2}\frac{d^{2}}{dt^{2}}F_{i}(A^{(t)})|_{t=0}
&=\int_{\partial A}\int_{\partial A}(1-x_{i}^{2})(1-y_{i}^{2})\langle X(x),N(x)\rangle\langle X(y),N(y)\rangle \gamma_{n}(x)\gamma_{n}(y)dxdy\\
&\qquad+\int_{\partial A}\mathrm{div}(V_{i}(x,0)X(x))\langle X(x),N(x)\rangle dx\\
&\stackrel{\eqref{Athree1}}{=}(\int_{\partial A}(1-x_{i})^{2}\langle X(x),N(x)\rangle \gamma_{n}(x)dx)^{2}\\
&\qquad+\int_{\partial A}\mathrm{div}((1-x_{i}^{2})\gamma_{n}(x)X(x))\langle X(x),N(x)\rangle dx(\int_{A}(1-y_{i}^{2})d\gamma_{n}(y)).
\end{aligned}
\end{equation}

We compute the last term as follows.
\begin{equation}\label{Athree3}
\begin{aligned}
&\mathrm{div}((1-x_{i}^{2})\gamma_{n}(x)X(x))
=\sum_{j=1}^{n}\frac{\partial}{\partial x_{j}}((1-x_{i}^{2})\gamma_{n}(x)X^{(j)}(x))\\
&\qquad=\sum_{j=1}^{n}(1-x_{i}^{2})((-x_{j})X^{(j)}(x)+\frac{\partial}{\partial x_{j}}X^{(j)}(x))\gamma_{n}(x)-2\cdot 1_{\{i=j\}}x_{i}\gamma_{n}(x)X^{(j)}(x)\\
&\qquad=(1-x_{i}^{2})\Big(\langle -x,X(x)\rangle+\mathrm{div}(X(x))\Big)\gamma_{n}(x)-2x_{i}\gamma_{n}(x)X^{(i)}(x).
\end{aligned}
\end{equation}
We now combine \eqref{Athree2} and \eqref{Athree3} and sum over $i\in\{1,\ldots,n\}$.  Since $A=B(0,r)$ or $A=B(0,r)^{c}$, the first term from \eqref{Athree3} vanishes by Lemma \ref{lemma3} after it is integrated, resulting in \eqref{newsix}.
\end{proof}

We continue to use the assumptions and notation from Assumtion \ref{as1}.  Also, for any $A\subset\R^{n}$, define
$$F(A)\colonequals\sum_{i=1}^{n}\int_{\R^{n}}\int_{\R^{n}} 1_{A}(x)(1-x_{i}^{2})(1-y_{i}^{2})\gamma_{n}(x)\gamma_{n}(y)1_{A}(y)dxdy.$$
\begin{cor}[\embolden{Second Variation Bound}]\label{varbndcor}
Let $r>0$.  Assume $A=A^{(0)}= B(0,r)$ or $A=A^{(0)}=B(0,r)^{c}$.  Let $f(x)=\langle X(x),N(x)\rangle$ for any $x\in\partial B(0,r)$.  Assume that $\int_{\partial B(0,r)}f(x)dx=0$ and $\int_{\partial B(0,r)}\abs{f(x)}^{2}dx=1$.  Then
$$\frac{1}{2}\frac{d^{2}}{dt^{2}}F(A^{(t)})|_{t=0}\leq\frac{2r^{n+1}e^{-r^{2}}\mathrm{Vol}(S^{n-1})}{n(n+2)(2\pi)^{n}}\left(r^{2}-n-2\right).$$
Moreover, equality holds when $f$ is a multiple of the function $x\mapsto (n-1)x_{1}^{2}-\sum_{j=2}^{n}x_{j}^{2}$.
\end{cor}
\begin{proof}
Suppose $A=A^{(0)}=B(0,r)$, and let $f(x)=\langle X(x),N(x)\rangle$ for any $x\in \partial B(0,r)$.  Then, using Lemma \ref{lemma61} we get
\begin{flalign*}
\frac{1}{2}\frac{d^{2}}{dt^{2}}F(A^{(t)})|_{t=0}
&=\frac{e^{-r^{2}}}{(2\pi)^{n}}\sum_{i=1}^{n}(\int_{\partial B(0,r)}(1-x_{i}^{2})f(x)dx)^{2}\\
%&\qquad-\frac{re^{-r^{2}/2}}{(2\pi)^{n/2}}\int_{B(0,r)}(1-y_{1}^{2})d\gamma_{n}(y)\sum_{i=1}^{n}\int_{\partial B(0,r)}(1-x_{i}^{2})(f(x))^{2}dx\\
&\qquad-\frac{2re^{-r^{2}/2}}{(2\pi)^{n/2}}\int_{B(0,r)} (1-y_{1}^{2})d\gamma_{n}(y))\int_{\partial B(0,r)}(f(x))^{2} dx.
\end{flalign*}
Since $\int_{\partial B(0,r)}f(x)dx=0$, Lemma \ref{lemma50} applies, yielding
\begin{flalign*}
\frac{1}{2}\frac{d^{2}}{dt^{2}}F(A^{(t)})|_{t=0}
%&=\frac{e^{-r^{2}}}{(2\pi)^{n}}\sum_{i=1}^{n}(\int_{\partial B(0,r)}x_{i}^{2}f(x)dx)^{2}\\
%&\qquad-\frac{re^{-r^{2}/2}(n-r^{2})}{(2\pi)^{n/2}}\int_{B(0,r)}(1-y_{1}^{2})d\gamma_{n}(y)\int_{\partial B(0,r)}(f(x))^{2}dx\\
%&\qquad-\frac{2re^{-r^{2}/2}}{(2\pi)^{n/2}}(\int_{B(0,r)} (1-y_{1}^{2})d\gamma_{n}(y))\int_{\partial B(0,r)}(f(x))^{2} dx\\
&=\frac{e^{-r^{2}}}{(2\pi)^{n}}\sum_{i=1}^{n}(\int_{\partial B(0,r)}x_{i}^{2}f(x)dx)^{2}\\
&\qquad-\left(\frac{2re^{-r^{2}/2}}{(2\pi)^{n/2}}\int_{B(0,r)}(1-y_{1}^{2})d\gamma_{n}(y)\right)\int_{\partial B(0,r)}(f(x))^{2} dx\\
&\leq\frac{2r^{n+3}e^{-r^{2}}\mathrm{Vol}(S^{n-1})}{n(n+2)(2\pi)^{n}}\int_{\partial B(0,r)}(f(x))^{2} dx\\
&\qquad-\left(\frac{2re^{-r^{2}/2}}{(2\pi)^{n/2}}\int_{B(0,r)}(1-y_{1}^{2})d\gamma_{n}(y)\right)\int_{\partial B(0,r)}(f(x))^{2} dx.\\
\end{flalign*}

So, if $\int_{\partial B(0,r)}(f(x))^{2} dx=1$, we have
\begin{equation}\label{eight7}
\frac{1}{2}\frac{d^{2}}{dt^{2}}F(A^{(t)})|_{t=0}\\
\leq\frac{2r^{n+3}e^{-r^{2}}\mathrm{Vol}(S^{n-1})}{n(n+2)(2\pi)^{n}}
-\left(\frac{2re^{-r^{2}/2}}{(2\pi)^{n/2}}\int_{B(0,r)}(1-y_{1}^{2})d\gamma_{n}(y)\right).
\end{equation}

Then, substituting Lemma \ref{newlem},
\begin{equation}
\begin{aligned}
\frac{1}{2}\frac{d^{2}}{dt^{2}}F(A^{(t)})|_{t=0}
&\leq\frac{r^{n+1}e^{-r^{2}}\mathrm{Vol}(S^{n-1})}{n(2\pi)^{n}}\left(r^{2}\frac{2}{n+2}-2\right)\\
&=\frac{2r^{n+1}e^{-r^{2}}\mathrm{Vol}(S^{n-1})}{n(n+2)(2\pi)^{n}}\left(r^{2}-n-2\right).  %
\end{aligned}
\end{equation}

The equality case then comes directly from Lemma \ref{lemma50}, if we can find a vector field $X$ such that $\langle X(x),N(x)\rangle=(n-1)x_{1}^{2}-x_{2}^{2}-\cdots-x_{n}^{2}$ for any $x\in\partial A$ and such that $\frac{d^{2}}{dt^{2}}|_{t=0}\gamma_{n}(A^{t})=0$.  (Since $\int_{\partial A}[(n-1)x_{1}^{2}-x_{2}^{2}-\cdots-x_{n}^{2}]dx=0$, we know $\frac{d}{dt}|_{t=0}\gamma_{n}(A^{t})=0$ by Lemma \ref{lemma3}.)

We can construct the vector field $X$ explicitly.  For any $1\leq i\leq n$, define $g_{i}(x)\colonequals(x_{1}^{2}+\cdots+x_{n}^{2}-r^{2})x_{i}/2$, $x\in\R^{n}$.  Then $\frac{\partial}{\partial x_{i}}g_{i}(x)=x_{i}^{2}$ and $g_{i}(x)=0$ when $x\in\partial A$.  We therefore define
$$X(x)\colonequals r\big((n-1)(x_{1}+g_{1}(x)),-x_{2}-g_{2}(x),\ldots,-x_{n}-g_{n}(x)\big),\qquad\forall\,x\in\R^{n},$$
Then if $x\in\partial A$, we have
\begin{flalign*}
&\frac{1}{r}(\mathrm{div}(X)(x)-\langle X(x),x\rangle)\\
&=(n-1)\frac{\partial}{\partial x_{1}}g_{1}(x)-\frac{\partial}{\partial x_{2}}g_{2}(x)-\cdots-\frac{\partial}{\partial x_{n}}g_{n}(x)
-((n-1)x_{1}^{2}-x_{2}^{2}-\cdots-x_{n}^{2})=0.
\end{flalign*}
%Also, $N(x)=\pm x/\vnorm{x}_{2}$ for any $x\in\partial A$, and 
So, $\frac{d^{2}}{dt^{2}}|_{t=0}\gamma_{n}(A^{(t)})=0$ by Lemma \ref{lemma3}, as desired.  Lastly, note that $N(x)=\pm x/\vnorm{x}_{2}$ for all $x\in\partial A$, so $\pm\langle X(x),N(x)\rangle=(n-1)x_{1}^{2}-x_{2}^{2}-\cdots-x_{n}^{2}$ for any $x\in\partial A$ since $g_{i}(x)=0$ for any $x\in\partial A$, and for any $1\leq i\leq n$.

%
%
%And if $n=2$, we have $\int_{B(0,r)}(1-y_{1}^{2})d\gamma_{n}(y)=(1/2)r^{2}e^{-r^{2}/2}$, so
%\begin{equation}\label{eight1}
%\frac{1}{2}\frac{d^{2}}{dt^{2}}F(A^{(t)})|_{t=0}
%\leq \frac{r^{3}e^{-r^{2}}}{4\pi}\left(-4+r^{2}\left(\frac{1}{4\pi}+1\right)\right).
%\end{equation}
%\snote{Above seems wrong}
%Similarly, if $A=A^{(0)}=B(0,r)^{c}$, then \snote{???}
%%\begin{equation}\label{eight2}
%%\frac{1}{2}\frac{d^{2}}{dt^{2}}F(A^{(t)})|_{t=0}
%%\leq \frac{r^{3}e^{-r^{2}}}{4\pi}\left(r^{2}\left(-1+\frac{\pi}{8}\right)+4\right).
%%\end{equation}
\end{proof}

\section{Local Optimality}\label{secopt}

We continue to use the assumptions and notation from Assumtion \ref{as1}.  Also, for any $A\subset\R^{n}$, define
$$F(A)\colonequals\sum_{i=1}^{n}\int_{\R^{n}}\int_{\R^{n}} 1_{A}(x)(1-x_{i}^{2})(1-y_{i}^{2})\gamma_{n}(x)\gamma_{n}(y)1_{A}(y)dxdy.$$

\begin{proof}[Proof of Theorem \ref{newthm}]
Apply Corollary \ref{varbndcor}.
\end{proof}

\begin{remark}
 Note that there is a ``phase transition'' that occurs in Theorem \ref{newthm} (and in Corollary \ref{varbndcor}), where the ball or its complement changes from locally maximizing $F$ to not locally maximizing $F$.  We do not currently have an intuitive explanation for this phenomenon.
\end{remark}

\begin{remark}\label{goodrk}
%Consider $n=2$ and $a=.1$.  Then $\gamma_{n}(B(0,r))=1-e^{-r^{2}/2}=a= .1$ when $r=\sqrt{\log(100/81)}$, and $\gamma_{n}(B(0,r')^{c})=e^{-(r')^{2}/2}=.1$ when $r'=\sqrt{\log(100)}$.  From Corollary \ref{varbndcor}, $B(0,r)$ locally maximizes $F$, whereas $B(0,r')^{c}$ does not.  And Lemma \ref{newlem} says $\int_{B(0,r)}(1-x_{1}^{2})d\gamma_{2}(x)=\frac{1}{2}r^{2}e^{-r^{2}/2}$
%
%
%\snote{Need to edit}
Here we provide some more details for the calculation demonstrated in the Introduction which demonstrated the incorrectness of Conjecture \ref{SGP1}.

Let $r=2.4$.  Then $\gamma_{2}(B(0,r))=1-e^{r^{2}/2}\approx .943865$.  And from Lemma \ref{newlem}, $(\int_{B(0,r)}(1-x_{1}^{2})d\gamma_{2}(x))^{2}+(\int_{B(0,r)}(1-x_{2}^{2})d\gamma_{2}(x))^{2}=\frac{1}{2}r^{4}e^{-r^{2}}\approx.0522732$.

Let $A=\{(x_{1},x_{2})\in\R^{2}\colon x_{1}^{2}/(2.5)^{2}+x_{2}^{2}/(2.31394)^{2}\leq1\}$.  A numerical computation shows that $\gamma_{2}(A)\approx.943865$, $\int_{A}(1-x_{1}^{2})d\gamma_{2}(x)\approx .143076$ and $\int_{A}(1-x_{2}^{2})d\gamma_{2}(x)\approx .178889$, so $(\int_{A}(1-x_{1}^{2})d\gamma_{2}(x))^{2}+(\int_{A}(1-x_{2}^{2})d\gamma_{2}(x))^{2}\approx .0524720>.0522732$.

That is, if $F(A)\colonequals(\int_{A}(1-x_{1}^{2})d\gamma_{2}(x))^{2}+(\int_{A}(1-x_{2}^{2})d\gamma_{2}(x))^{2}$, then $F(A^{c})\approx .0524720>.0522732\approx F(B(0,r)^{c})$.  And if $r'=\sqrt{-2\log(1-e^{-2.88})}$, then $\gamma_{2}(B(0,r'))=\gamma_{2}(B(0,r)^{c})$, and again from Lemma \ref{newlem}, $F(B(0,r'))=\frac{1}{2}(r')^{4}e^{-(r')^{2}}=2(\log(1-e^{-2.88}))^{2}(1-e^{-2.88})^{2}\approx.0059468$.  In summary,
$$F(A)>\max(F(B(0,r)^{c}),F(B(0,r'))),\qquad \gamma_{2}(B(0,r)^{c})=\gamma_{2}(B(0,r'))\approx\gamma_{2}(A).$$
That is, Conjectures \ref{SGP} and \ref{SGP1} are false.

In fact, this behavior is generic for other measure restrictions when $n=2$.  If $1-e^{-r^{2}/2}=e^{-(r')^{2}/2}$, then $r'=\sqrt{-2\log(1-e^{-r^{2}/2})}$, $F(B(0,r))=\frac{1}{2}r^{4}e^{-r^{2}}$, and $F(B(0,r')^{c})=2(\log(1-e^{-r^{2}/2}))^{2}(1-e^{-r^{2}/2})^{2}\geq F(B(0,r))$ for all $0<r<\sqrt{2}$.  So, $\max(F(B(0,r)),F(B(0,r')))=F(B(0,r'))$.  And from Corollary \ref{varbndcor}, $B(0,r')$ locally maximizes $F$ only when $r'<2$.  That is, $B(0,r')$ locally maximizes $F$ only when $r>\sqrt{4-2\log(e^{2}-1)}\approx .53928$.

In summary, if $0<r< .53928$, and if $a=\gamma_{2}(B(0,r))$, then Conjectures \ref{SGP} and \ref{SGP1} are false, since $F(B(0,r))<F(B(0,r')^{c})$, and $B(0,r')^{c}$ does not maximize $F$ by Corollary \ref{varbndcor}, since $r'>2$.  Moreover, as mentioned in the Introduction, if $A'=\{(x_{1},x_{2})\in\R^{2}\colon x_{1}^{2}\leq1.90999\}$, then a numerical computation shows $\gamma_{2}(A')\approx .943865$, and $F(A')=F((A')^{c})\approx.0604796$.  That is,

% x^2 /a^2 +y^2 /b^2 =1
% a=2.5, b=2.314
% so -a\leq x\leq a
% and -b\sqrt{1-x^2 /a^2}\leq y\leq b\sqrt{1-x^{2}/a^{2}}
%
%exp((-x^2 -y^2)/2)*(1/(2*pi))
%(1-x^2)exp((-x^2 -y^2)/2)*(1/(2*pi))
%dydx
%-2.5, 2.5
%-2.31394*sqrt(1-(x^2)/(2.5)^2)  , 2.31304*sqrt(1-(x^2)/(2.5)^2)
%
%http://www.wolframalpha.com/widgets/view.jsp?id=f5f3cbf14f4f5d6d2085bf2d0fb76e8a
\end{remark}

\begin{proof}[Proof of Theorem \ref{thm3}]
Let $\rho\in(-1,1)$.  Let $G(x,y)=e^{-\frac{\vnorm{\rho x-y}^{2}}{2(1-\rho^{2})}}\gamma_{n}(x)=e^{\frac{-\vnorm{x}_{2}^{2}-\vnorm{y}_{2}^{2}+2\rho\langle x,y\rangle}{2(1-\rho^{2})}}$ for any $x,y\in\R^{n}$.  For any $A\subset\R^{n}$, define
$$F_{\rho}(A)\colonequals\int_{\R^{n}}\int_{\R^{n}} 1_{A}(x)G(x,y)1_{A}(y)dxdy
=\int_{\R^{n}}\int_{\R^{n}} 1_{A}(x)T_{\rho}1_{A}(x)d\gamma_{n}(x).
$$
$$
F(A)\colonequals\int_{\R^{n}}\int_{\R^{n}} 1_{A}(x)\left(\sum_{i=1}^{n}(1-x_{i}^{2})(1-y_{i}^{2})\right)\gamma_{n}(x)\gamma_{n}(y)1_{A}(y)dxdy.
$$
%From Corollary \ref{varbndcor}, if $A^{(0)}=B(0,r)$, then
%\begin{equation}\label{eight3}
%\frac{1}{2}\frac{d^{2}}{dt^{2}}|_{t=0}\frac{d^{2}}{d\rho^{2}}|_{\rho=0}F(A^{(t)})
%\leq\frac{r^{2}e^{-r^{2}}}{4\pi}\left(-3+r^{2}\left(\frac{1}{4\pi}+1\right)\right).
%\end{equation}
%From Corollary \ref{varbndcor}, if $A^{(0)}=B(0,r)^{c}$, then
%\begin{equation}\label{eight4}
%\frac{1}{2}\frac{d^{2}}{dt^{2}}F(A^{(t)})|_{t=0}
%\leq \frac{r^{3}e^{-r^{2}}}{4\pi}\left(r^{2}\left(-1+\frac{1}{4\pi}\right)+3\right).
%\end{equation}

We require the following well-known Gaussian/Mehler heat kernel expansion for $G$, which appears e.g. in \cite[Section 2.2]{heilman14b}: for any $x,y\in\R^{n}$, and for any $\rho\in(-1,1)$,
\begin{equation}\label{Csix}
e^{-(\|x\|_{2}^{2}+\|y\|_{2}^{2})/2}
\sum_{k=0}^{\infty}\rho^{k}\sum_{\substack{\ell\in\N^{n}\colon\\ \ell_{1}+\cdots+\ell_{n}=k}}h_{\ell}(x)h_{\ell}(y)\ell!
=(1-\rho^{2})^{-n/2}e^{-\frac{(\|x\|_{2}^{2}+\|y\|_{2}^{2}-2\rho\langle x,y\rangle)}{2(1-\rho^{2})}}.
\end{equation}
Combining \eqref{Csix} with the bound for Hermite polynomials in Lemma \ref{lemma6}, for any $x,y\in\R^{n}$,
\begin{equation}\label{Cseven}
\begin{aligned}
&\Big|(1-\rho^{2})^{-n/2}G(x,y)
-\gamma_{n}(x)\gamma_{n}(y)\Big(1+\rho\sum_{i=1}^{n}x_{i}y_{i}
+\rho^{2}\sum_{\substack{\ell\in\N^{n}\colon\\ \ell_{1}+\cdots+\ell_{n}=2}}h_{\ell}(x)h_{\ell}(y)\ell!\Big)\Big|\\
&\qquad\qquad\qquad\qquad
\leq\gamma_{n}(x)\gamma_{n}(y)\sum_{k=3}^{\infty}\rho^{k}(n+k)^{n}3^{k}k^{n}(1+\vnorm{x}_{2}^{k})(1+\vnorm{y}_{2}^{k})\Big).
\end{aligned}
\end{equation}

Below, $C$ denotes a large constant that depends on $n,r,\rho$, whose value can change each time it appears.  Also, $A$ denotes $B(0,r)$ or $B(0,r)^{c}$.  We now show that the formulas from Lemmas \ref{lemma60} and \ref{lemma61} are close in a precise sense.

Write $X(x)=f(x)N(x)$, where $x\in\partial B(0,r)$.  Since $\frac{d}{dt}\gamma_{n}(A^{(t)})=0$, Lemma \ref{lemma3} implies that $\int_{\partial B(0,r)}f(x)dx=0$.  Also, recall that if $A\subset\R^{n}$ is symmetric, then \eqref{one3.1} implies that $\int_{A}x_{i}d\gamma_{n}(x)=0$ for any $i\in\{1,\ldots,n\}$.  Combining these facts with \eqref{Cseven} and choosing $\rho$ sufficiently small (depending on $n$ and $r$), we have
\begin{equation}\label{end1}
\begin{aligned}
&\Big|\int_{\partial A}\int_{\partial A}G(x,y)f(x)f(y) dxdy
-\rho^{2}\int_{\partial A}\int_{\partial A}\sum_{\substack{\ell\in\N^{n}\colon\\ \ell_{1}+\cdots+\ell_{n}=2}}h_{\ell}(x)h_{\ell}(y)\ell!f(x)f(y)\Big|\\
&\leq C\abs{\rho}^{5/2}(\int_{\partial A}\abs{f(x)}\gamma_{n}(x)dx)^{2}
\leq C\abs{\rho}^{5/2}(\int_{\partial A}\abs{f(x)}^{2}\gamma_{n}(x)dx).
\end{aligned}
\end{equation}

Similarly, it follows from \eqref{six1.8} that, $\forall$ $x\in\R^{n}$,
$$
\Big|T_{\rho}1_{A}(x)-\gamma_{n}(a)-\rho^{2}\sum_{\substack{\ell\in\N^{n}\colon\\ \ell_{1}+\cdots+\ell_{n}=2}}\ell! h_{\ell}(x)(\int_{A}h_{\ell}(y)d\gamma_{n}(y))\Big|
\leq C\abs{\rho}^{5/2}(1+\vnorm{x}_{2}^{k}).
$$
$$
\Big|\frac{\partial}{\partial x_{i}}T_{\rho}1_{A}(x)-\rho^{2}\sum_{\substack{\ell\in\N^{n}\colon\\ \ell_{1}+\cdots+\ell_{n}=2}}\ell! h_{\ell}'(x)(\int_{A}h_{\ell}(y)d\gamma_{n}(y))\Big|
\leq C\abs{\rho}^{5/2}(1+\vnorm{x}_{2}^{k}).
$$

From Remark \ref{rotrk}, we may assume that $\int_{A^{(t)}}x_{i}x_{j}d\gamma_{n}(x)=0$ whenever $i,j\in\{1,\ldots,n\}$ with $i\neq j$, and for all $t\in(-1,1)$.  Also, from \eqref{Bone6} below (where the $G$ we use there is $G(x,y)\colonequals y_{i}y_{j}$ for all $x,y\in\R^{n}$), we may assume that $\int_{\partial A}x_{i}x_{j}f(x)d\gamma_{n}(x)=0$ whenever $i,j\in\{1,\ldots,n\}$ with $i\neq j$.  Consequently,
%\begin{equation}\label{end2}
%\begin{aligned}
%&\Big|\int_{\partial A}\langle x,X(x)\rangle T_{\rho}1_{A}(x)\gamma_{n}(x)dx\\
%&\qquad-\frac{1}{2}\sum_{i=1}^{n}(\int_{A} (1-y_{i}^{2})d\gamma_{n}(y))\int_{\partial A}(1-x_{i}^{2})\langle-x,X(x)\rangle\langle X(x),N(x)\rangle \gamma_{n}(x)dx\Big|\\
%&\qquad\qquad\qquad\leq C\abs{\rho}^{5/2}(\int_{\partial A}\abs{f(x)}^{2}dx).
%\end{aligned}
%\end{equation}

\begin{equation}\label{end3}
\begin{aligned}
&\Big|\int_{\partial A}\langle\nabla T_{\rho}1_{A}(x),X(x)\rangle\langle X(x),N(x)\rangle dx\\
&\quad-\sum_{i=1}^{n}\int_{\partial A}(-x_{i})X^{(i)}\langle X(x),N(x)\rangle d\gamma_{n}(x)(\int_{A} (1-y_{i}^{2})d\gamma_{n}(y))\Big|\\
&\qquad\qquad\qquad\leq C\abs{\rho}^{5/2}(\int_{\partial A}\abs{f(x)}^{2}dx).
\end{aligned}
\end{equation}

So, combining \eqref{end1} and \eqref{end3} with Lemmas \ref{lemma60} and \ref{lemma61}, (and $\int_{\partial A}x_{i}x_{j}f(x)\gamma_{n}(x)=0$ if $i,j\in\{1,\ldots,n\}$ with $i\neq j$,)
$$\abs{\frac{1}{\rho^{2}}\frac{d^{2}}{dt^{2}}|_{t=0}F_{\rho}(A^{(t)})
-\frac{1}{2}\frac{d^{2}}{dt^{2}}|_{t=0}F(A^{(t)})}
\leq C\abs{\rho}^{1/2}\int_{\partial B(0,r)}(f(x))^{2}dx.$$

That is, for $\rho$ sufficiently small (depending on $n$ and $r$), Theorem \ref{thm3} follows from Theorem \ref{newthm}.  That is, the ball or its complement is a local maximum of noise stability among symmetric sets.
\end{proof}

\section{Asymptotics for Second Degree Fourier Coefficients}\label{sechdim}

\begin{lemma}[\cite{matsunawa76}, Theorem 2.1]\label{lemma26}
Let $m\geq3$ be an integer.  Then there exists $\lambda_{m}$ such that
\begin{flalign*}
&0\leq\frac{1}{360m(m-1)(m+1)}-\frac{1}{120m^{2}(m-1)(m+1)}\\
&\qquad\qquad\qquad\leq\lambda_{m}\leq\frac{1}{360m(m-1)(m+1)}+\frac{11}{480m^{2}(m-1)(m+1)},
\end{flalign*}
and such that for any integer $n\geq6$,
$$\Gamma(n/2)=\sqrt{2\pi}((n-2)/2)^{(n-1)/2}e^{-(n-2)/2}e^{1/(6(n-2))}e^{-\lambda_{(n-2)/2}}.$$
\end{lemma}

\begin{lemma}\label{lemma25}
For $n\in\N$ and $s\in\R$ define $r(s,n)\colonequals\sqrt{n+s\sqrt{2n}}$.  Then, as $n\to\infty$, the following asymptotic holds:
$$\sum_{i=1}^{n}\left(\int_{B(0,r(s,n))}(1-x_{i}^{2})d\gamma_{n}(x)\right)^{2}
=\frac{1}{\pi}\exp(-s^{2}+2+s^{3}2\sqrt{2}n^{-1/2}/3-s^{4}n^{-1}+O(n^{-3/2})).$$
Moreover, in the case $s=0$, the quantity $\sum_{i=1}^{n}(\int_{B(0,\sqrt{n})}(1-x_{i}^{2})d\gamma_{n}(x))^{2}$ strictly increases as $n$ increases.
\end{lemma}
\begin{proof}
We begin with the first statement.  Using Lemma \ref{newlem},
\begin{flalign*}
&\sum_{i=1}^{n}\left(\int_{B(0,r(s,n))}(1-x_{i}^{2})d\gamma_{n}(x)\right)^{2}\\
&\qquad=2^{-n}\frac{4}{n}\frac{1}{2\pi}\left(\frac{2}{n-2}\right)^{n-1}e^{n-2}e^{-1/(3(n-2))}e^{-\lambda^{2}_{(n-2)/2}}(n+s\sqrt{2n})^{n}e^{-n-s\sqrt{2n}}\\
&\qquad=\frac{1}{\pi}e^{-2}e^{-1/(3(n-2))}e^{-\lambda^{2}_{(n-2)/2}}\frac{(1+s\sqrt{2}/\sqrt{n})^{n}}{(1-2/n)^{n-1}e^{s\sqrt{2n}}}.
\end{flalign*}

Taking the logarithm of the fraction, we get
\begin{flalign*}
&n\log(1+s\sqrt{2}/\sqrt{n})-(n-1)\log(1-2/n)-s\sqrt{2n}\\
&\qquad=n(s\sqrt{2}n^{-1/2}-s^{2}n^{-1}+s^{3}2\sqrt{2}n^{-3/2}/3-s^{4}n^{-2}+O(n^{-5/2}))\\
&\qquad\qquad\qquad-(n-1)(-2n^{-1}-2n^{-2}-8n^{-3}/3-O(n^{-4}))
-s\sqrt{2n}\\
&\qquad=-s^{2}+2+s^{3}2\sqrt{2}n^{-1/2}/3-s^{4}n^{-1}+O(n^{-3/2}).
\end{flalign*}

Combining these estimates proves the asymptotic formula.

We now consider the case $s=0$.  For any $t>2$, define
$$g(t)=e^{-\frac{1}{3(t-2)}}\frac{t^{t-1}}{(t-2)^{t-1}}=e^{-\frac{1}{3(t-2)}+(1-t)[\log(t)-\log(t-2)]}.$$
Then
\begin{flalign*}
g'(t)
&=\left(\frac{1}{3(t-2)^{2}}+\log(t)-\log(t-2)+(t-1)\left(\frac{1}{t}-\frac{1}{t-2}\right)\right)g(t)\\
&=\left(\frac{1}{3(t-2)^{2}}+\log\left(1+\frac{2}{t-2}\right)-\frac{1}{t-2}-\frac{1}{t}\right)g(t).
\end{flalign*}
Let $x=1/(t-2)$ so that $t-2=1/x$, $t=2+1/x$, and $1/t=x/(2x+1)$.
Let $h(x)=(1/3)x^{2}+\log(1+2x)-x-x/(2x+1)$.  Then $h'(x)=(2/3)x+2/(2x+1)-1-[(2x+1)-2x]/(2x+1)^{2}=(2/3)x+2/(2x+1)-1-1/(2x+1)^2$.  Then $(2x+1)^{2}h'(x)=(2/3)x(2x+1)^{2}+2(2x+1)-(2x+1)^{2}-1=(8/3)x^{3}+(8/3)x^{2}+(2/3)x+4x+2-4x^{2}-4x-1-1=(8/3)x^{3}-(4/3)x^{2}+(2/3)x=(1/3)x(8x^{2}-4x+2)$.  And the function $x\mapsto 8x^{2}-4x+2$ has a positive minimum at $x=1/4$.  So, $h'(x)>0$ for all $x\in(0,1)$.  That is, $g'(t)\geq0$ for any $t\geq3$.  Therefore, the quantity $\sum_{i=1}^{n}(\int_{B(0,\sqrt{n})}(1-x_{i}^{2})d\gamma_{n}(x))^{2}$ strictly increases as $n$ increases, if $n\geq6$.  (The case $1\leq n< 6$ follows by direct computation.)
\end{proof}
%\begin{remark}
%Combining \eqref{seven24} and Lemma \ref{lemma24} shows that, as $n\to\infty$, maximizing the sum of squared second order Hermite coefficients reduces to a ``one-dimensional'' problem.  Specifically, a ball in $\R^{n}$ of radius $\sqrt{n+s\sqrt{2n}}$ has measure $\gamma_{1}(-\infty,s)$ as $n\to\infty$, and the squared sum of its second order Fourier coefficients is equal to $e^{-s^{2}}/\pi$ as $n\to\infty$.
%\end{remark}

\section{Appendix: Proof of the Second Variation Formula}\label{sec2var}

Let $A\subset\R^{n}$ be a set with smooth boundary, and let $N\colon\partial A\to S^{n-1}$ denote the unit exterior normal to $\partial A$.  Let $X\colon\R^{n}\to\R^{n}$ be a vector field.  Let $\Psi\colon\R^{n}\times(-1,1)$ such that $\Psi(x,0)=x$ and such that $\frac{d}{dt}|_{t=0}\Psi(x,t)=X(\Psi(x,t))$ for all $x\in\R^{n},t\in(-1,1)$.  For any $t\in(-1,1)$, let $A^{(t)}=\Psi(A,t)$.   Define
\begin{equation}\label{Btwo5.1}
V(x,t)\colonequals\int_{A^{(t)}}G(x,y)dy
\end{equation}
%Define also $V(x,t)=\int_{A^{(t)}}$

\begin{lemma}[\embolden{The First Variation}\,\cite{chokski07}; also {\cite[Lemma 3.1, Equation (7)]{heilman14}}]\label{latelemma3}
Let $G\colon\R^{n}\times\R^{n}\to\R$ be a Schwartz function.
\begin{equation}\label{Bone6}
\frac{d}{dt}|_{t=0}\int_{\R^{n}} 1_{A^{(t)}}(y)G(x,y)dy
=\int_{\partial A}G(x,y)\langle X(y),N(y)\rangle dy.
\end{equation}
In particular, setting $G(x,y)=\gamma_{n}(y)$, we get
$$\frac{d}{dt}|_{t=0}\gamma_{n}(A^{(t)})=\int_{\partial A}\langle X(y),N(y)\rangle \gamma_{n}(y)dy.$$
\end{lemma}

\begin{lemma}[\embolden{The Second Variation}, {\cite[Theorem 2.6]{chokski07}}]\label{latelemma}

Let $G\colon\R^{n}\times\R^{n}\to\R$ be a Schwartz function.  Then %Let $F\colonequals \int_{\R^{n}} 1_{A}(x)G(x,y)1_{A}(y)dxdy$.  Assume that
%$$\frac{d}{dt}|_{t=0}\gamma_{n}(A^{(t)})=0,$$
%$$\frac{d^{2}}{dt^{2}}|_{t=0}\gamma_{n}(A^{(t)})=0,$$
%
%Then a normal variation of $A$ satisfies

\begin{flalign*}
\frac{1}{2}\frac{d^{2}}{dt^{2}}|_{t=0}\int_{\R^{n}}\int_{\R^{n}} 1_{A^{(t)}}(x)G(x,y)1_{A^{(t)}}(y)dy
&=\int_{\partial A}\int_{\partial A}G(x,y)\langle X(x),N(x)\rangle\langle X(y),N(y)\rangle dxdy\\
&\qquad+\int_{\partial A}\mathrm{div}(V(x,0)X(x))\langle X(x),N(x)\rangle dx.
\end{flalign*}
%We take $G(x,y)=e^{-(x^{2}-2\rho xy+y^{2})/[2(1-\rho^{2})]}$ or $G(x,y)=\langle x,y\rangle e^{-(x^{2}+y^{2})/2}$.
%...
\end{lemma}
\begin{proof}
%Let $A\subset\R^{n}$ be a set with smooth boundary, and let $N\colon\partial A\to S^{n-1}$ denote the unit exterior normal to $\partial A$.  Let $X\colon\R^{n}\to \R^{n}$.  Let $\Psi\colon\R^{n}\times(-1,1)$ such that $(d/dt)|_{t=0}\Psi(x,t)=X(x)$ for all $x\in\R^{n}$, and such that $\Psi(x,0)=x$ for all $x\in\R^{n}$.
%Define $A^{(t)}\colonequals\Psi(A,t)$.
Write $\Psi$ and $X$ in their components as $\Psi=(\Psi^{(1)},\ldots,\Psi^{(n)})$, $X=(X^{(1)},\ldots,X^{(n)})$.  We use subscript notation to denote partial derivatives, and we let $\mathrm{div}(X)=\sum_{i=1}^{n}X_{i}^{(i)}$ denote the divergence of $X$.  Let $J\Psi(y,t)$ denote $\abs{\det D\Psi(y,t)}=\abs{\det(\partial\Psi^{(i)}(y,t)/\partial y_{j})_{1\leq i,j\leq n}}\in\R$.

By assumption,
\begin{equation}\label{Btwo1}
\frac{d\Psi}{dt}|_{t=0}=X(\Psi(x,0))=X(x).
\end{equation}
Since $\Psi$ is smooth, we can write
$$D\Psi(x,t)=I+tDX+\frac{1}{2}t^{2}DZ+o(t^{2}),$$
$$Z\colonequals\frac{d^{2}\Psi}{dt^{2}}|_{t=0},\quad Z^{(i)}=\sum_{j=1}^{n}X_{j}^{(i)}X^{(j)}.$$
We then have the determinant expansion
\begin{flalign*}
&\det(D\Psi(x,t))
=1+t\mathrm{Tr}(DX)+\frac{1}{2}t^{2}[\mathrm{Tr}(DZ)+(\mathrm{Tr}(DX))^{2}-\mathrm{Tr}((DX)^{2})]+o(t^{2})\\
&\quad=1+t\mathrm{Tr}(DX)+\frac{1}{2}t^{2}[\mathrm{div}(Z)+(\mathrm{div}(X))^{2}-\sum_{i,j=1}^{n}X_{j}^{(i)}X_{i}^{(j)}]+o(t^{2})\\
&\quad=1+t\mathrm{Tr}(DX)+\frac{1}{2}t^{2}[\sum_{i,j=1}^{n}X_{ij}^{(i)}X^{(j)}+\sum_{i,j=1}^{n}X_{j}^{(i)}X_{i}^{(j)}+(\mathrm{div}(X))^{2}-\sum_{i,j=1}^{n}X_{j}^{(i)}X_{i}^{(j)}]+o(t^{2})\\
&\quad=1+t\mathrm{Tr}(DX)+\frac{1}{2}t^{2}\mathrm{div}(\mathrm{div}(X)X)+o(t^{2}).\\
\end{flalign*}

Since $J\Psi(x,t)=\abs{\det(D\Psi(x,0))}$, we therefore have
\begin{equation}\label{Btwo2}
J\Psi(x,0)=1.
\end{equation}
\begin{equation}\label{Btwo3}
(d/dt)J\Psi|_{t=0}=\mathrm{div}(X).
\end{equation}
\begin{equation}\label{Btwo4}
\frac{d^{2}\Psi^{(i)}}{dt^{2}}|_{t=0}=\sum_{j=1}^{n}X_{j}^{(i)}X^{(j)}.
\end{equation}
\begin{equation}\label{Btwo5}
\frac{d^{2}}{dt^{2}}J\Psi(x,t)|_{t=0}=\mathrm{div}((\mathrm{div}(X))X).
\end{equation}

Let
\begin{equation}\label{Btwo5.2}
F(A^{(t)})=\int_{\R^{n}} 1_{A^{(t)}}(x)G(x,y)1_{A^{(t)}}(y)dxdy=\int_{A^{(t)}}V(x,t)=\int_{A}V(\Psi(x,t),t)J\Psi(x,t)dx.
\end{equation}

In the sequel, we will use the chain rule and divergence theorem repeatedly.

\begin{equation}\label{Btwo6}
\begin{aligned}
\frac{d}{dt}F(A^{(t)})
&=\int_{A}\sum_{i=1}^{n}V_{x_{i}}(\Psi(x,t),t)\Psi_{t}^{(i)}(x,t)J\Psi(x,t)+V(\Psi(x,t),t)\frac{d}{dt}(J\Psi(x,t))\\
&\qquad+V_{t}(\Psi(x,t),t)J\Psi(x,t).
\end{aligned}
\end{equation}

\noindent
\embolden{Step 1. Computing the Second Derivative of $F(A^{(t)})$ with respect to $t$}.

\begin{equation}\label{Btwo7}
\begin{aligned}
&\frac{d^{2}}{dt^{2}}F(A^{(t)})
=\int_{A}\sum_{i,j=1}^{n}V_{x_{i}x_{j}}(\Psi(x,t),t)\Psi_{t}^{(i)}(x,t)\Psi_{t}^{(j)}(x,t)J\Psi(x,t)\\
&\qquad+2\sum_{i=1}^{n}V_{x_{i},t}(\Psi(x,t),t)\Psi_{t}^{(i)}(x,t)J\Psi(x,t)
+\sum_{i=1}^{n}V_{x_{i}}(\Psi(x,t),t)\Psi_{tt}^{(i)}(x,t)J\Psi(x,t)\\
&\qquad+2\sum_{i=1}^{n}V_{x_{i}}(\Psi(x,t),t)\Psi_{t}^{(i)}(x,t)(d/dt)J\Psi(x,t)
+2V_{t}(\Psi(x,t),t)(d/dt)J\Psi(x,t)\\
&\qquad+V(\Psi(x,t),t)(d^{2}/dt^{2})(J\Psi(x,t))
+V_{tt}(\Psi(x,t),t)J\Psi(x,t) dx.
\end{aligned}
\end{equation}

\begin{equation}\label{Btwo8}
\begin{aligned}
&\frac{d^{2}}{dt^{2}}F(A^{(t)})|_{t=0}
\stackrel{\substack{\eqref{Btwo1}\wedge\eqref{Btwo2}\wedge\\ \eqref{Btwo3}\wedge\eqref{Btwo4}\wedge\eqref{Btwo5}}}{=}
\int_{A}\sum_{i,j=1}^{n}V_{x_{i}x_{j}}(x,0)X^{(i)}(x)X^{(j)}(x)\\
&\qquad+2\sum_{i=1}^{n}V_{x_{i},t}(x,t)X^{(i)}(x)
+\sum_{i,j=1}^{n}V_{x_{i}}(x,t)X_{x_{j}}^{(i)}(x)X^{(j)}(x)\\
&\qquad+2\sum_{i=1}^{n}V_{x_{i}}(x,t)X^{(i)}(x)\mathrm{div}(X(x))
+2V_{t}(x,0)\mathrm{div}(X(x))\\
&\qquad+V(x,0)\mathrm{div}((\mathrm{div}(X(x)))X(x))
+V_{tt}(x,t) dx.
\end{aligned}
\end{equation}

From \eqref{Bone6}, $V_{t}(x,0)=\int_{\partial A}G(x,y)\langle X(y),N(y)\rangle dy$.  So, combining the second and fifth terms of \eqref{Btwo8}, then applying the divergence theorem,
\begin{equation}\label{Btwo9}
\begin{aligned}
&\int_{A}2\langle \nabla_{x}V_{t}(x,0),X(x)\rangle+2V_{t}(x,0)\mathrm{div}(X(x))dx\\
&\qquad=2\int_{A}\mathrm{div}(V_{t}(x,0)X(x))dx
=2\int_{\partial A}V_{t}(x,0)\langle X(x),N(x)\rangle dx\\
&\qquad=2\int_{\partial A}\int_{\partial A}G(x,t)\langle X(x),N(x)\rangle\langle X(y),N(y)\rangle dxdy.
\end{aligned}
\end{equation}

Combining the first, third and fourth terms of \eqref{Btwo8}, and using the divergence theorem,
\begin{equation}\label{Btwo10}
\begin{aligned}
&\int_{A}\sum_{i,j=1}^{n}V_{x_{i}x_{j}}(x,0)X^{(i)}(x)X^{(j)}(x)
+\sum_{i,j=1}^{n}V_{x_{i}}(x,t)X_{x_{j}}^{(i)}(x)X^{(j)}(x)\\
&\qquad\qquad\qquad\qquad\qquad\qquad\qquad\qquad\qquad
+\sum_{i=1}^{n}V_{x_{i}}(x,t)X^{(i)}(x)\mathrm{div}(X(x)) dx\\
&\qquad=\int_{A}\mathrm{div}(\langle\nabla_{x}V(x,0),X(x)\rangle X(x))dx
=\int_{\partial A}\langle \nabla_{x}V(x,0),X(x)\rangle\langle X(x),N(x)\rangle dx
\end{aligned}
\end{equation}

%Applying the divergence theorem to the sixth term of \eqref{Btwo8},
%\begin{equation}\label{Btwo10.1}
%\begin{aligned}
%&\int_{A} V(x,0)\mathrm{div}(\mathrm{div}(X(x))X(x))\\
%&\qquad=\int_{\partial A}V(x,0)\mathrm{div}(X(x))\langle X(x),N(x)\rangle
%-\int_{A}\mathrm{div}(X(x))\langle X(x),\nabla_{x}V(x,0)\rangle.
%\end{aligned}
%\end{equation}

Combining the sixth term and one of the fourth terms of \eqref{Btwo8}, then applying the divergence theorem,
%\int_{A}<F,\nabla g>+gdiv(F)=\int_{\partial A}g<F,N>
% so here, F=div(X)X, g=V(x,0)
% so \int_{A}g div(F)=\int_{\partial A}g<F,N>-\int_{A}<F,\nabla g>
\begin{equation}\label{Btwo11}
\begin{aligned}
&\int_{A} V(x,0)\mathrm{div}(\mathrm{div}(X(x))X(x))dx
+\langle\nabla_{x} V(x,0),X(x)\rangle\mathrm{div}(X(x))dx\\
&\qquad=\int_{\partial A}V(x,0)(\mathrm{div}(X(x)))\langle X(x),N(x)\rangle dx
\end{aligned}
\end{equation}

\noindent
\embolden{Step 2. Combining the Terms}.

Now, substituting \eqref{Btwo9}, \eqref{Btwo10} and \eqref{Btwo11} into \eqref{Btwo8},
\begin{equation}\label{Btwo12}
\begin{aligned}
&\frac{d^{2}}{dt^{2}}F(A^{(t)})|_{t=0}
=2\int_{\partial A}\int_{\partial A}G(x,y)\langle X(x),N(x)\rangle\langle X(y),N(y)\rangle dxdy\\
&\quad+\int_{\partial A}\langle \nabla_{x}V(x,0),X(x)\rangle\langle X(x),N(x)\rangle dx
+\int_{\partial A}V(x,0)(\mathrm{div}(X(x)))\langle X(x),N(x)\rangle dx\\
&\quad+\int_{A}V_{tt}(x,t) dx.\\
&\,=2\int_{\partial A}\int_{\partial A}G(x,y)\langle X(x),N(x)\rangle\langle X(y),N(y)\rangle dxdy\\
&\quad+\int_{\partial A}\mathrm{div}(V(x,0)X(x))\langle X(x),N(x)\rangle dx+\int_{A}V_{tt}(x,0) dx.
\end{aligned}
\end{equation}

\noindent
\embolden{Step 3. Computing the final term, $V_{tt}$}.

It therefore remains to compute $\int_{A}V_{tt}(x,t)dx$.  From \eqref{Btwo5.1},
\begin{flalign*}
V_{t}(x,t)&=\frac{d}{dt}\int_{A}G(x,\Psi(y,t))J\Psi(y,t)dy\\
&=\int_{A}\langle\nabla_{z}G(x,\Psi(y,t))(d/dt)\Psi(y,t)\rangle J\Psi(y,t)+G(x,\Psi(y,t))(d/dt)J\Psi(y,t) dy
\end{flalign*}

So, applying the Chain rule, and then the divergence theorem,
%\int g div(F)+\int <F,nabla g>=\int_{\partial }g<F,N>, or
% \int <nabla g,F>=-\int g div(F)+\int_{\partial }g<F,N>
\begin{equation}\label{Btwo13}
\begin{aligned}
&\int_{A}V_{tt}(x,0)dx
=\frac{d}{dt}|_{t=0}\int_{A}V_{t}(x,0)dx\\
&\,=\frac{d}{dt}|_{t=0}\int_{A}\int_{A}\langle\nabla_{z}G(x,\Psi(y,t))(d/dt)\Psi(y,t)\rangle J\Psi(y,t)+G(x,\Psi(y,t))(d/dt)J\Psi(y,t) dydx\\
&\,=\frac{d}{dt}|_{t=0}\int_{A}\int_{A}\langle\nabla_{y}G(x,\Psi(y,t))[D\Psi(y,t)]^{-1}(d/dt)\Psi(y,t)\rangle J\Psi(y,t)\\
&\qquad\qquad+G(x,\Psi(y,t))(d/dt)J\Psi(y,t) dydx\\
&\,=\frac{d}{dt}|_{t=0}-\int_{A}\int_{A}G(x,\Psi(y,t))\mathrm{div}([D\Psi(y,t)]^{-1}(d/dt)\Psi(y,t)\rangle J\Psi(y,t)) dydx\\
&\quad+\int_{A}\int_{\partial A}\langle G(x,\Psi(y,t))[D\Psi(y,t)]^{-1}(d/dt)\Psi(y,t)\rangle J\Psi(y,t), N(y)\rangle dydx\\
&\quad+\int_{A}\int_{A}G(x,\Psi(y,t))(d/dt)J\Psi(y,t) dydx
\end{aligned}
\end{equation}

We now differentiate the three terms in \eqref{Btwo13}.

\begin{equation}\label{Btwo14.0}
D\Psi=I+tD X+O(t^{2}),\quad [D\Psi]^{-1}=I-tD X+O(t^{2}).
\end{equation}

\begin{equation}\label{Btwo14}
(d/dt)|_{t=0}([D\Psi(y,t)]^{-1})=-DX(y).
\end{equation}

%(DX)X=(X_{i}^{(j)})X=\sum_{i}  X_{i}^{(j)}X^{(i)}
\begin{equation}\label{Btwo15}
\begin{aligned}
&(d/dt)|_{t=0}G(x,\Psi(y,t))\mathrm{div}([D\Psi(y,t)]^{-1}(d/dt)\Psi(y,t)\rangle J\Psi(y,t))\\
&\stackrel{\substack{\eqref{Btwo1}\wedge\eqref{Btwo4}\wedge\\ \eqref{Btwo3}\wedge\eqref{Btwo14}\wedge\eqref{Btwo14.0}}}{=}
\langle \nabla_{y}G(x,y),X(y)\rangle\mathrm{div}(X(y))\\
&\qquad\qquad\qquad+G(x,y)\mathrm{div}(-(DX)X+(\sum_{j=1}^{n}X_{x_{j}}^{(i)}X^{(j)})_{i}+X\mathrm{div}(X)) \\
&\qquad=\langle \nabla_{y}G(x,y),X(y)\rangle\mathrm{div}(X(y))+G(x,y)\mathrm{div}(X\mathrm{div}(X))\\
&\qquad=\mathrm{div}_{y}(G(x,y)X(y)\mathrm{div}(X(y))).
\end{aligned}
\end{equation}

As in \eqref{Btwo15},
\begin{equation}\label{Btwo16}
\begin{aligned}
&(d/dt)|_{t=0}\langle G(x,\Psi(y,t))[D\Psi(y,t)]^{-1}(d/dt)\Psi(y,t)\rangle J\Psi(y,t), N(y)\rangle\\
&\qquad= \langle \nabla_{y}G(x,y),X(y)\rangle X(y)+G(x,y)X(y)\mathrm{div}(X(y))\\
&\qquad= X(y)\mathrm{div}_{y}(G(x,y)X(y)).
\end{aligned}
\end{equation}

\begin{equation}\label{Btwo17}
\begin{aligned}
&(d/dt)|_{t=0}G(x,\Psi(y,t))(d/dt)J\Psi(y,t) \\
&\qquad\stackrel{\eqref{Btwo3}\wedge\eqref{Btwo5}}{=}\langle \nabla_{y}G(x,y),X(y)\rangle\mathrm{div}(X(y))+G(x,y)\mathrm{div}(X(y)\mathrm{div}(X(y)))\\
&\qquad=\mathrm{div}_{y}(G(x,y)X(y)\mathrm{div}(X(y))).
\end{aligned}
\end{equation}

Substituting \eqref{Btwo15}, \eqref{Btwo16} and \eqref{Btwo17} into \eqref{Btwo13} and noting that \eqref{Btwo15} and \eqref{Btwo17} cancel,

\begin{equation}\label{Btwo18}
\begin{aligned}
\int_{A}V_{tt}(x,0)dx
&= \int_{A}\int_{\partial A}\mathrm{div}_{y}(G(x,y)X(y))\langle X(y),N(y)\rangle dydx\\
&=\int_{\partial A}\mathrm{div}_{y}\left[\left(\int_{A}G(x,y)dx\right) X(y)\right]\langle X(y), N(y)\rangle dy\\
&\stackrel{\eqref{Btwo5.1}}{=}\int_{\partial A}\mathrm{div}(V(x,0) X(x))\langle X(x),N(x)\rangle dx
\end{aligned}
\end{equation}

\noindent
\embolden{Step 4. Combining all terms together}

Substituting \eqref{Btwo18} into \eqref{Btwo12}, we finally get
\begin{flalign*}
\frac{d^{2}}{dt^{2}}F(A^{(t)})|_{t=0}
&=2\int_{\partial A}\int_{\partial A}G(x,y)\langle X(x),N(x)\rangle\langle X(y),N(y)\rangle dxdy\\
&\qquad+2\int_{\partial A}\mathrm{div}(V(x,0)X(x))\langle X(x),N(x)\rangle dx.
\end{flalign*}

\end{proof}

\begin{lemma}\label{lemmab}
$$\frac{d^{2}}{dt^{2}}|_{t=0}\gamma_{n}(A^{(t)})=\int_{\partial A}(\mathrm{div}(X(x))-\langle X(x),x\rangle)\langle X(x),N(x)\rangle \gamma_{n}(x)dx.$$
\end{lemma}
\begin{proof}
Let $G(x,y)\colonequals\gamma_{n}(x)\gamma_{n}(y)$ for any $x,y\in\R^{n}$.  Then for any $A\subset\R^{n}$, $\int_{A}\int_{A}G(x,y)dxdy=(\gamma_{n}(A))^{2}$.  By Lemma \ref{latelemma3},
$$\frac{d}{dt}|_{t=0}\gamma_{n}(A^{(t)})=\int_{\partial A}\langle X(y),N(y)\rangle \gamma_{n}(y)dy.$$
And \eqref{Btwo5.1} says $V(x,t)\colonequals\int_{A^{(t)}}G(x,y)dy=\gamma_{n}(x)\gamma_{n}(A^{(t)})$ for any $x\in\R^{n},t\in\R$.  So, 
$$\mathrm{div}(V(x,0)X(x))=(\mathrm{div}(X(x))-\langle X(x),x\rangle)\gamma_{n}(x)\gamma_{n}(A),\qquad\forall\,x\in\R^{n}.$$
Then, by the Chain Rule and Lemma \ref{latelemma},
\begin{flalign*}
&\gamma_{n}(A)\frac{d^{2}}{dt^{2}}|_{t=0}\gamma_{n}(A^{(t)})+\left(\int_{\partial A}\langle X(y),N(y)\rangle \gamma_{n}(y)dy\right)^{2}\\
&\gamma_{n}(A)\frac{d^{2}}{dt^{2}}|_{t=0}\gamma_{n}(A^{(t)})+\left(\frac{d}{dt}|_{t=0}\gamma_{n}(A^{(t)})\right)^{2}
=\frac{1}{2}\frac{d^{2}}{dt^{2}}|_{t=0}(\gamma_{n}(A^{(t)}))^{2}\\
&=\left(\int_{\partial A}\langle X(y),N(y)\rangle \gamma_{n}(y)dy\right)^{2}
+\gamma_{n}(A)\int_{\partial A}(\mathrm{div}(X(x))-\langle X(x),x\rangle)\langle X(x),N(x)\rangle \gamma_{n}(x)dx.
\end{flalign*}
The Lemma follows.
%Then, apply \eqref{Btwo18}, and note that, by \eqref{Btwo5.1}, we have $V(x,0)=\int_{A}d\gamma_{n}(y)=\gamma_{n}(A)$.
\end{proof}

%\noindent\textbf{Acknowledgement}.  Thanks to Oded Regev for helpful discussions.  Thanks to Elchanan Mossel for encouraging me to publish these results.

%\bibliographystyle{abbrv}
\bibliographystyle{amsalpha}
\bibliography{symsetv9}

\end{document}